\documentclass[reqno]{amsart}

\numberwithin{equation}{section}

\newtheorem{thm}{Theorem}
\newtheorem{lemma}{Lemma}
\newtheorem{cor}{Corollary}
\newtheorem{prop}{Proposition}
\newtheorem{defi}{Definition}

\newcommand{\rn}{{\mathbb{R}}^n}
\newcommand{\Om}{\Omega}
\newcommand{\begincal}{\begin{eqnarray*}}
\newcommand{\fincal}{\end{eqnarray*}}

\newcommand{\txie}{\tilde{x}_{i,\alpha}}
\newcommand{\nue}{\nu_\alpha}
\newcommand{\rie}{r_{i,\alpha}}
\newcommand{\Rie}{R_{i,\alpha}}

\newcommand{\raa}{R_\alpha}

\def \Va {V_\alpha}
\def \Ua {U_\alpha}

\def \xaun {x_{1, \alpha}}
\def \yaun {y_{1, \alpha}}

\def \Uie {U_{i,\alpha}}
\def \Uje {U_{j,\alpha}}
\def \Uke {U_{\kappa,\alpha}}

\def \Psiie {\Psi_{i,\alpha}}
\def \Psije {\Psi_{j,\alpha}}

\def \Thetaeps {\Theta_\alpha}

\def \Omai {\Omega_{i,\alpha}}
\def \Omaj {\Omega_{j,\alpha}}

\def \etae {\eta_\alpha}

\def \vphiepsgamma {\varphi_{\gamma,\alpha}}

\def \waR {W_{\alpha,R}}

\def \psie {\psi_\alpha}

\def \Leps {L_\alpha}
\def \xkappae {x_{\kappa,\alpha}}
\def \betae {\beta_\alpha}
\def \psikappae {\Psi_{\kappa,\alpha}}
\def \rhoeps {\rho_\alpha}
\def \omegaeps {\Omega_\alpha}
\def \tomegaie {\tilde{\Omega}_{i,\alpha} }
\def \omegaie {\Omega_{i,\alpha} }
\def \nukappa {\nu_{\kappa,\alpha}}

\def \mukappa {\mu_{\kappa,\alpha}}
\def \mukappaun {\mu_{\kappa+1,\alpha}}
\def \Rkappaeps {R_{\kappa,\alpha}}
\def \Rkappauneps {R_{\kappa+1,\alpha}}

\def \petitrkappauneps {r_{\kappa+1,\alpha}}
\def \epstozero {\alpha\to +\infty}
\def \epspositif {\alpha>0}
\def \bue {\bar{u}_\alpha}
\def \Ua {U_\alpha}
\def \tUe {\tilde{U}_\alpha}
\def \va {V_\alpha}
\def \rr {\mathbb{R}}
\def \rn {\mathbb{R}^n}
\def \rnm {\mathbb{R}^n_-}
\def \nn {\mathbb{N}}
\def \omegabar {\overline{\Omega}}
\def \crit {2^\star}
\def \tg {\tilde{g}}

\def \tpip {\tilde{\pi}_\varphi}
\def \pip {\pi_\varphi}

\def \ua {u_\alpha}
\def \ra {r_\alpha}

\def \ma {\mu_\alpha}
\def \xa {x_\alpha}
\def \ya {y_\alpha}
\def \za {z_\alpha}

\def \tv {\tilde{v}}

\def \tue {\tilde{u}_\alpha}
\def \ma {\mu_{\alpha}}
\def \tve {\tilde{v}_\alpha}
\def \twe {\tilde{w}_\alpha}

\def \tga{\tilde{g}_\alpha}

\def \bua {\bar{u}_\alpha}
\def \tua {\tilde{u}_\alpha}

\def \tva {\tilde{v}_\alpha}

\def \va {v_\alpha}
\def \wa {w_\alpha}

\def \ea {\epsilon_\alpha}

\def \eps {\epsilon}

\def \tuai {\tilde{u}_{i,\alpha}}

\def \xai {x_{i,\alpha}}
\def \mai {\mu_{i,\alpha}}

\def \tpi {\tilde{\pi}}

\def \vai {v_{i,\alpha}}
\def \tvai {\tilde{v}_{i,\alpha}}

\def \Uai {U_{i,\alpha}}
\def \tUai {\tilde{U}_{i,\alpha}}

\def \Vai {V_{i,\alpha}}

\def \xaj {x_{j,\alpha}}
\def \maj {\mu_{j,\alpha}}

\def \Vaj {V_{j,\alpha}}

\def \xak {x_{k,\alpha}}
\def \mak {\mu_{k,\alpha}}

\def \sai {s_{i,\alpha}}
\def \saj {s_{j,\alpha}}

\begin{document}
\title[Lin-Ni's problem]
{The Lin-Ni's problem for mean convex domains}
\author{Olivier Druet}
\address{\noindent Olivier Druet. Ecole normale sup\'erieure de Lyon, D\'epartement de
Math\'ematiques - UMPA, 46 all\'ee d'Italie, 69364 Lyon cedex 07,
France}
\email{Olivier.Druet@umpa.ens-lyon.fr}

\author{Fr\'ed\'eric Robert}
\address{\noindent Fr\'ed\'eric Robert. Institut \'Elie Cartan, Universit\'e Henri Poincar\'e Nancy 1,
B.P. 239, F-54506 Vandoeuvre-l\`es-Nancy Cedex, France}
\email{frobert@iecn.u-nancy.fr}

\author[J.Wei]{Juncheng Wei}
\address{\noindent Juncheng Wei. Department of Mathematics,
Chinese University of Hong Kong, Shatin, Hong Kong}
\email{wei@math.cuhk.edu.hk}

\date{February 26th 2011. Mathematics Subjet Classification 2010: 35J20, 35J60}

\begin{abstract}
We prove some refined asymptotic estimates for postive blowing up solutions to $\Delta u+\epsilon u=n(n-2)u^{\frac{n+2}{n-2}}$ on $\Omega$, $\partial_\nu u=0$ on $\partial\Omega$; $\Omega$ being a smooth bounded domain of $\rn$, $n\geq 3$. In particular, we show that concentration can occur only on boundary points with nonpositive mean curvature when $n=3$ or $n\geq 7$. As a direct consequence, we prove the validity of the Lin-Ni's conjecture in dimension $n=3$ and $n\geq 7$ for mean convex domains and with bounded energy. Recent examples by Wang-Wei-Yan \cite{wwy2} show that the bound on the energy is a necessary condition.
\end{abstract}

\maketitle

\centerline{\it Fr\'ed\'eric Robert dedicates this work to Cl\'emence Climaque}

\tableofcontents

\section{Introduction}
Let $\Omega$ be a smooth bounded domain of $\rn$, $n\geq 2$. In \cite{linnitakagi}, Lin, Ni and Takagi got interest in solutions $u\in C^2(\omegabar)$ to the elliptic problem
$$\left\{\begin{array}{ll}
\Delta u+\epsilon u=n(n-2) u^{q-1} & \hbox{ in }\Omega\\
u>0& \hbox{ in }\Omega\\
\partial_\nu u=0 & \hbox{ on }\partial\Omega
\end{array}\right.\eqno{(E_q)}$$
where $\epsilon>0$ is a parameter and $q>2$. Here and in the sequel, $\Delta:=-\hbox{div}(\nabla)$ is the Laplace operator with minus-sign convention. This problem has its origins in the analysis of the Gierer-Meinhardt model in mathematical biology: this model is a system of nonlinear evolution equations of parabolic type, and the stationary problem with infinite diffusion constant splits into two equations like $(E_q)$. We refer to the surveys \cite{nisurvey,weisurvey} for the justifications of the model and its simplification. (It can also be considered as stationary system to the chemotaxis  Keller-Segel model with logarithmic sensitivity function. See \cite{linnitakagi}.)

\medskip\noindent Problem $(E_q)$ enjoys a variational structure, since its solutions are critical points of the functional
$$u\mapsto \frac{1}{2}\int_\Omega|\nabla u|^2\, dx+\frac{\epsilon}{2}\int_\Omega u^2\, dx-\frac{1}{q}\int_\Omega |u|^q\, dx,$$
a functional that is defined for all $u\in H_1^2(\Omega)\cap L^q(\Omega)$, where $H_1^2(\Omega)$ is the standard Sobolev space of $L^2-$functions with derivatives also in $L^2$ endowed with the norm $\Vert\cdot\Vert_2+\Vert\nabla\cdot\Vert_2$. In particular, it follows from Sobolev's embedding theorem that $H_1^2(\Omega)\hookrightarrow L^q(\Omega)$ continuously in case $2<q\leq \crit$ where $\crit:=\frac{2n}{n-2}$ (we assume here that $n\geq 3$): therefore the functional above is defined on $H_1^2(\Omega)$ when $2<q\leq \crit$. Moreover, the Sobolev embedding above is compact in case $q<\crit$.

\medskip\noindent The system $(E_q)$ enjoys at least a solution, namely the constant solution $u\equiv\left( \frac{\epsilon}{n(n-2)}\right)^{\frac{1}{q-2}}$. In a series of seminal works, Lin-Ni-Takagi \cite{linnitakagi} and Ni-Takagi \cite{nitakagicpam,nitakagiduke} got interest in the potential existence of nonconstant solutions to $(E_q)$. In particular, it is showed in \cite{nitakagicpam,nitakagiduke} that for $\epsilon$ large and $q <\crit$, least energy solution concentrates at boundary points of maximum mean curvature. It turns out  that when $\epsilon $ is large  there exist an abundance of solutions to $(E_q)$. This case has become a hot research topic in the last fifteen years.  In the subcritical case $q<\crit$, we refer to Gui-Wei \cite{GW}, Gui-Wei-Winter \cite{GWW}, Dancer-Yan \cite{DY},  Lin-Ni-Wei \cite{LNW} and the references therein. In the critical case $q=\crit$, we refer to Adimurthi-Mancini-Yadava \cite{amy}, Adimurthi-Pacella-Yadava \cite{APY}, del Pino-Musso-Pistoia \cite{DMP},  Ghoussoub-Gui \cite{GG}, Gui-Lin \cite{guilin}, Ni-Pan-Takagi \cite{NPT}, Rey-Wei \cite{RW}, Wei-Yan \cite{WY} and the references therein.  In the present article, we restrict our attention to that case when {\bf $\epsilon>0$ is small}. In case $2<q<\crit$, variational techniques and the compactness of the embedding imply that for small positive $\epsilon$, the constant solution is the sole solution to $(E_q)$. This uniqueness result incited Lin and Ni to conjecture the extension of this result to the critical case $q=\crit$:

\medskip\noindent{\bf Question (Lin-Ni \cite{linni}):} {\it Is the constant solution the only solution to $(E_{\crit})$ when $\epsilon>0$ is small?}

\medskip\noindent The mathematical difficulty of this question comes from the conformal invariance of $(E_{\crit})$ and its associated unstability. Indeed, for $\mu>0$ and $x_0\in\rn$, define
\begin{equation}\label{def:U}
U_{x_0,\mu}(x):=\left(\frac{\mu}{\mu^2+|x-x_0|^2}\right)^{\frac{n-2}{2}}\hbox{ for all }x\in\rn.
\end{equation}
The scalar curvature equation for the pulled back of the spherical metric via the stereographic projection (or direct computations) yields $\Delta U_{x_0,\mu}=n(n-2)U_{x_0,\mu}^{\crit-1}$ in $\rn$. Therefore, there is an abundance of solutions to $\Delta u=u^{\crit-1}$, some of them being peaks blowing-up to infinity since $\lim_{\mu\to 0}U_{x_0,\mu}(x_0)=+\infty$: in this sense, the equation is unstable since it enjoys many solutions that are far from each other. There are no such solutions in the subcritical case $q<\crit$ (see Caffarelli-Gidas-Spruck \cite{cgs}). This conformal dynamic transfers on the Lin-Ni's problem and it follows from the famous Struwe decomposition \cite{struwe} that families of solutions $(u_\epsilon)_{\epsilon>0}$ to $(E_{\crit})$ with bounded energy may develop some peaks like \eqref{def:U} when $\epsilon\to 0$: more precisely, there exists $N\in\nn$ such that for any $i\in\{1,...,N\}$, there exists sequences $(x_{i,\epsilon})_{\epsilon}\in\rn$, $(\mu_{i,\epsilon})_i\in \rr_{>0}$ such that $\lim_{\epsilon\to 0}\mu_{i,\epsilon}=0$ and, up to the extraction of a subfamily,
\begin{equation}\label{dec:struwe}
u_\epsilon=\sum_{i=1}^N U_{x_{i,\epsilon}, \mu_{i,\epsilon}}+R_\epsilon
\end{equation}
where $\lim_{\epsilon\to 0}R_\epsilon=0$ in $H_1^2(\Omega)$. This decomposition is refered to as the integral decomposition. When there is at least one peak, then there are nonconstant solutions. Conversely, in case there is no peak, elliptic estimates and simple integrations by parts (see Section \ref{sec:prelim}) yield the sole constant solution for small $\epsilon$.

\medskip\noindent In the radial case, that is when $\Omega$ is a ball and when $u$ is radially symmetrical, Adimurthi-Yadava solved the problem in \cite{adimurthiyadava91,adimurthiyadava97}: when $n=3$ or $n\geq 7$, the answer to Lin-Ni's question is positive, and it is negative for $n\in \{4,5,6\}$.  In the asymmetric case, the complete answer is not known yet, but there are a few results. When $n=3$, it was proved independently by Zhu \cite{zhu} and Wei-Xu \cite{weixu} that the answer to Lin-Ni's question is positive when $\Omega$ is convex. When $n=5$, Rey-Wei \cite{ReyWei} constructed solutions to $(E_{\crit})$ as a sum of peaks like \eqref{def:U} for $\epsilon\to 0$. In the present paper, we concentrate on the localization of the peaks in the general case.

\medskip\noindent Let $(\ea)_{\alpha\in\nn}\in (0,1]$ be a sequence such that
$$\lim_{\alpha\to +\infty}\ea=0.$$
We consider a sequence $(\ua)_{\alpha\in\nn}\in C^2(\omegabar)$ such that
\begin{equation}\label{eq:ua}
\left\{\begin{array}{ll}
\Delta\ua+\ea\ua=n(n-2)\ua^{\crit-1} & \hbox{ in }\Omega\\
\ua>0& \hbox{ in }\Omega\\
\partial_\nu\ua=0 & \hbox{ on }\partial\Omega
\end{array}\right.\end{equation}
We assume that there existe $\Lambda>0$ such that
\begin{equation}\label{bnd:nrj}
\int_\Omega \ua^{\crit}\, dx\leq \Lambda
\end{equation}
for all $\alpha\in\nn$.
\begin{defi}\label{def:sing} We say that $x\in \overline{\Omega}$ is a non-singular point of $(\ua)$ if there exists $\delta>0$ and $C>0$ such that
$$\Vert \ua\Vert_{L^\infty(B_\delta(x)\cap\Omega)}\leq C$$
for all $\alpha\in\nn$. We say that $x\in \overline{\Omega}$ is a singular point if it is not a non-singular point.
\end{defi}
The singular points are exactly the points where the peaks are located. In the sequel, $H(x)$ denotes the mean curvature at $x\in\partial\Omega$ of the canonically oriented boundary $\partial\Omega$. With our sign convention, the mean curvature of the canonically  oriented boundary of the unit ball is positive. We prove the following theorem:
\begin{thm}\label{th:sing:bndy} Let $(\ua)_{\alpha\in\nn}\in C^2(\omegabar)$ and $\epsilon>0$ such that \eqref{eq:ua} and \eqref{bnd:nrj} hold. Let ${\mathcal S}$ denote the (possibly empty) set of singular points for $(\ua)$. Assume that $n=3$ or $n\geq 7$: then ${\mathcal S}$ is finite and
$${\mathcal S}\subset\{x\in \partial\Omega/\, H(x)\leq 0\}.$$
\end{thm}
As a consequence, we get the following:
\begin{thm}\label{th:conj}[Lin-Ni's conjecture for mean convex domains] Let $\Omega$ be a smooth bounded domain of $\rn$, $n=3$ or $n\geq 7$. Assume that $H(x)>0$ for all $x\in\partial\Omega$. Then for all $\epsilon>0$, there existe $\epsilon_0(\Omega,\Lambda)>0$ such that for all $\epsilon\in (0,\epsilon_0(\Omega,\Lambda))$ and for any $u\in C^2(\omegabar)$, we have that
$$\left\{\begin{array}{ll}
\Delta u+\epsilon u=n(n-2)u^{\crit-1} & \hbox{ in }\Omega\\
u>0& \hbox{ in }\Omega\\
\partial_\nu u=0 & \hbox{ on }\partial\Omega\\
\int_\Omega u^{\crit}\, dx\leq\Lambda
\end{array}\right\}\;\Rightarrow\; u\equiv \left(\frac{\epsilon}{n(n-2)}\right)^{\frac{n-2}{4}}.$$
\end{thm}

The method we use to prove Theorem \ref{th:sing:bndy} relies on a sharp control of the solutions to \eqref{eq:ua} in the spirit of Druet-Hebey-Robert \cite{dhr}, our first result being that (see Proposition \ref{prop:claimE} and \eqref{est:pointwise} in Section \ref{sec:asymp:c1})
\begin{equation}\label{ineq:co}
\ua\leq C\left(\bua+\sum_{i=1}^N U_{x_{i,\alpha}, \mu_{i,\alpha}}\right)
\end{equation}
where $\bua$ is the average of $\ua$ on $\Omega$ and the peaks are as in Struwe's decomposition \eqref{dec:struwe}. In particular, we pass from an integral description to a pointwise description. As in Druet \cite{druetjdg} (see also Ghoussoub-Robert \cite{GhRoIMRP} and Druet-Hebey \cite{druethebey}), this pointwise description allows us to determine exactly where two peaks may interact, and to describe precisely the behavior of $\ua$ there. The localization of the singular points then follows from a succession of Pohozaev identities.

\medskip\noindent These results appeal some remarks. In dimension $n=3$, our result must be compared to Zhu's result: in \cite{zhu}, no bound on the energy is assumed, but the convexity is required; in our result, we require the bound on the energy, but a weak convexity only is needed. The assumption on the energy \eqref{bnd:nrj} may seem technical for who is familiar with the Yamabe equation: indeed, in general, see Druet \cite{druetimrn}, Li-Zhu \cite{lizhu}, Schoen \cite{schoen} and Khuri-Marques-Schoen \cite{kms}, any solution to the Yamabe equation automatically satisfies a bound on the energy like \eqref{bnd:nrj}. For the Lin-Ni's problem, this is not the case: recently, it was proved that solutions to $(E_{\crit})$ may accumulate with infinite energy when the mean curvature is negative somewhere (see Wang-Wei-Yan \cite{wwy}) or when $\Omega$ is a ball (see Wang-Wei-Yan \cite{wwy2}), a domain with positive mean curvature: therefore, the answer to Lin-Ni's question is negative if one does not impose the bound \eqref{bnd:nrj}.

\medskip\noindent The influence of curvature is reminiscent in the asymptotic analysis of equations like \eqref{eq:ua}. For instance, in Druet \cite{druetjdg,druetimrn} and in Li-Zhu \cite{lizhu}, it is proved that for Yamabe-type equations, the peaks are located where the potential of the equation touches the scalar curvature; we refer to Hebey-Robert-Wen \cite{hrw} and Hebey-Robert \cite{hr} for the corresponding localization for fourth-order problems. In Ghoussoub-Robert \cite{GhRoGAFA, GhRoIMRP}, that is for a singular Dirichlet-type problem, the peaks are located where the mean curvature is nonnegative: in Theorem \ref{th:sing:bndy} above, that is for a Neumann problem, we conversely prove that the peaks are located at points of nonpositive mean curvature. For Neumann-type equations like \eqref{eq:ua}, the role of the mean curvature has been enlighted, among others, by Adimurthi-Mancini-Yadava \cite{amy}, Lin-Wang-Wei \cite{lww}, Gui-Lin \cite{guilin}, Ni-Pan-Takagi \cite{NPT}, Rey-Wei \cite{RW} and Wei-Yan \cite{WY}.

\medskip\noindent The present paper is devoted to the asymptotic analysis of solutions $(\ua)_\alpha$ of \eqref{eq:ua} satisfying \eqref{bnd:nrj} when $n\geq 3$. In Sections 2 to 7, we prove the pointwise control \eqref{ineq:co}. Section 8 is devoted to the convergence of the $(\ua)_\alpha$'s at the scale where  peaks interact. In Sections 9 and 10, we prove an asymptotic relation mixing the heights of the peaks, the distance between peaks and the mean curvature. Finally, we prove Theorems \ref{th:sing:bndy} and \ref{th:conj} in Section 11.

\medskip\noindent{\bf Notations:} in the sequel, we define $\rnm:=\{(x_1,x')\in\rn/\, x_1<0\}$ and we assimilate $\partial\rnm=\{(0,x')/\, x'\in\rr^{n-1}\}$ to $\rr^{n-1}$. Given two sequences $(a_\alpha)_\alpha\in\rr$ and $(b_\alpha)_\alpha\in\rr$, we say that $a_\alpha\asymp b_\alpha$ when $\alpha\to +\infty$ if $a_\alpha=O(b_\alpha)$ and $b_\alpha=O(a_\alpha)$ when $\alpha\to +\infty$. For $U$ an open subset of $\rn$, $k\in \nn$, $k\geq 1$, and $p\geq 1$, we define $H_k^p(U)$ as the completion of $C^\infty(\bar{U})$ for the norm $\sum_{i=1}^k\Vert\nabla^i\Vert_p$.

\medskip\noindent{\it Acknowledgements:} This work was initiated and partly carried out during the visits of F.Robert in Hong-Kong. He expresses his thanks J.Wei for the invitations and his gratitude for his friendly support in April 2010.
F.Robert was partially supported by the ANR grant ANR-08-BLAN-0335-01 and by a regional grant from Universit\'e Nancy 1 and R\'egion Lorraine. The research of J.Wei is partially supported by RGC of HK, HK/France Joint Grant, and ``Focused Research Scheme'' of CUHK.

\section{$L^\infty-$bounded solutions}\label{sec:prelim}
Let $\Omega\subset \rn$ be a smooth domain (see Definition \ref{def:smoothdomain} of Section \ref{sec:reflection} below), $n\geq 3$. We consider a sequence $\left(\ua\right)_{\alpha\in\nn}$ of positive solutions of
\begin{equation}\label{eq1}
\left\{\begin{array}{ll}
{\displaystyle \Delta \ua +\ea \ua = n(n-2)\ua^{2^\star-1}}&{\displaystyle \hbox{ in }\Omega} \\
{\displaystyle \ua>0}&{\displaystyle\hbox{ in }\Omega}\\
{\displaystyle \partial_\nu \ua = 0}&{\displaystyle \hbox{ on }\partial \Omega}
\end{array}\right.
\end{equation}
We assume in the following that
\begin{equation}\label{eq2}
\int_\Om \ua^{2^\star}\, dx \le \Lambda
\end{equation}
for some $\Lambda>0$.  We claim that
\begin{equation}\label{eq4}
\ua\rightharpoonup 0 \hbox{ weakly in }H_{1}^2\left(\Omega\right)\hbox{ as }\epstozero.
\end{equation}
We prove the claim. Indeed, after integrating \eqref{eq1} on $\Omega$, it follows from Jensen's inequality that

$$\left(\frac{1}{|\Omega|}\int_{\Om} \ua\, dx\right)^{2^\star-1}\leq \frac{1}{|\Omega|}\int_{\Om} \ua^{2^\star-1}\, dx=\frac{\ea \int_\Om \ua \, dx}{n(n-2)|\Omega|}$$
for all $\alpha\in\nn$. Then, we get that
\begin{equation}\label{upp:bua}
\bua\leq\left(\frac{\ea}{n(n-2)}\right)^{\frac{n-2}{4}}
\end{equation}
for all $\alpha\in\nn$, where, given $\bua:=\frac{1}{|\Omega|}\int_\Omega \ua\, dx$ denote the average of $\ua$ on $\Omega$. Multiplying \eqref{eq1} by $\ua$ and integrating on $\Omega$, we get that $(\ua)_\alpha$ is bounded in $H_1^2(\Omega)$. Therefore, up to a subsequence, $(\ua)_\alpha$ converges weakly. The convergence \eqref{eq4} then follows from \eqref{upp:bua}. This proves the claim.

\medskip\noindent We prove in this section the following:
\begin{prop}\label{prop:claimA} Assume that the sequence $\left(\ua\right)_\alpha$ is uniformly bounded in $L^\infty\left(\Om\right)$. Then there exists $\alpha_0>0$ such that $\ua\equiv \left(\frac{\ea}{n(n-2)}\right)^{\frac{n-2}{4}}$ for all $\alpha\geq \alpha_0$.
\end{prop}

\medskip\noindent{\it Proof of Proposition \ref{prop:claimA}:} Assume that there exists $M>0$ such that $\ua\le M$ in $\Om$ for all $\epspositif$.  By standard elliptic theory (see Theorem 9.11 in \cite{gt} together with Theorem \ref{app:th:existence} of Section \ref{sec:Green}), we deduce then thanks to (\ref{eq4}) that $\ua\to 0$ in $L^\infty\left(\Om\right)$. Multiplying equation (\ref{eq1}) by $\ua - \bue$  ($\bua$ is the average of $\ua$ defined above) and integrating by parts, we then get that
\begincal
&&\int_{\Om} \left\vert \nabla \ua\right\vert^2\, dx + \ea \int_{\Om} \left(\ua-\bue\right)^2\, dx\\ &&\quad =
n(n-2) \int_\Om \ua^{2^\star-1}\left(\ua-\bue\right)\, dx \\
&&\quad = n(n-2) \int_\Om \left(\ua^{2^\star-1}-\bue^{2^\star-1}\right)\left(\ua-\bue\right)\, dx\\
&&\quad =O\left( \left(\left\Vert \ua\right\Vert_\infty^{2^\star-2} + \bue^{2^\star-2}\right)\int_\Om \left(\ua-\bue\right)^2\, dx\right)\\
&&\quad = o\left(\int_\Om \left(\ua-\bue\right)^2\, dx\right) = o\left(\int_{\Om} \left\vert \nabla \ua\right\vert^2\, dx\right)
\fincal
when $\alpha\to +\infty$ thanks to Poincar\'e's inequality. This yields $\int_{\Om} \left\vert \nabla \ua\right\vert^2\, dx=0$ for $\alpha$ large and thus $\ua$ is a constant for $\alpha>\alpha_0$ for some $\alpha_0>0$. The constant is easily seen to be $\left(\frac{\ea}{n(n-2)}\right)^{\frac{n-2}{4}}$ thanks to equation (\ref{eq1}). This ends the proof of Proposition \ref{prop:claimA}.\hfill $\Box$

\medskip\noindent For the rest of the article, we assume that
\begin{equation}\label{eq3}
\lim_{\alpha\to +\infty}\left\Vert \ua\right\Vert_\infty =+\infty.
\end{equation}
Under this assumption, the sequence $\left(\ua\right)$ will develop some concentration points. In sections \ref{sec:exhaus} to \ref{sec:asymp:c1}, we provide sharp pointwise estimates on $\ua$ and thus describe precisely how the sequence $\left(\ua\right)$ behaves in $C^1\left(\bar{\Om}\right)$. In section \ref{sec:cv:sing} to \ref{sec:bnd}, we get precise informations on the patterns of concentration points which can appear. This permits to conclude the proof of the main theorems in section \ref{sec:proof:th}.

\section{Smooth domains and extensions of solutions to elliptic equations}\label{sec:reflection}
We first define smooth domains:
\begin{defi}\label{def:smoothdomain} Let $\Omega$ be an open subset of $\rn$, $n\geq 2$. We say that $\Omega$ is a smooth domain if for all $x\in \partial\Omega$, there exists $\delta_x>0$, there exists $U_x$ an open  neighborhood of $x$ in $\rn$, there exists $\varphi: B_{\delta_x}(0)\to U_x$ such that

$$\label{ppty:phi}\begin{array}{ll}
(i) & \varphi\hbox{ is a }C^\infty-\hbox{diffeomorphism}\\
(ii)& \varphi(0)=x\\
(iii)& \varphi(B_{\delta_x}(0)\cap \{x_1<0\})=\varphi(B_{\delta_x}(0))\cap\Omega\\
(iv) &\varphi(B_{\delta_x}(0)\cap \{x_1=0\})=\varphi(B_{\delta_x}(0))\cap\partial\Omega
\end{array}$$
\end{defi}
The outward normal vector is then defined as follows:
\begin{defi}\label{def:vecnormal}
Let $\Omega$ be a smooth domain of $\rn$. For any $x\in\partial\Omega$, there exists a unique $\nu(x)\in\rn$ such that $\nu(x)\in (T_{x}\partial\Omega)^\bot$, $\Vert \nu(x)\Vert=1$ and $(\partial_1\varphi(0),\nu(x))>0$ for $\varphi$ as in Definition \ref{def:smoothdomain}. This definition is independent of the choice of such a chart $\varphi$ and the map $x\mapsto\nu(x)$ is in $C^\infty(\partial\Omega,\rn)$.
\end{defi}

\medskip\noindent  Let $\Omega$ be a smooth bounded domain of $\rn$ as above. We consider the following problem:
\begin{equation}\label{NB}
\left\{\begin{array}{ll}
\Delta u=f &\hbox{ in }\Omega\\
\partial_\nu u=0&\hbox{ in }\partial\Omega
\end{array}\right.
\end{equation}
where $u\in C^2(\omegabar)$ and $f\in C^0(\omegabar)$. Note that the solution $u$ is defined up to the addition of a constant and that it is necessary that $\int_\Omega f\, dx=0$ (this is a simple integration by parts). It is useful to extend solutions to \eqref{NB} to a neighborhood of each point of $\partial\Omega$. For this, a variational formulation of \eqref{NB} is required: multiplying \eqref{NB} by $\psi\in C^\infty(\omegabar)$ and integrating by parts leads us to the following definition:
\begin{defi} We say that $u\in H_1^1(\Omega)$ is a weak solution to \eqref{NB} with $f\in L^1(\Omega)$ if
$$\int_\Omega (\nabla u,\nabla \psi)\, dx=\int_\Omega f\psi\, dx\hbox{ for all }\psi\in C^\infty(\omegabar).$$
\end{defi}
\noindent In case $u\in C^2(\omegabar)$, as easily checked, $u$ is a weak solution to \eqref{NB} iff it is a classical solution to \eqref{NB}.

\smallskip\noindent We let $\xi$ be the standard Euclidean metric on $\rn$ and we set
\begin{equation*}
\left\{\begin{array}{llll}
\tpi: & \rn & \to & \rn\\
     & (x_1,x')& \mapsto & (-|x_1|,x')
\end{array}\right.
\end{equation*}
Given a chart $\varphi$ as in Definition \ref{def:smoothdomain}, we define
$$\tpip:=\varphi\circ\tpi\circ\varphi^{-1}.$$
Up to taking $U_{x_0}$ smaller, the map $\tpip$ fixes $U_{x_0}\cap \omegabar$ and ranges in $\omegabar$. We prove the following useful extension lemma:
\begin{lemma}\label{app:lem:ext} Let $x_0\in\partial\Omega$. There exist $\delta_{x_0}>0$, $U_{x_0}$ and a chart $\varphi$ as in Definition \ref{def:smoothdomain} such that the metric $\tg:=\tpip^\star\xi=(\varphi\circ\tpi\circ\varphi^{-1})^\star\xi$ is in $C^{0,1}(U_{x_0})$ (that is Lipschitz continuous), $\tg_{|\Omega}=\xi$, the Christoffel symbols of the metric $\tg$ are in $L^\infty(U_{x_0})$ and $d\varphi_0$ is an orthogonal transformation. Let $u\in H_1^1(\Omega\cap U_{x_0})$ and $f\in L^1(\Omega\cap U_{x_0})$ be functions such that
\begin{equation}\label{app:eq:u:dist}
\int_\Omega (\nabla u,\nabla \psi)\, dx=\int_\Omega f\psi\, dx\hbox{ for all }\psi\in C^\infty_{c}(\omegabar\cap U_{x_0}).
\end{equation}
For all $v: \Omega\cap U_{x_0}\to \rr$, we define
$$\tilde{v}:=v\circ  \tpip\hbox{ in }U_{x_0}.$$
Then, we have that $\tilde{u}\in H_1^1(U_{x_0})$, $\tilde{u}_{|\Omega}=u$, $f\in L^1(U_{x_0})$ and
$$\Delta_{\tg}\tilde{u}=\tilde{f}\hbox{ in the distribution sense,}$$
where $\Delta_{\tg}:=-\hbox{div}_{\tg}(\nabla)$.
\end{lemma}
Here, by "distribution sense", we mean that
$$\int_{U_{x_0}} (\nabla \tilde{u},\nabla \psi)_{\tg}\, dv_{\tg}=\int_{U_{x_0}} \tilde{f}\psi\, dv_{\tg}\hbox{ for all }\psi\in C^\infty_{c}(U_{x_0}),$$
where $dv_{\tg}$ is the Riemannian element of volume associated to $\tg$ and $(\cdot,\cdot)_{\tg}$ is the scalar product on $1-$forms.

\medskip\noindent{\it Proof of Lemma \ref{app:lem:ext}:} Given a chart $\hat{\varphi}$ at $x_0$ defined on $B_{\tilde{\delta}_{x_0}}(0)$ as in Definition \ref{def:smoothdomain}, we define the map
$$\left\{\begin{array}{llll}
\varphi: & B_{\tilde{\delta}_{x_0}}(0) & \to & \rn\\
     & (x_1,x')& \mapsto & x_1\nu(\hat{\varphi}(0,x'))+\hat{\varphi}(0,x')
\end{array}\right.$$
The inverse function theorem yields the existence of $\delta_{x_0}>0$ and $U_{x_0}\subset\rn$ open such that $\varphi: B_{\delta_{x_0}}(0)\to U_{x_0}$ is a smooth diffeomorphism being a chart at $x_0$ as in Definition \ref{def:smoothdomain}. Moreover, the pull-back metric satisfies the following properties:
$$(\varphi^\star\xi)_{11}=1,\; (\varphi^\star\xi)_{1i}=0\;\forall i\neq 1.$$
In particular, up to a linear transformation on the $\{x_1=0\}$ hyperplane, we can assume that $d\varphi_0$ is an orthogonal transformation. It is easily checked that $((\varphi\circ\tpi)^\star\xi)_{ij}=(\varphi^\star\xi)_{ij}\circ\tpi$ outside $\{x_1=0\}$ for all $i,j$, and then we prologate $(\varphi\circ\tpi)^\star\xi$ as a Lipschitz continuous function in $U_{x_0}$, and so is $\tg:=(\varphi\circ\tpi\circ\varphi^{-1})^\star\xi$. In addition, as easily checked, if $\tilde{\Gamma}_{ij}^k$'s denote the Christoffel symbols for the metric $\tg$, we have that $\tilde{\Gamma}_{ij}^k\in L^\infty$. Therefore, the coefficients of $\Delta_{\tg}$ are in $L^\infty$ and the principal part is Lipschitz continuous.\par

\medskip\noindent We fix $\psi\in C^\infty_c(U_{x_0})$. For convenience, in the sequel, we define $\pi:=\tilde{\pi}_{|\rn_+}$, that is
\begin{equation*}
\left\{\begin{array}{llll}
\pi: & \rn_+ & \to & \rn_-\\
     & (x_1,x')& \mapsto & (-x_1,x').
\end{array}\right.
\end{equation*}
Clearly, $\pi$ is a smooth diffeomorphism. As for $\tpip$, we define
$$\pip:=\varphi\circ\pi\circ\varphi^{-1}$$
that maps (locally) $\omegabar^c$ to $\omegabar$. With changes of variable, we get that
$$\int_{U_{x_0}} (\nabla \tilde{u},\nabla \psi)_{\tg}\, dv_{\tg}=\int_{\Omega\cap U_{x_0}}(\nabla u, \nabla (\psi+\psi\circ \pip^{-1}\circ\varphi^{-1}))\, dx$$
and
$$\int_{U_{x_0}} \tilde{f}\psi\, dv_{\tg}=\int_{\Omega\cap U_{x_0}}f(\psi+\psi \circ\pip^{-1})\, dx.$$
It then follows from \eqref{app:eq:u:dist} that $\Delta_{\tg} \tilde{u}=\tilde{f}$ in $U_{x_0}$ in the distribution sense. This ends the proof of Lemma \ref{app:lem:ext}.\hfill$\Box$

\medskip\noindent In the particular case of smooth solutions, we have the following lemma:
\begin{lemma}\label{lem:ext} Let $x_0\in\partial\Omega$. There exist $\delta_{x_0}>0$, $U_{x_0}$ and a chart $\varphi$ as in Definition \ref{def:smoothdomain} such that the metric $\tg:=(\varphi\circ\tpi\circ\varphi^{-1})^\star\xi$ is in $C^{0,1}(U_{x_0})$ (that is Lipschitz continuous), $\tg_{|\Omega}=\xi$, the Christoffel symbols of the metric $\tg$ are in $L^\infty(U_{x_0})$ and $d\varphi_0$ is an orthogonal transformation. We let $u\in C^2(\omegabar\cap U_{x_0})$ and all $f\in C^1_{loc}(\rr)$ be such that
$$\left\{\begin{array}{ll}
\Delta u=f(u) & \hbox{ in }\omegabar\cap U_{x_0}\\
\partial_\nu u=0& \hbox{ in }\partial\Omega\cap U_{x_0}
\end{array}\right.$$
and we define
$$\tilde{u}:=u\circ \varphi\circ \tpi\circ\varphi^{-1}\hbox{ in }U_{x_0}.$$
Then, in addition to the regularity of $\tg$, we have that
$$\tilde{u}\in C^{2}(U_{x_0}),\; \tilde{u}_{|\Omega}=u\hbox{ and }\Delta_{\tg}\tilde{u}=f(\tilde{u})\hbox{ for all } x\in U_{x_0},$$
where $\Delta_{\tg}:=-\hbox{div}_{\tg}(\nabla)$.
\end{lemma}

\section{Exhaustion of the concentration points}\label{sec:exhaus}

We prove in this section the following~:

\begin{prop}\label{prop:claimB} Let $(\ua)_{\alpha\in\nn}\in C^2(\omegabar)$ and $\Lambda>0$ such that \eqref{eq:ua} and \eqref{bnd:nrj} hold for all $\alpha\in\nn$. Then there exists $N\in {\mathbb N}^\star$, $N$ sequences $\left(\xai\right)_{i=1,\dots,N}$  of points in $\overline{\Omega}$ and $N$ sequences $\mu_{1,\alpha} \geq \mu_{2,\alpha}\geq \dots \geq \mu_{N,\alpha}$ of positive real numbers such that, after passing to a subsequence, the following assertions hold~:

\smallskip (i) For any $1\le i\le N$, $\xai \to x_i$ as $\epstozero$ for some $x_i\in \bar{\Om}$ and $\mai\to 0$ as $\epstozero$. Moreover, either $\frac{d\left(\xai,\partial\Om\right)}{\mai}\to +\infty$ as $\epstozero$ or $\xai\in \partial \Om$.

\smallskip (ii) For any $1\le i<j\le N$,
$$\frac{\left\vert \xai-\xaj\right\vert^2}{\mai\maj}+\frac{\mai}{\maj}+\frac{\maj}{\mai}\to +\infty\hbox{ as }\epstozero\hskip.1cm.$$

\smallskip (iii) For any $1\le i\le N$, we define
$$\tuai:= \mai^{\frac{n-2}{2}} \ua\left(\xai+\mai \, .\,\right)\hbox{ if }\lim_{\alpha\to +\infty}\frac{d\left(\xai,\partial\Om\right)}{\mai}=+\infty,$$
and
$$\tuai:= \mai^{\frac{n-2}{2}} \tue\circ\varphi\left(\varphi^{-1}(\xai)+\mai \, .\,\right)\hbox{ if }\xai\in\partial\Omega\hbox{ for all }\alpha\in\nn$$
where $\tue$ is the extension of $\ua$ around $x_0:=\lim_{\alpha\to +\infty}\xai$ and $\varphi$ are as in Lemma \ref{lem:ext}. Then
\begin{equation}\label{cv:ua:ma:2}
\lim_{\alpha\to +\infty}\Vert\tuai-U_0\Vert_{C^1\left(K\cap \bar{\Om}_{i,\alpha}\right)}=0
\end{equation}
for all compact subsets $K\subset\subset\rn\setminus {\mathcal S}_i$ if $\xai\not\in\partial\Omega$ and $K\subset\subset\rn\setminus \left({\mathcal S}_i\cup\pi^{-1}({\mathcal S}_i)\right)$ if $\xai\in\partial\Omega$ where the function $U_0$ is given by
$$U_0(x):= \left(1+\vert x\vert^2\right)^{1-\frac{n}{2}}$$
and ${\mathcal S}_i$ is defined by
$${\mathcal S}_i := \left\{\lim_{\epstozero} \frac{\xaj-\xai}{\mai}\hbox{, }i< j\le N\right\}\hskip.1cm.$$
In the definition of ${\mathcal S}_i$, we allow the limit to be $+\infty$ (and in fact, we discard these points).

\smallskip (iv) We have that
$$\raa^{\frac{n-2}{2}} \left\vert \ua - \sum_{i=1}^N \Uie\right\vert \to 0\hbox{ in }L^\infty\left(\bar{\Om}\right)\hbox{ as }\epstozero$$
where
$$\raa(x):= \min_{1\le i\le N} \sqrt{\left\vert \xai-x\right\vert^2 +\mai^2}$$
and
$$\Uie(x):= \mai^{1-\frac{n}{2}} U_0\left(\frac{x-\xai}{\mai}\right)\hskip.1cm.$$
\end{prop}

\medskip\noindent{\it Proof of Proposition \ref{prop:claimB}:} For $N\geq 1$, we say that property ${\mathcal P}_N$ holds if there exist $N$ sequences $\left(\xai\right)_{i=1,\dots,N}$  of points in $\overline{\Omega}$ and $N$ sequences $\mu_{1,\alpha} \geq \mu_{2,\alpha}\geq \dots \geq \mu_{N,\alpha}$ of positive real numbers such that, after passing to a subsequence, assertions (i)-(ii)-(iii) of the claim hold for these sequences. We divide the proof of Proposition \ref{prop:claimB} in three steps.

\medskip\noindent{\bf Step  \ref{prop:claimB}.1:} We claim that there exists $N_{max}\geq 1$ such that $\left({\mathcal P}_N\right)$ can not hold for $N\geq N_{max}$.

\medskip\noindent {\it Proof of Step  \ref{prop:claimB}.1:} Let $N\geq 1$ be such that $\left({\mathcal P}_N\right)$ holds. Let $\left(\xai\right)_{i=1,\dots,N}$ be $N$ sequences of points in $M$ and $\mu_{1,\alpha} \geq \mu_{2,\alpha}\geq \dots \geq \mu_{N,\alpha}$ be $N$ sequences of positive real numbers such that the assertions (i)-(ii)-(iii) of Proposition \ref{prop:claimB} hold after passing to a subsequence. Let $R>0$ and set
$$\Omai\left(R\right) = B_{R\mai}(\xai) \setminus \bigcup_{i< j\le N} B_{\frac{1}{R}\mai}(\xaj)\hskip.1cm.$$
It easily follows from (ii) that
$$\Omai\left(R\right) \cap \Omaj\left(R\right)=\emptyset$$
for $\alpha$ large enough. Thus we can write that
$$\int_{\Om} \ua^{2^\star}\, dx \geq \sum_{i=1}^N\int_{\Omai\left(R\right)\cap \Om}\ua^{2^\star}\, dx$$
for $\alpha$ large enough. It follows then from (iii) that
$$\int_{\Om} \ua^{2^\star}\, dx \geq \frac{N}{2}\int_{\rn} U_0^{2^\star}\, dx -\eta(R)+o(1)$$
where $\eta(R)\to 0$ as $R\to +\infty$. Letting $R\to +\infty$ and thanks to (\ref{eq2}), we then get that
$$N\le \frac{2\Lambda}{\int_{\rn} U_0^{2^\star}\, dx}\hskip.1cm.$$
This ends the proof of Step  \ref{prop:claimB}.1. \hfill$\Box$

\medskip\noindent {\bf Step \ref{prop:claimB}.2:}  We claim that ${\mathcal P}_1$ holds.

\medskip\noindent{\it Proof of Step \ref{prop:claimB}.2.}  We let $\xa\in \bar{\Om}$ be such that
\begin{equation}\label{eqB.2.1}
\ua\left(\xa\right) = \max_{\bar{\Om}}\ua
\end{equation}
and we set
\begin{equation}\label{eqB.2.2}
\ua\left(\xa\right)=\ma^{1-\frac{n}{2}}\hskip.1cm.
\end{equation}
Thanks to (\ref{eq3}), we know that $\ma\to 0$ as $\epstozero$. We set
\begin{equation}\label{eqB.2.3}
\va\left(x\right) :=\ma^{\frac{n}{2}-1} \ua\left(\xa+\ma x\right)
\end{equation}
for $x\in \Omega_\alpha = \left\{x\in \rn\hbox{ s.t. } \xa+\ma x\in \Om\right\}$. It is clear that
$$\Delta \va +\ea\ma^2 \va = n(n-2)\va^{2^\star-1}\hbox{ in }\Omega_\alpha$$
with $\partial_\nu \va=0$ on $\partial \Omega_\alpha$ and
$$0\le \va \le \va(0)=1\hbox{ in }\Omega_\alpha\hskip.1cm.$$

\medskip\noindent{\bf Step \ref{prop:claimB}.2.1}: we assume that
\begin{equation}\label{case:b21}
\lim_{\epstozero}\frac{d\left(\xa,\partial \Om\right)}{\ma}=+\infty.
\end{equation}
It follows from standard elliptic theory (see \cite{gt}) that, after passing to a subsequence,
$$\va \to v\hbox{ in }C^2_{loc}\left(\rn\right)\hbox{ as }\epstozero$$
where $v\in C^2(\rn)$ is such that
$$\Delta v= n(n-2)v^{2^\star-1}$$
and
$$0\le v\le v(0)=1\hskip.1cm.$$
By the classification result of Caffarelli-Gidas-Spruck \cite{cgs}, we then get that $v=U_0$. This proves ${\mathcal P}_1$ in case \eqref{case:b21}. This ends Step \ref{prop:claimB}.2.1.

\medskip\noindent{\bf Step \ref{prop:claimB}.2.2}: we assume that there exists $\rho\geq 0$ such that
\begin{equation}\label{case:b22}
\lim_{\epstozero}\frac{d\left(\xa,\partial \Om\right)}{\ma}=\rho.
\end{equation}
We let $x_0:=\lim_{\epstozero}\xa$. We then have $x_0\in \partial\Omega$ and we choose $\varphi$ and $\delta_{x_0}>0$, $U_{x_0}$ as in Lemma \ref{lem:ext}. Let $\delta\in (0,\delta_{x_0})$. Denoting by $\tue\in C^{2}(U_{x_0})$ the local extension of $\ua$ on $U_{x_0}$ with respect to $\varphi$, we then have that
\begin{equation}\label{eq:tua}
\Delta_{\tg} \tue+\ea\tue=\tue^{\crit-1}\hbox{ in }U_{x_0}.
\end{equation}
Since $d\varphi_0$ is an orthogonal transformation, we have that
\begin{equation}\label{dist:phi}
d(\varphi(x),\partial\Omega)=(1+o(1))|x_1|
\end{equation}
for all $x\in B_{\delta_0}(0)\cap \overline{\rnm}$, where $\lim_{x\to 0}o(1)=0$ uniformly locally. We let $(x_{\alpha,1},\xa')\in \{x_1\leq 0\}\times\rr^{n-1}$ be such that $\xa:=\varphi(x_{\alpha,1},\xa')$ for all $\alpha\in\nn$. It follows from \eqref{case:b22} and \eqref{dist:phi} that
\begin{equation}\label{lim:rho}
\lim_{\epstozero}\frac{|x_{\alpha,1}|}{\ma}=\rho.
\end{equation}
We define
\begin{equation}\nonumber
\tve(x):=\ma^{\frac{n-2}{2}}\tue(\varphi((0,\xa')+\ma x))\hbox{ for all }x\in B_{\delta/\ma}(0).
\end{equation}
It follows from \eqref{eq:tua} that
\begin{equation}\label{eq:tva}
\Delta_{\tga} \tve+\ea\ma^2\tve=n(n-2)\tve^{\crit-1}\hbox{ in }B_{\delta/\ma}(0),
\end{equation}
where $\tga(x)=(\varphi^\star \tg)((0,\xa')+\ma x)=((\varphi^{-1}\circ \tilde{\pi})^\star\xi)((0,\xa')+\ma x)$. Since $0<\tve\leq \tve(\rhoeps,0)=1$ and \eqref{lim:rho} holds, it follows from standard elliptic theory (see Theorem 9.11 in \cite{gt}) that there exists $V\in C^1(\rn)$ such that
\begin{equation}\label{lim:tva}
\lim_{\epstozero}\tve=V\hbox{ in }C^{1}_{loc}(\rn),
\end{equation}
where $0\leq V\leq V(\rho,0)=1$. Passing to the limit $\epstozero$ in \eqref{eq:tva} and using that $d\varphi_0$ is an orthogonal transformation, we get that $\Delta V=n(n-2)V^{\crit-1}$ weakly in $\rn$. Since $V\in C^1(\rn)$, one gets that $V\in C^2(\rn)$ and it follows from Caffarelli-Gidas-Spruck \cite{cgs} that
$$V(x)=\left(\frac{1}{1+|x-(\rho,0)|^2}\right)^{\frac{n-2}{2}}$$
for all $x\in\rn$. The Neumann boundary condition $\partial_\nu\ua=0$ rewrites $\partial_1\tve=0$ on $\partial\rnm$. Passing to the limit, one gets that $\partial_1 V=0$ on $\partial\rnm$, and therefore $\rho=0$ and $V\equiv U_0$. In particular, we have that
$$\lim_{\epstozero}\frac{x_{\alpha,1}}{\ma}=0.$$
Taking $\tilde{x}_{\alpha}:=\varphi(0, \xa')$, we can then perform the above analysis of Step \ref{prop:claimB}.2.2 with $\tilde{x}_{\alpha}\in\partial\Omega$ instead of $\xa$. This proves ${\mathcal P}_1$ in case \eqref{case:b22}. This ends Step \ref{prop:claimB}.2.2.

\medskip\noindent Steps \ref{prop:claimB}.2.1 and Step \ref{prop:claimB}.2.2 prove that ${\mathcal P}_1$ holds. Step \ref{prop:claimB}.2 is proved. \hfill $\Box$

\medskip\noindent{\bf Remark:} For ${\mathcal P}_1$, we can be a little more precise and prove the following claim:
\begin{equation}\label{loc:xa:bnd}
\xa\in\partial\Omega\hbox{ for }\alpha\in\nn\hbox{ large.}
\end{equation}
We prove the claim by contradiction and assume that $\xa\not\in\Omega$ for a subsequence. Define $\rhoeps:=\frac{x_{\alpha,1}}{\ma}$. Then $\rhoeps<0$ for $\alpha$ large. Since $(\rhoeps,0)$ is a maximum point of $\tve$, we have that $\partial_1\tve(\rhoeps,0)=0$. Since $\partial_1\tve(0)=0$ (Neumann boundary condition), it then follows from Rolle's Theorem that there exists $\tau_\alpha\in (0,1)$ such that $\partial_{11}\tve(\tau_\alpha\rhoeps,0)=0$. Letting $\epstozero$, we get that $\partial_{11}U_0(0)=0$: a contradiction. This proves the claim.

\medskip\noindent{\bf Step \ref{prop:claimB}.3:} Assume that ${\mathcal P}_N$ holds for some $N\geq 1$. Let $\left(\xai\right)_{i=1,\dots,N}$  be $N$ sequences of points in $\overline{\Omega}$ and $\mu_{1,\alpha} \geq \mu_{2,\alpha}\geq \dots \geq \mu_{N,\alpha}$ be $N$ sequences of positive real numbers such that assertions (i)-(ii)-(iii) of the claim hold. We claim that if assertion (iv) of Proposition \ref{prop:claimB} does not hold for this sequence of points, then ${\mathcal P}_{N+1}$ holds.

\medskip\noindent{\it Proof of Step \ref{prop:claimB}.3:} We assume that (iv) does not hold for these sequences. In other words assume that there exists $\epsilon_0>0$ such that
\begin{equation}\label{eqB.3.1}
\max_{\bar{\Om}} \left(\raa^{\frac{n-2}{2}} \left\vert \ua - \sum_{i=1}^N \Uie\right\vert \right)\geq \eps_0
\end{equation}
for all $\alpha\in\nn$ where
$$\raa(x)^2:= \min_{1\le i\le N} \bigl(\left\vert \xai-x\right\vert^2 +\mai^2\bigr)$$
and
$$\Uie(x): = \mai^{1-\frac{n}{2}} U_0\left(\frac{x-\xai}{\mai}\right)\hskip.1cm.$$
We let $\ya\in \bar{\Om}$ be such that
\begin{equation}\label{eqB.3.2}
\max_{\bar{\Om}} \left(\raa^{\frac{n-2}{2}} \left\vert \ua - \sum_{i=1}^N \Uie\right\vert \right)
= \raa\left(\ya\right)^{\frac{n-2}{2}} \left\vert \ua\left(\ya\right) - \sum_{i=1}^N \Uie\left(\ya\right)\right\vert
\end{equation}
and we set
\begin{equation}\label{eqB.3.3}
\ua\left(\ya\right)= \nue^{1-\frac{n}{2}}\hskip.1cm.
\end{equation}

\medskip\noindent{\bf Step \ref{prop:claimB}.3.1:} We claim that
\begin{equation}\label{eqB.3.4}
\raa\left(\ya\right)^{\frac{n-2}{2}} \Uie\left(\ya\right)\to 0\hbox{ as }\epstozero\hbox{ for all }1\le i\le N\hskip.1cm.
\end{equation}
Indeed, assume on the contrary that there exists $1\le i\le N$ such that
\begin{equation}\label{eqB.3.5}
\raa\left(\ya\right)^{\frac{n-2}{2}} \Uie\left(\ya\right)\geq \eta_0
\end{equation}
for some $\eta_0>0$. This means that
\begin{equation}\label{eqB.3.6}
\frac{\raa\left(\ya\right)}{\mai}\geq \eta_0^{\frac{2}{n-2}}  \left(1+\frac{\vert \ya-\xai\vert^2}{\mai^2}\right)\hskip.1cm.
\end{equation}
Since $\raa\left(\ya\right)^2\le \left\vert \ya-\xai\right\vert^2 + \mai^2$, we get in particular that, up to a subsequence,
\begin{equation}\label{eqB.3.6bis}
\frac{\vert \ya-\xai\vert}{\mai}\to R\hbox{ as }\epstozero
\end{equation}
for some $R>0$. Coming back to (\ref{eqB.3.6}), we can also write that
\begin{equation}\label{eqB.3.6ter}
\frac{\left\vert \xaj-\ya\right\vert^2}{\mai^2}+\frac{\maj^2}{\mai^2} \geq \eta_0^{\frac{4}{n-2}} \left(1+R^2\right)^2 +o(1)
\end{equation}
for all $1\le j\le N$. These two equations permit to prove thanks to (ii) of Proposition \ref{prop:claimB} (which holds by assumption) that $\frac{\left\vert \xaj-\ya\right\vert}{\mai}\geq \eta_0^{\frac{2}{n-2}} \left(1+R^2\right) +o(1)$ for all $i<j\le N$. Thus
$$\lim_{\epstozero} \frac{\ya-\xai}{\mai}\not\in {\mathcal S}_i$$
and we use (iii) of Proposition \ref{prop:claimB} to get that
$$\mai^{\frac{n-2}{2}} \left\vert \ua\left(\ya\right)- \Uie\left(\ya\right)\right\vert \to 0$$
as $\epstozero$. Since $\raa\left(\ya\right) = O\left(\mai\right)$, we thus get that
$$\raa\left(\ya\right)^{\frac{n-2}{2}}\left\vert \ua\left(\ya\right)- \Uie\left(\ya\right)\right\vert \to 0$$
as $\epstozero$. Let $1\le j\le N$, $j\neq i$. We write now that
$$\raa\left(\ya\right) \Uje\left(\ya\right)^{\frac{2}{n-2}} = O\left(\frac{\mai}{\maj} \left(1+\frac{\left\vert \ya-\xaj\right\vert^2}{\maj^2}\right)^{-1}\right)=o(1)$$
thanks to (\ref{eqB.3.6bis}), (\ref{eqB.3.6ter}) and assertion (ii) of Proposition \ref{prop:claimB}. Thus we arrive to
$$\raa\left(\ya\right)^{\frac{n-2}{2}} \left\vert \ua\left(\ya\right) - \sum_{i=1}^N \Uie\left(\ya\right)\right\vert\to 0$$
as $\epstozero$ which contradicts (\ref{eqB.3.5}) and thus proves (\ref{eqB.3.4}). This ends Step \ref{prop:claimB}.3.1.

\medskip\noindent Note that, coming back to (\ref{eqB.3.1}), (\ref{eqB.3.2}), (\ref{eqB.3.3}) with (\ref{eqB.3.4}), we get that
\begin{equation}\label{eqB.3.7}
\frac{\raa\left(\ya\right)}{\nue} \geq \eps_0^{\frac{2}{n-2}}+o(1)\hskip.1cm.
\end{equation}

\medskip\noindent{\bf Step \ref{prop:claimB}.3.2:} We claim that
\begin{equation}\label{eqB.3.8}
\nue\to 0\hbox{ as }\epstozero\hskip.1cm.
\end{equation}

\smallskip\noindent We prove the claim. If $\raa\left(\ya\right)\to 0$ as $\epstozero$, then \eqref{eqB.3.8} follows from (\ref{eqB.3.7}). Assume now that $\raa\left(\ya\right)\geq 2\delta_0$ for some $\delta_0>0$. Using \eqref{eqB.3.2} and (\ref{eqB.3.4}), we get that
$$\ua\le 2^{\frac{n}{2}}\ua\left(\ya\right)+o(1)$$
in $B_{\delta_0}(\ya)\cap \bar{\Om}$ for $\alpha$ large enough. If $\ua(\ya)\to +\infty$ when $\alpha\to +\infty$, then \eqref{eqB.3.8} holds. If $\ua\left(\ya\right) =O(1)$, we then get by standard elliptic theory (see \cite{gt} and Lemma \ref{lem:ext}) and thanks to (\ref{eq4}) that $\ua\left(\ya\right)\to 0$ as $\epstozero$, which contradicts (\ref{eqB.3.7}) since $\Om$ is a bounded domain. This proves (\ref{eqB.3.8}) and ends Step \ref{prop:claimB}.3.2.

\medskip\noindent Note also that (\ref{eqB.3.4}) directly implies that
\begin{equation}\label{eqB.3.8bis}
\frac{\left\vert \xai-\ya\right\vert^2}{\mai\nue}+\frac{\mai}{\nue}\to +\infty\hbox{ as }\epstozero
\end{equation}
for all $1\le i\le N$. We set now
\begin{equation}\label{eqB.3.9}
\wa(x): = \nue^{\frac{n-2}{2}}\ua\left(\ya +\nue x\right)
\end{equation}
in $\omegaeps: = \left\{x\in \rn\hbox{ s.t. }\ya+\nue x\in \Om\right\}$. We then have that
\begin{equation}\label{eq:we:1}
\Delta \wa+ \ea\nue^2\wa = n(n-2)\wa^{2^\star-1}
\end{equation}
in $\omegaeps$ and $\partial_\nu \wa = 0$ on $\partial\omegaeps$. We define
$${\mathcal S}:=\left\{\lim_{\epstozero} \frac{\xai-\ya}{\nue}\, , \, 1\le i\le N\hbox{ s.t. }\left\vert \xai-\ya\right\vert=O\left(\nue\right)\hbox{ and }\mai=o\left(\nue\right)\right\}\hskip.1cm.$$
Let us fix $K\subset\subset \rn\setminus{\mathcal S}$ a compact set. We note that, thanks to (\ref{eqB.3.4}) and (\ref{eqB.3.2}),
\begin{equation}\label{eqB.3.10}
\left(\frac{\raa\left(\ya+\nue x\right)}{\raa\left(\ya\right)}\right)^{\frac{n-2}{2}}\wa\left(x\right)\le 1+o(1)
\end{equation}
for all $x\in K\cap \omegaeps$, where $\lim_{\epstozero}\sup_{K\cap \omegaeps}o(1)=0$. Let $\za\in B_R(0)\cap \overline{\Omega}_\alpha \setminus \bigcup_{x\in {\mathcal S}}B_{R^{-1}}(x)$ for some $R>0$ fixed.

\medskip\noindent{\bf Step \ref{prop:claimB}.3.3:} We claim that
\begin{equation}\label{eqB.3.11}
\wa\left(\za\right)=O\left(1\right)\hskip.1cm.
\end{equation}

\smallskip\noindent We prove the claim. It is clear from (\ref{eqB.3.10}) if $\frac{\raa\left(\ya+\nue \za\right)}{\raa\left(\ya\right)}\not\to 0$ as $\epstozero$. Assume now that
\begin{equation}\label{eq:3.115}
\frac{\raa\left(\ya+\nue \za\right)}{\raa\left(\ya\right)} \to 0\hbox{ as }\epstozero\hskip.1cm.
\end{equation}
Up to a subsequence, we let $1\le i\le N$ be such that
$$\raa\left(\ya+\nue \za\right)^2 = \left\vert \xai-\ya-\nue \za\right\vert^2+\mai^2\hskip.1cm.$$
We then write thanks to \eqref{eq:3.115} that
$$ \left\vert \xai-\ya-\nue \za\right\vert^2+\mai^2 = o\left(\left\vert \xai-\ya\right\vert^2\right) + o\left(\mai^2\right)$$
which implies that $\left\vert \xai-\ya\right\vert = O\left(\nue\right)$ and that $\mai=o\left(\nue\right)$. This leads to
$$\left\vert \frac{\xai-\ya}{\nue}- \za\right\vert^2 \to 0\hbox{ as }\epstozero\hskip.1cm,$$
which is absurd since, thanks to the definition of ${\mathcal S}$ and to the fact that $d\left(\za,{\mathcal S}\right)\geq \frac{1}{R}$,
$$\left\vert \frac{\xai-\ya}{\nue}-\za\right\vert \geq \frac{1}{2R}$$
for $\alpha$ large. Thus (\ref{eqB.3.11}) is proved. This ends Step \ref{prop:claimB}.3.3.

\medskip\noindent Thanks to (\ref{eqB.3.7}), we easily get that $0\in \omegabar_\alpha\setminus {\mathcal S}$.

\medskip\noindent{\bf Step \ref{prop:claimB}.3.4:} Assume first that
$$\lim_{\epstozero}\frac{d\left(\ya,\partial\Om\right)}{\nue}=+\infty.$$
It follows from Step \ref{prop:claimB}.3.3 that $(\wa)_\alpha$ is bounded in $L^\infty$ on all compact subsets of $\rn\setminus{\mathcal S}$. Then, by standard elliptic theory (see \cite{gt}), it follows from \eqref{eq:we:1} that, after passing to a subsequence,
$$\wa\to w_0 \hbox{ in }C^1_{loc}\left(\rn\setminus {S}\right)$$
where $w_0$ satisfies
$$\Delta_\xi w_0 = n(n-2)w_0^{2^\star-1}$$
in $\rn\setminus  {S}$ and $w_0(0)=1$. Noting that, since $\left(\ua\right)$ is uniformly bounded in $H_1^2\left(\Omega\right)$, we have that $w_0\in H_{1,loc}^2\left(\rn\right)$, we easily get that $w_0$ is in fact a smooth solution of the above equation and that, by Caffarelli-Gidas-Spruck \cite{cgs},
$w_0 (x)= \lambda^{\frac{n-2}{2}}U_0\left(\lambda x + x_0\right)$ for some $\lambda>0$ and some $x_0\in \rn$. If we set
$$x_{N+1,\alpha} := \ya -\frac{\nue}{\lambda}x_0$$
and
$$\mu_{N+1,\alpha} = \lambda^{-1} \nue\hskip.1cm,$$
it is easily checked that, up to reorder the concentration points such that the sequence of weights is non-increasing, assertions (i)-(ii)-(iii) of Proposition \ref{prop:claimB} hold for the $N+1$ sequences $\left(\xai,\mai\right)_{i=1,\dots,N+1}$. Here one must use in particular (\ref{eqB.3.8bis}) to get (ii). This ends Step \ref{prop:claimB}.3.4.

\medskip\noindent{\bf Step \ref{prop:claimB}.3.5:} Assume now that
$$\lim_{\epstozero}\frac{d\left(\ya,\partial\Om\right)}{\nue}=\rho\geq 0.$$
One proceeds similarly, using the extension $\tue$ of $\ua$ as in Lemma \ref{lem:ext} as was done for Step \ref{prop:claimB}.2.2. More precisely, let $y_0:=\lim_{\epstozero}\ya\in \partial\Omega$. We choose $\varphi$ and $\delta_{y_0}>0$, $U_{y_0}$ as in Lemma \ref{lem:ext}. Let $\delta\in (0,\delta_{y_0})$. Denoting by $\tue\in C^{2}(U_{y_0})$ the local extension of $\ua$ on $U_{y_0}$, we then have that
\begin{equation}\label{eq:tua:bis}
\Delta_{\tg} \tue+\ea\tue=\tue^{\crit-1}\hbox{ in }U_{y_0}.
\end{equation}
As in Step \ref{prop:claimB}.2.2, we let $\ya:=\varphi(y_{\alpha,1},\ya')$ and we have
$$\lim_{\epstozero}\frac{y_{\alpha,1}}{\nue}=-\rho\leq 0.$$
We define
\begin{equation}\nonumber
\tve(x):=\nue^{\frac{n-2}{2}}\tue(\varphi((0,\ya')+\nue x))\hbox{ for all }x\in B_{\delta/\nue}(0).
\end{equation}
It follows from \eqref{eq:tua:bis} that
\begin{equation}\label{eq:tva:bis}
\Delta_{\tga} \twe+\ea\nue^2\twe=n(n-2)\twe^{\crit-1}\hbox{ in }B_{\delta/\nue}(0),
\end{equation}
where $\tga(x)=(\varphi^\star \tg)((0,\ya')+\nue x)=((\varphi^{-1}\circ \tilde{\pi})^\star\xi)((0,\ya')+\nue x)$. We define
$$\tilde{I}:=\left\{i\in\{1,...,N\}\hbox{ s.t. }\left\vert \xai-\ya\right\vert=O\left(\nue\right)\hbox{ and }\mai=o\left(\nue\right)\right\}$$
and
$$\tilde{\mathcal S}=\left\{\lim_{\epstozero} \frac{\varphi^{-1}(\xai)-(0,\ya')}{\nue}/\, i\in\tilde{I}\right\}\hskip.1cm.$$
We let $K\subset \rn\setminus \left(\tilde{\mathcal S}\cup \pi^{-1}(\tilde{\mathcal S})\right)$ a compact set. Here, \eqref{eqB.3.11} rewrites $0<\twe(x)\leq C(K)$ for all $x\in K\cap\rnm$. The symmetry of $\twe$ yields
$$0<\twe(x)\leq C(K)\hbox{ for all }x\in K\hbox{ and all }\epspositif.$$
We are then in position to use elliptic theory to get the convergence of $\twe$ in $C^1_{loc}(\rn\setminus \left(\tilde{\mathcal S}\cup \pi^{-1}(\tilde{\mathcal S})\right))$, and the proof goes as in Step \ref{prop:claimB}.3.4. This ends Step \ref{prop:claimB}.3.5.

\medskip Proposition \ref{prop:claimB} follows from Step \ref{prop:claimB}.1 to Step \ref{prop:claimB}.3. Indeed, Step \ref{prop:claimB}.2 tells us that ${\mathcal P}_1$ holds. Then we construct our sequences of points and weights thanks to Step \ref{prop:claimB}.3. Thanks to Step \ref{prop:claimB}.1, we know that the process has to stop. When it stops, (i)-(iv) of the claim holds for these points and weights. This proves Proposition \ref{prop:claimB}.\hfill$\Box$

\section{A first upper-estimate}\label{sec:first:est}

We consider in the following the concentration points $\left(\xai,\mai\right)_{1\le i\le N}$ given by Proposition \ref{prop:claimB}. We recall that they are ordered in such a way that
$$\mu_{1,\alpha}\geq \dots\geq \mu_{N,\alpha}$$
and we shall denote in the following $\ma=\mu_{1,\alpha}$. Let us fix some notations and make some remarks before going on. We let
\begin{equation}\label{eq2.1}
{\mathcal S} := \left\{\lim_{\epstozero}\xai\, ,\,\, 1\le i\le N\right\}
\end{equation}
where the limits do exist, up to a subsequence. For $\delta>0$ small enough, we let
\begin{equation}\label{eq2.2}
\etae\left(\delta\right) = \sup_{\bar{\Om}\cap \left\{d\left(x,{\mathcal S}\right)\geq 2\delta\right\}} \ua\hskip.1cm.
\end{equation}
Thanks to Proposition \ref{prop:claimB} (iv) and to standard elliptic theory (see Theorem 9.11 of \cite{gt}), we get that
\begin{equation}\label{eq2.3}
\etae\left(\delta\right)\to 0\hbox{ as }\epstozero\hbox{ for all }\delta>0\hskip.1cm.
\end{equation}
Note that, as a consequence of (iii) of Proposition \ref{prop:claimB}, there exists $C>0$ such that
$$C\ma^{\frac{n-2}{2}} \le \int_{\Om} \ua^{2^\star-1}\, dx$$
while
$$\int_{\Om} \ua^{2^\star-1}\, dx = \ea \left\vert \Om \right\vert \bue$$
thanks to equation (\ref{eq1}). This proves that
\begin{equation}\label{eq2.4bis}
\ma^{\frac{n-2}{2}}=O\left(\ea \bue\right)=o\left(\bue\right)
\end{equation}
when $\alpha\to +\infty$. At last, we fix $R_0>0$ such that
\begin{equation}\label{eq2.5}
\hbox{for any }1\le i\le N\, , \,\, \vert x\vert \le \frac{R_0}{2}\hbox{ for all }x\in {\mathcal S}_i
\end{equation}
where ${\mathcal S}_i$ is as in Proposition \ref{prop:claimB}, (iii). And we let
\begin{equation}\label{eq2.6}
\ra\left(x\right):=\min_{i=1,\dots, N} \left\vert \xai-x\right\vert\hskip.1cm.
\end{equation}

\medskip\noindent  We prove in this section the following~:

\begin{prop}\label{prop:claimC} There exists $C_1>0$ and some sequence $\betae \to 0$ as $\epstozero$ such that
\begin{equation}\label{est:claimC}
\left\vert \ua(x)-\bue\right\vert \le C_1\ma^{\frac{n-2}{2}} \raa\left(x\right)^{2-n} +\betae \bue
\end{equation}
for all $x\in \bar{\Om}$ and all $\epspositif$.
\end{prop}

\medskip\noindent{\it Proof of Proposition \ref{prop:claimC}:} We divide the proof in two main steps. We start by proving the following~:

\medskip\noindent{\bf Step \ref{prop:claimC}.1:} We claim that for any $0<\gamma<\frac{1}{2}$, there exists $R_\gamma>0$, $\delta_\gamma>0$ and $C_\gamma>0$ such that
$$\ua(x)\le C_\gamma \left(\ma^{\frac{n-2}{2}\left(1-2\gamma\right)} \ra(x)^{\left(2-n\right)\left(1-\gamma\right)}+\etae\left(\delta_\gamma\right) \ra(x)^{\left(2-n\right)\gamma}\right)$$
for all $\epspositif$ and all $x\in \bar{\Om}\setminus \bigcup_{i=1}^N B_{R_\gamma\mai}(\xai)$.

\medskip\noindent {\it Proof of Step \ref{prop:claimC}.1.} We divide the proof in two parts, depending whether we work in the interior of $\Omega$ or near its boundary. Let $0<\gamma<\frac{1}{2}$. We define
$$\Phi_{\gamma}(x,y):=|x-y|^{(2-n)(1-\gamma)}\hbox{ for all }x,y\in\rnm,\; x\neq y.$$

\medskip\noindent{\bf Step \ref{prop:claimC}.1.1:} We fix $x_0\in \Omega$ and we let $\delta_0>0$ such that $B_{\delta_0}(x_0)\subset\subset \Omega$. We claim that there exists $R_\gamma>0$ such that
\begin{equation}\label{ineq:C1:1}
\ua(x)\le C_\gamma \left(\ma^{\frac{n-2}{2}\left(1-2\gamma\right)} \ra(x)^{\left(2-n\right)\left(1-\gamma\right)}+\etae\left(\delta_\gamma\right) \ra(x)^{\left(2-n\right)\gamma}\right)
\end{equation}
for all $\epspositif$ and all $x\in \overline{B_{\delta_0}(x_0)}\setminus \bigcup_{i=1}^N B_{R_\gamma\mai}(\xai)$.

\medskip\noindent We prove the claim. We let
\begin{equation}\label{eqC.1.1}
\vphiepsgamma(x)= \ma^{\frac{n-2}{2}\left(1-2\gamma\right)}\sum_{i=1}^N \Phi_\gamma\left(\xai,x\right) + \etae\left(\delta\right) \sum_{i=1}^N \Phi_{1-\gamma}\left(\xai,x\right)
\end{equation}
where $\Phi_\gamma$ and $\Phi_{1-\gamma}$ are as above and $\delta>0$ will be chosen later on. We let $\xa\in \overline{B_{\delta_0}(x_0)}\setminus \bigcup_{i=1}^N B_{R\mai}(\xai)$ be such that
\begin{equation}\label{def:sup}
\sup_{\overline{B_{\delta_0}(x_0)}\setminus \bigcup_{i=1}^N B_{R\mai}(\xai)} \frac{\ua}{\vphiepsgamma} = \frac{\ua\left(\xa\right)}{\vphiepsgamma\left(\xa\right)}\hskip.1cm.
\end{equation}
In particular, $\xa\in \Omega$.

\medskip\noindent We claim that, up to choose $\delta>0$ small enough and $R>0$ large enough, we have that
\begin{equation}
\xa\in \partial \left(\bigcup_{i=1}^N B_{R\mai}(\xai)\right)\hbox{ or }\ra\left(\xa\right)\geq \delta \label{eq:sup}
\end{equation}
for $\epspositif$ small. We prove \eqref{eq:sup} by contradiction. We assume on the contrary that
\begin{equation}\label{eqC.1.3}
\xa\not\in  \partial \left(\bigcup_{i=1}^NB_{R\mai}(\xai)\right)\hbox{ and }\ra\left(\xa\right)< \delta
\end{equation}
for all $\epspositif$. Since $\xa\in\Omega$, we write then thanks to (\ref{eqC.1.3}) and the second order characterization of the supremum \eqref{def:sup} that
$$\frac{\Delta \ua\left(\xa\right)}{\ua\left(\xa\right)} \geq \frac{\Delta \vphiepsgamma\left(\xa\right)}{\vphiepsgamma\left(\xa\right)}\hskip.1cm.$$
Thanks to (\ref{eq1}), we have that
$$\frac{\Delta \ua\left(\xa\right)}{\ua\left(\xa\right)} \le n(n-2)\ua\left(\xa\right)^{2^\star-2}$$
which leads to
\begin{equation}\label{eqC.1.4}
\frac{\Delta \vphiepsgamma\left(\xa\right)}{\vphiepsgamma\left(\xa\right)}\le n(n-2)\ua\left(\xa\right)^{2^\star-2}\hskip.1cm.
\end{equation}
Direct computations yield the existence of $D_\gamma>0$ such that
\begin{equation}\label{estAB}
\begin{array}{l}
\hbox{ (A) }D_\gamma^{-1}\le \left\vert x-y\right\vert^{\left(n-2\right)\left(1-\gamma\right)}\Phi_\gamma\left(x,y\right) \le D_\gamma \hbox{ for all }x,\,y\in \rn,\; x\neq y.\\
\\
\hbox{ (B) }\frac{\Delta_y \Phi_\gamma\left(x,y\right)}{\Phi_\gamma\left(x,y\right)}\geq \frac{1}{D_\gamma} \left\vert x-y\right\vert^{-2} - D_\gamma\hbox{ for all }x,\, y\in \Om,\; x\neq y.
\end{array}\end{equation}
Let us write now thanks \eqref{estAB} that
\begincal
\Delta\vphiepsgamma\left(\xa\right) &\geq& D_\gamma^{-1}\ma^{\frac{n-2}{2}\left(1-2\gamma\right)} \sum_{i=1}^N \left\vert \xai-\xa\right\vert^{-2} \Phi_\gamma\left(\xai,\xa\right)\\
&&+ D_{1-\gamma}^{-1}\etae\left(\delta\right) \sum_{i=1}^N \left\vert \xai-\xa\right\vert^{-2} \Phi_{1-\gamma}\left(\xai,\xa\right)\\
&&- D_\gamma \ma^{\frac{n-2}{2}\left(1-2\gamma\right)} \sum_{i=1}^N \Phi_\gamma\left(\xai,\xa\right) \\
&&- D_{1-\gamma} \etae\left(\delta\right)\sum_{i=1}^N \Phi_{1-\gamma}\left(\xai,\xa\right)\\
&&\geq \left(D_\gamma^{-2}\ra\left(\xa\right)^{-2}-ND_\gamma^2\right) \ma^{\frac{n-2}{2}\left(1-2\gamma\right)} \ra\left(\xa\right)^{-\left(n-2\right)\left(1-\gamma\right)}\\
&& + \left(D_{1-\gamma}^{-2}\ra\left(\xa\right)^{-2} -ND_{1-\gamma}^2\right) \etae\left(\delta\right) \ra\left(\xa\right)^{-\left(n-2\right)\gamma}\\
\fincal
We choose $\delta>0$ such that
$$D_\gamma^{-2}\delta^{-2} \geq 2ND_\gamma^2\hbox{ and }D_{1-\gamma}^{-2}\delta^{-2} \geq 2ND_{1-\gamma}^2$$
so that, using once again \eqref{estAB}, the above becomes
\begincal
\Delta\vphiepsgamma\left(\xa\right) &\geq& \frac{1}{2}D_\gamma^{-2}\ra\left(\xa\right)^{-2} \ma^{\frac{n-2}{2}\left(1-2\gamma\right)} \ra\left(\xa\right)^{-\left(n-2\right)\left(1-\gamma\right)}\\
&&+\frac{1}{2}D_{1-\gamma}^{-2}\ra\left(\xa\right)^{-2}\etae\left(\delta\right) \ra\left(\xa\right)^{-\left(n-2\right)\gamma}\\
&\geq & \frac{1}{2N}D_\gamma^{-3}\ra\left(\xa\right)^{-2} \ma^{\frac{n-2}{2}\left(1-2\gamma\right)} \sum_{i=1}^N \Phi_\gamma\left(\xai,\xa\right)\\
&&+\frac{1}{2N}D_{1-\gamma}^{-3}\ra\left(\xa\right)^{-2}\etae\left(\delta\right) \sum_{i=1}^N \Phi_{1-\gamma}\left(\xai,\xa\right)\\
&\geq & \frac{1}{2N}\left( \max\left\{D_\gamma,D_{1-\gamma}\right\}\right)^{-3}\ra\left(\xa\right)^{-2} \vphiepsgamma\left(\xa\right)\hskip.1cm.
\fincal
Coming back to (\ref{eqC.1.4}), we thus get that
$$\ra\left(\xa\right)^2 \ua\left(\xa\right)^{2^\star-2} \geq  \frac{1}{2n(n-2)N}\left( \max\left\{D_\gamma,D_{1-\gamma}\right\}\right)^{-3}\hskip.1cm.$$
Using point (iv) of Proposition \ref{prop:claimB}, it is easily check that one can choose $R>0$ large enough such that this is absurd. And with these choices of $\delta$ and $R$, \eqref{eq:sup} is proved.

\smallskip\noindent Assume that $\ra\left(\xa\right)\geq \delta$. Then we have that $\ua\left(\xa\right)\le \etae\left(\delta\right)$ so that, thanks to \eqref{estAB}, we get in this case that
$\ua\left(\xa\right) = O\left(\vphiepsgamma\left(\xa\right)\right)$.

\smallskip\noindent Assume that $\xa\in \partial B_{R\mai}(\xai)$ for some $1\le i\le N$. Then, up to increase a little bit $R$ so that $R\geq 4R_0$, $R_0$ as in (\ref{eq2.5}), we get thanks to (iii) of Proposition \ref{prop:claimB} that
$$\ua\left(\xa\right)=O\left(\mai^{1-\frac{n}{2}}\right)$$
while, using \eqref{estAB},
$$\vphiepsgamma \left(\xa\right) \geq D_\gamma^{-1} \ma^{\frac{n-2}{2}\left(1-2\gamma\right)} \left(R\mai\right)^{\left(2-n\right)\left(1-\gamma\right)}$$
so that, once again,
$$\ua\left(\xa\right) = O\left(\vphiepsgamma\left(\xa\right)\right)$$
since $\mai\le \ma$.

Thus we have proved so far that there exists $C>0$ such that
$$\ua\left(x\right)\le C \vphiepsgamma(x)\hbox{ in }\overline{B_{\delta_0}(x_0)}\setminus \bigcup_{i=1}^N B_{R\mai}(\xai)\hskip.1cm.$$
It remains to use point (A) of \eqref{estAB} above to prove \eqref{ineq:C1:1} and therefore Step \ref{prop:claimC}.1.1.

\medskip\noindent{\bf Step \ref{prop:claimC}.1.2:} We fix $x_0\in \partial\Omega$. Then there exists $\delta_0>0$ such that
\begin{equation}\label{ineq:C1:2}
\ua(x)\le C_\gamma \left(\ma^{\frac{n-2}{2}\left(1-2\gamma\right)} \ra(x)^{\left(2-n\right)\left(1-\gamma\right)}+\etae\left(\delta_\gamma\right) \ra(x)^{\left(2-n\right)\gamma}\right)
\end{equation}
for all $\epspositif$ and all $x\in (B_{\delta_0}(x_0)\cap\Omega)\setminus \bigcup_{i=1}^N B_{R_\gamma\mai}(\xai)$.

\medskip\noindent We prove the claim. Indeed, via the extension of Lemma \ref{lem:ext}, the proof goes roughly as in Step \ref{prop:claimC}.1. We only enlight here the main differences. As usual, since $x_0\in\partial\Omega$, we consider $\delta_{x_0}$, $U_{x_0}$ and a chart $\varphi$ as in Lemma \ref{lem:ext}. We let $\tue$ be the $C^2-$extension of $\ua$ on $U_{x_0}$: it satisfies that
\begin{equation}\label{eq:tue:C1}
\Delta_{\tg}\tue+\ea\tue=n(n-2)\tue^{\crit-1}\hbox{ in }U_{x_0}.
\end{equation}
We let $J:=\{i\in \{1,...,N\}/\, \lim_{\epstozero}\xai=x_0\}$ and we let $\delta_0>0$ such that
$$B_{\delta_0}(x_0)\subset U_{x_0}\hbox{ and }|\xai-x_0|\geq 2\delta_0\hbox{ for all }i\in \{1,...,N\}\setminus J.$$
For all $i\in J$, we define
$$\txie:=\pip^{-1}(\xai)=\varphi\circ\pi^{-1}\circ\varphi^{-1}(\xai),$$
where $\pi(x_1,x')=(-x_1,x')$ is the usual symmetry. We define
\begin{eqnarray}\label{eqC.1.2}
\vphiepsgamma(x)&:=& \ma^{\frac{n-2}{2}\left(1-2\gamma\right)}\sum_{i\in J} (\Phi_\gamma\left(\xai,x\right)+\Phi_\gamma\left(\txie,x\right)) \nonumber\\
&&+ \etae\left(\delta\right) \sum_{i\in J} (\Phi_{1-\gamma}\left(\xai,x\right)+\Phi_{1-\gamma}\left(\txie,x\right)) \nonumber\\
&&+\ma^{\frac{n-2}{2}\left(1-2\gamma\right)}\sum_{i\in J^c} \Phi_\gamma\left(\xai,x\right) + \etae\left(\delta\right) \sum_{i\in J^c} \Phi_{1-\gamma}\left(\xai,x\right) \nonumber
\end{eqnarray}
where $\Phi_\gamma$ and $\Phi_{1-\gamma}$ are as above and $\delta\in (0,\delta_0)$ will be chosen later on. For the sake of clearness, we define
$$\waR:=\overline{B_{\delta_0}(x_0)\cap\Omega}\setminus \left(\bigcup_{i=1}^N B_{R\mai}(\xai)\cup \bigcup_{i\in J}B_{R\mai}(\txie)\right).$$
We let $\xa\in \waR$ be such that
$$\sup_{x\in \waR} \frac{\tue}{\vphiepsgamma} = \frac{\tue\left(\xa\right)}{\vphiepsgamma\left(\xa\right)}\hskip.1cm.$$
We claim that, up to choose $\delta>0$ small enough and $R>0$ large enough,
\begin{equation}\label{eqC.1.2:bis}
\xa\in \partial \left(\bigcup_{i=1}^N B_{R\mai}(\xai)\cup \bigcup_{i\in J}B_{R\mai}(\txie)\right)\hbox{ or }\ra\left(\xa\right)\geq \delta
\end{equation}
for $\epspositif$ small. We prove it by contradiction. We assume on the contrary that
\begin{equation}\label{eqC.1.3:bis}
\xa\not\in  \partial \left(\bigcup_{i=1}^NB_{R\mai}(\xai)\cup \bigcup_{i\in J}B_{R\mai}(\txie)\right)\hbox{ and }\ra\left(\xa\right)< \delta
\end{equation}
for all $\epspositif$. First, it follows from the choice of $\delta_0$ and of $\etae(\delta)$ that $\xa\in B_{\delta_0}(\xa)$. Therefore, if $\xa\not\in\partial\Omega$, we write then thanks to (\ref{eqC.1.3:bis}) that
$$\frac{\Delta_g \tue\left(\xa\right)}{\tue\left(\xa\right)} \geq \frac{\Delta_g \vphiepsgamma\left(\xa\right)}{\vphiepsgamma\left(\xa\right)}\hskip.1cm.$$
Thanks to (\ref{eq:tue:C1}), we have that
$$\frac{\Delta_g \tue\left(\xa\right)}{\tue\left(\xa\right)} \le \tue\left(\xa\right)^{2^\star-2}$$
which leads to
\begin{equation*}
\frac{\Delta_g \vphiepsgamma\left(\xa\right)}{\vphiepsgamma\left(\xa\right)}\le \tue\left(\xa\right)^{2^\star-2}\hskip.1cm.
\end{equation*}
Since the coefficients of $\Delta_g$ are in $L^\infty$ with a continuous principal part (the metric $g$ is Lipschitz continuous), direct computations yield the existence of $D_\gamma>0$ such that

\smallskip (A') $D_\gamma^{-1}\le \left\vert x-y\right\vert^{\left(n-2\right)\left(1-\gamma\right)}\Phi_\gamma\left(x,y\right) \le D_\gamma $ for all $x,\,y\in \rn$, $x\neq y$.

\smallskip (B') $\frac{(\Delta_g)_y \Phi_\gamma\left(x,y\right)}{\Phi_\gamma\left(x,y\right)}\geq \frac{1}{D_\gamma} \left\vert x-y\right\vert^{-2} - D_\gamma$ for all $x,\, y\in \Om$, $x\neq y$.

\medskip\noindent And then the proof goes exactly as in Step \ref{prop:claimC}.1.1, using the convergence of the rescalings of $\tue$ proved in Proposition \ref{prop:claimB}. In case $\xa\in\partial\Omega$, we approximate it by a sequence of points in $\Omega$ and also conclude. This proves that there exists $C>0$ such that
\begin{eqnarray*}
\tue(x)&\leq &C \ma^{\frac{n-2}{2}\left(1-2\gamma\right)}\sum_{i\in J} (|x-\xai|^{(2-n)(1-\gamma)}+|x-\txie|^{(2-n)(1-\gamma)})\\
&& + \etae\left(\delta\right) \sum_{i\in J} (|x-\xai|^{(2-n)\gamma}+|x-\txie|^{(2-n)\gamma}) \\
&&+\ma^{\frac{n-2}{2}\left(1-2\gamma\right)}\sum_{i\in J^c} |x-\xai|^{(2-n)(1-\gamma)} + \etae\left(\delta\right) \sum_{i\in J^c} |x-\xai|^{(2-n)\gamma}
\end{eqnarray*}
for all $x\in \waR$. As easily checked, there exists $C>0$ such that $|x-\txie|\geq C|x-\xai|$ for all $x\in B_{\delta_0}(x_0)\cap\Omega$. Therefore, we get that there exists $C>0$ such that \eqref{ineq:C1:2} holds. This ends the proof of Step \ref{prop:claimC}.1.2.

\medskip\noindent Since $\overline{\Omega}$ is compact, Step \ref{prop:claimC}.1 is a consequence of Steps \ref{prop:claimC}.1.1 and \ref{prop:claimC}.1.2. \hfill $\Box$

\medskip\noindent{\bf Step \ref{prop:claimC}.2:} We claim that  there exists $C>0$ such that
$$\ua(x)\le C\left(\ma^{\frac{n-2}{2}} \raa(x)^{2-n}+\bue\right)$$
for all $x\in \bar{\Om}$ and all $\epspositif$.

\medskip\noindent{\it Proof of Step \ref{prop:claimC}.2:} We fix $0<\gamma<\frac{1}{n+2}$ in the following. We let $\left(\xa\right)$ be a sequence of points in $\bar{\Om}$ and we claim that
\begin{equation}\label{eqC.2.1}
\ua\left(\xa\right)\le \bue + O\left( \ma^{\frac{n-2}{2}} \raa\left(x\right)^{2-n} \right)+ o\bigl(\etae\left(\delta_\gamma\right)\bigr)\hskip.1cm.
\end{equation}
Note that this clearly implies the estimate of Step \ref{prop:claimC}.2 if we are then able to prove that $\etae\left(\delta_\gamma\right)=O\left(\bue\right)$. Let us prove this last fact before proving (\ref{eqC.2.1}). A direct consequence of (\ref{eqC.2.1}) and (\ref{eq2.4bis}) is that
$$\etae\left(\delta_\gamma\right) = O\left(\ma^{\frac{n-2}{2}}\right) +O\left(\bue\right)=O\left(\bue\right)\hskip.1cm,$$
thus proving the above assertion. We are left with the proof of (\ref{eqC.2.1}).

\smallskip\noindent{\bf Step \ref{prop:claimC}.2.1:} Assume first that $\raa\left(\xa\right) = O\left(\ma\right)$. We use then (iv) of Proposition \ref{prop:claimB} to write that
$$\raa\left(\xa\right)^{\frac{n-2}{2}} \ua\left(\xa\right) = \sum_{i=1}^N \raa\left(\xa\right)^{\frac{n-2}{2}}\Uie\left(\xa\right) + o(1)\hskip.1cm.$$
We can thus write that
\begincal
\ma^{1-\frac{n}{2}} \raa\left(\xa\right)^{n-2} \ua\left(\xa\right)
&= &O\left(\sum_{i=1}^N \ma^{1-\frac{n}{2}} \mai^{\frac{n}{2}-1} \raa\left(\xa\right)^{n-2} \left(\mai^2 + \left\vert \xai - \xa\right\vert^2\right)^{1-\frac{n}{2}}\right)\\
&&+o\left(\left(\frac{\raa\left(\xa\right)}{\ma}\right)^{\frac{n-2}{2}}\right)\\
&=&O\left(\sum_{i=1}^N \ma^{1-\frac{n}{2}} \mai^{\frac{n}{2}-1} \right)+ o(1)= O(1)
\fincal
since $\mai^2 + \left\vert \xai-\xa\right\vert^2 \le \raa\left(\xa\right)^2$ and $\mai\le \ma$ for all $1\le i\le N$. Thus the estimate (\ref{eqC.2.1}) clearly holds in this situation. This ends Step \ref{prop:claimC}.2.1.

\smallskip\noindent{\bf Step \ref{prop:claimC}.2.2:} Assume now that
\begin{equation}\label{eqC.2.2}
\frac{\raa\left(\xa\right)}{\ma}\to +\infty \hbox{ as }\epstozero\hskip.1cm.
\end{equation}
We use the Green representation formula, see Appendix A, and equation (\ref{eq1}) to write that
$$\ua\left(\xa\right)-\bue = \int_{\Om} G\left(\xa,x\right) \left(n(n-2)\ua(x)^{2^\star-1}-\ea \ua(x)\right)\, dx$$
where $G$ is the Green's function for the Neumann problem. Since adding a constant to $G$ does not change the representation above and using the pointwise estimates of Proposition \ref{app:prop:upp:bnd}, we get that
\begin{eqnarray*}
\ua\left(\xa\right)&\le &\bue + \int_{\Om} (G\left(\xa,x\right)+m(\Omega)) \left(n(n-2)\ua(x)^{2^\star-1}-\ea \ua(x)\right)\, dx\\
&\leq& \bue + n(n-2)\int_{\Om} (G\left(\xa,x\right)+m(\Omega))  \ua(x)^{2^\star-1}\, dx\\
&\leq &\bue + C\int_{\Om} \left\vert \xa-x\right\vert^{2-n} \ua(x)^{2^\star-1}\, dx\hskip.1cm.
\end{eqnarray*}
Using now Step \ref{prop:claimC}.1, this leads to
\begincal
\ua\left(\xa\right)&\le& \bue + O\left(\etae\left(\delta_\gamma\right)^{2^\star-1}\int_{\Om} \left\vert \xa-x\right\vert^{2-n} \ra\left(x\right)^{-(n+2)\gamma}\, dx\right)\\
&& + O\left(\ma^{\frac{n+2}{2}\left(1-2\gamma\right)} \int_{\left\{\ra(x)\geq R\ma\right\}} \left\vert \xa-x\right\vert^{2-n} \ra\left(x\right)^{-(n+2)\left(1-\gamma\right)}\, dx\right)\\
&& + O\left(\int_{\left\{\ra(x)\le R\ma\right\}} \left\vert \xa-x\right\vert^{2-n} \ua(x)^{2^\star-1}\, dx\right)\hskip.1cm,
\fincal
for some $R>>1$. The first term can easily be estimated since $2-(n+2)\gamma>0$. We get that
$$\etae\left(\delta_\gamma\right)^{2^\star-1}\int_{\Om} \left\vert \xa-x\right\vert^{2-n} \ra\left(x\right)^{-(n+2)\gamma}\, dx = O\left(\etae\left(\delta_\gamma\right)^{2^\star-1}\right)=o\bigl(\etae\left(\delta_\gamma\right)\bigr)$$
thanks to (\ref{eq2.3}). We estimate the second term:
\begincal
&&\int_{\left\{\ra(x)\geq R\ma\right\}} \left\vert \xa-x\right\vert^{2-n} \ra\left(x\right)^{-(n+2)\left(1-\gamma\right)}\, dx \\
&&\quad \le
\sum_{i=1}^N \int_{\left\{\left\vert \xai-x\right\vert \geq R\ma\right\}} \left\vert \xa-x\right\vert^{2-n} \left\vert x-\xai\right\vert^{-(n+2)\left(1-\gamma\right)}\, dx\\
&&\quad = O\left( \sum_{i=1}^N \ma^{n-(n+2)\left(1-\gamma\right)} \left\vert \xa-\xai\right\vert^{2-n}\right)\\
&&\quad = O\left(\ma^{n-(n+2)\left(1-\gamma\right)}  \raa\left(\xa\right)^{2-n}\right)
\fincal
since $n-(n+2)\left(1-\gamma\right)<0$ and $\left\vert \xai-\xa\right\vert \geq \frac{1}{2}\raa\left(\xa\right)$ for $\alpha$ large for all $1\le i\le N$ thanks to (\ref{eqC.2.2}). The last term is estimated thanks to (\ref{eq2}), to (\ref{eqC.2.2}) and to H\"older's inequalities by
\begincal
&&\int_{\left\{\ra(x)\le \ma\right\}} \left\vert \xa-x\right\vert^{2-n} \ua(x)^{2^\star-1}\, dx \\
&&\quad = O\left(\ra\left(\xa\right)^{2-n} \int_{\left\{\ra(x)\le \ma\right\}} \ua(x)^{2^\star-1}\, dx\right)\\
&&\quad = O\left(\ra\left(\xa\right)^{2-n}\ma^{\frac{n-2}{2}} \left(\int_{\left\{\ra(x)\le \ma\right\}} \ua(x)^{2^\star}\, dx\right)^{\frac{2^\star-1}{2^\star}}\right)\\
&&\quad =  O\left(\ma^{\frac{n-2}{2}} \ra\left(\xa\right)^{2-n}\right)\hskip.1cm.
\fincal
Combining all these estimates gives (\ref{eqC.2.1}) in this second case. This ends Step \ref{prop:claimC}.2.2. As already said, this ends the proof of Step \ref{prop:claimC}.2. \hfill $\Box$

\medskip The proof of Proposition \ref{prop:claimC} is now straightforward, using once again the Green representation formula. We write that, for any sequence $\left(\xa\right)$ of points in $\bar{\Om}$,
$$\ua\left(\xa\right)-\bue = n(n-2)\int_{\Om} G\left(\xa,x\right)\ua\left(x\right)^{2^\star-1}\, dx - \ea\int_\Om G\left(\xa,x\right)\ua\left(x\right)\, dx\hskip.1cm.$$
Let us write thanks to Appendix A, Step \ref{prop:claimC}.2 and Giraud's lemma that
\begincal
\int_\Om G\left(\xa,x\right)\ua\left(x\right)\, dx &=&O \left( \ma^{\frac{n-2}{2}} \int_{\Om} \left\vert \xa-x\right\vert^{2-n} \raa(x)^{2-n}\, dx \right)\\
&&+ O\left( \bue \int_{\Om} \left\vert \xa-x\right\vert^{2-n} \, dx\right)\\
&=&O\left(\ma^{\frac{n-2}{2}}\raa\left(\xa\right)^{2-n}\right) + O\left(\bue\right)
\fincal
(here one needs to spearate the case $n<4$, $n=4$ and $n>4$) and that
$$\int_{\Om} G\left(\xa,x\right)\ua\left(x\right)^{2^\star-1}\, dx = o\left(\bue\right) + O\left(\ma^{\frac{n-2}{2}}\raa\left(\xa\right)^{2-n}\right)\hskip.1cm.$$
Note that this last estimate has been proved in Step \ref{prop:claimC}.2.2. Combining these equations, we get the existence of some $C_1>0$ and some sequence $\betae$ as $\epstozero$ such that \eqref{est:claimC} holds. This proves Proposition \ref{prop:claimC}.\hfill$\Box$

\section{A sharp upper-estimate}\label{sec:sharp:upper}

Let us fix some notations. We let in the following
\begin{equation}\label{eqD.1}
\rie(x):= \min_{i\le j\le N} \left\vert \xai-x\right\vert\hbox{ and }\Rie(x)^2:=\min_{i\le j\le N} \left(\left\vert \xai-x\right\vert^2+\mai^2\right)\hskip.1cm.
\end{equation}
Note that $R_{1,\alpha}(x)=R_\alpha(x)$ and $r_{1,\alpha}(x)=r_\alpha(x)$.

\begin{defi} For $1\le i\le N$, we say that $\left({\mathcal I}_i\right)$ holds if there exists $C_i>0$ and a sequence $\betae$ as $\epstozero$ such that
\begin{equation}\label{eqD.2}
\left\vert \ua(x)-\bue -\sum_{j=1}^{i-1} \Vaj(x) \right\vert \le \betae\left(\bue +\sum_{j=1}^{i-1} \Uje(x)\right) + C_i\mai^{\frac{n-2}{2}} \Rie\left(x\right)^{2-n}
\end{equation}
for all $x\in \bar{\Om}$ and all $\epspositif$. Here, $\Vaj$ is as in Appendix B.
\end{defi}

\medskip This section is devoted to the proof of the following~:

\begin{prop}\label{prop:claimD} $\left({\mathcal I}_N\right)$ holds.
\end{prop}

\medskip\noindent{\it Proof of Proposition \ref{prop:claimD}:} Thanks to Proposition \ref{prop:claimC}, we know that $\left({\mathcal I}_1\right)$ holds. The aim of the rest of this section is to prove by induction on $\kappa$ that $\left({\mathcal I}_{\kappa}\right)$ holds for all $1\le \kappa\le N$. In the following, we fix $1\le \kappa\le N-1$ and we assume that $\left({\mathcal I}_{\kappa}\right)$ holds. The aim is to prove that $\left({\mathcal I}_{\kappa+1}\right)$ holds.
We proceed in several steps. Let us first set up some notations. In the following, we fix
\begin{equation}\label{eqD.3}
0<\gamma<\frac{1}{n+2}\hskip.1cm.
\end{equation}
We let, for any $1\le i\le \kappa$,
\begin{equation}\label{eqD.4}
\Psiie(x) := \min\left\{\mai^{\frac{n-2}{2}\left(1-2\gamma\right)}\Phi_\gamma\left(\xai,x\right)\, ;\, A_0 \mai^{-\frac{n-2}{2}\left(1-2\gamma\right)}\Phi_{1-\gamma}\left(\xai,x\right)\right\}
\end{equation}
for $x\in \Om\setminus \left\{\xai\right\}$ where
\begin{equation}\label{eqD.5}
A_0 := \frac{1}{D_\gamma D_{1-\gamma}} \left(4R_0\right)^{-\left(n-2\right)\left(1-2\gamma\right)}\hskip.1cm.
\end{equation}
Here $\Phi_\gamma$, $\Phi_{1-\gamma}$, $D_\gamma$ and $D_{1-\gamma}$ are given by \eqref{estAB} and $R_0$ is as in (\ref{eq2.5}). With this choice of $A_0$, we have with \eqref{estAB} that
$$\Psiie(x) = A_0 \mai^{-\frac{n-2}{2}\left(1-2\gamma\right)}\Phi_{1-\gamma}\left(\xai,x\right)$$
if $\left\vert \xai-x\right\vert \le 2R_0\mai$. Similarly, $\Psiie(x) =\mai^{\frac{n-2}{2}\left(1-2\gamma\right)}\Phi_\gamma\left(\xai,x\right)$ if $x$ is far enough from $\xai$. Note also that we have that
\begin{equation}\label{eqD.5bis}
A_1^{-1}  \le \frac{\Psiie(x)}{\Uie\left(x\right) \left(\frac{\left\vert \xai-x\right\vert}{\mai}+\frac{\mai}{\left\vert \xai-x\right\vert}\right)^{\left(n-2\right)\gamma}} \le A_1
\end{equation}
for all $x\in \bar{\Om}\setminus \left\{\xai\right\}$ and all $\epspositif$ for some $A_1>0$ independent of $\alpha$. We also define
\begin{equation}\label{eqD.6}
\psie(x) := \max\left\{\bue\, ;\, \ma^{\frac{n-2}{2}\left(1-2\gamma\right)}\right\} \sum_{i=1}^N \Phi_{1-\gamma}\left(\xai,x\right)
\end{equation}
and
\begin{equation}\label{eqD.7}
\Thetaeps\left(x\right):= \sum_{i=\kappa+1}^N \Phi_\gamma\left(\xai,x\right)\hskip.1cm.
\end{equation}
We set, for $1\le i\le \kappa$,
\begin{equation}\label{eqD.8}
\Om_{i,\alpha} := \left\{x\in \bar{\Om} \hbox{ s.t. } \Psiie(x)\geq  \Psije(x)\hbox{ for all }1\le j\le \kappa\right\}\hskip.1cm.
\end{equation}
We also fix $A_2>0$ that will be chosen later and we define $\nukappa$ by
\begin{equation}\label{eqD.10}
\nukappa^{\frac{n-2}{2}\left(1-2\gamma\right)} := \max\left\{ \mukappaun^{\frac{n-2}{2}\left(1-2\gamma\right)} \, ;\, \max_{1\le i\le \kappa} \sup_{\tomegaie} \frac{\Psiie}{\Thetaeps}\right\}
\end{equation}
where
\begin{equation}\label{eqD.11}
\tomegaie:= \left\{x\in \omegaie \hbox{ s.t. } \left\vert \xai-x\right\vert^2 \left\vert \ua(x) - \bue - \sum_{j=1}^\kappa \Vaj(x)\right\vert^{2^\star-2}\geq A_2\right\}\hskip.1cm.
\end{equation}
In the above definition, the suprema are by definition $-\infty$ if the set is empty. Remark that, in all these notations, we did not show the dependence in $\gamma$ of the various objects since $\gamma$ is fixed for all this section.

\medskip\noindent{\bf Step \ref{prop:claimD}.1:} We claim that $\nukappa=O\left(\mukappa\right)$ when $\alpha\to +\infty.$

\medskip\noindent{\it Proof of Step \ref{prop:claimD}.1:} This is clearly true if $\nukappa=\mukappaun$ since $\mukappaun\le \mukappa$.

\smallskip\noindent{\bf Step \ref{prop:claimD}.1.1:} Let us assume that there exists $\xa\in \tomegaie$ for some $1\le i\le \kappa$ such that
$$\Psiie\left(\xa\right) =\nukappa^{\frac{n-2}{2}\left(1-2\gamma\right)}  \Thetaeps\left(\xa\right)$$
which implies thanks to \eqref{estAB} that
\begin{equation}\label{eqD.1.1}
\nukappa^{1-2\gamma} = O\left(\Rkappauneps\left(\xa\right)^{2\left(1-\gamma\right)} \Psiie\left(\xa\right)^{\frac{2}{n-2}}\right)\hskip.1cm.
\end{equation}
Since $\left({\mathcal I}_\kappa\right)$ holds and $\xa\in \tomegaie$, we also have that
$$A_2 \le o(1) + o\left(\left\vert \xai-\xa\right\vert^2 \sum_{j=1}^{\kappa-1} \Uje\left(\xa\right)^{2^\star-2}\right) + O\left(\frac{\mukappa^2\left\vert \xai-\xa\right\vert^2}{\Rkappaeps\left(\xa\right)^4}\right)\hskip.1cm.$$
Noting that
$$\left\vert \xai-\xa\right\vert^{2}\Uie\left(\xa\right)^{2^\star-2} = \left(\frac{\left\vert \xai-\xa\right\vert}{\mai}+\frac{\mai}{\left\vert \xai-\xa\right\vert}\right)^{-2}$$
and using (\ref{eqD.5bis}), we get since $\xa\in \Om_{i,\alpha}$ that
\begincal
&&\left\vert \xai-\xa\right\vert^2 \sum_{j=1}^{\kappa-1} \Uje\left(\xa\right)^{2^\star-2}\\
&&\quad = O\left(\sum_{j=1}^{\kappa-1}\left(\frac{\left\vert \xai-\xa\right\vert}{\mai}+\frac{\mai}{\left\vert \xai-\xa\right\vert}\right)^{4\gamma-2}\left(\frac{\left\vert \xaj-\xa\right\vert}{\maj}+\frac{\maj}{\left\vert \xaj-\xa\right\vert}\right)^{-4\gamma}\right)\\\
&&\quad=O(1)
\fincal
since $\gamma<\frac{1}{2}$. Thus the previous equation leads to
\begin{equation}\label{eqD.1.2}
\Rkappaeps\left(\xa\right)^2 = O\bigl(\mukappa\left\vert \xai-\xa\right\vert\bigr)\hskip.1cm.
\end{equation}
If $\Rkappaeps\left(\xa\right)=\Rkappauneps\left(\xa\right)$, then (\ref{eqD.1.1}) and (\ref{eqD.1.2}) together with (\ref{eqD.5bis}) lead to
\begincal
\nukappa^{1-2\gamma}&=&O\left(\mukappa^{1-\gamma}\left\vert \xai-\xa\right\vert^{1-\gamma} \Psiie\left(\xa\right)^{\frac{2}{n-2}}\right)\\
&=&O\left(\mukappa^{1-\gamma}\left\vert \xai-\xa\right\vert^{1-\gamma}
\left(\frac{\left\vert \xai-\xa\right\vert}{\mai}+\frac{\mai}{\left\vert \xai-\xa\right\vert}\right)^{2\gamma}
\mai \left(\mai^2+\left\vert \xai-\xa\right\vert^2\right)^{-1}\right)\\
&=&O\left(\mukappa^{1-\gamma}\mai^{-\gamma} \left(\frac{\left\vert \xai-\xa\right\vert}{\mai}+\frac{\mai}{\left\vert \xai-\xa\right\vert}\right)^{3\gamma-1}\right)\\
&=& O\left(\mukappa^{1-\gamma}\mai^{-\gamma}\right) = O\left(\mukappa^{1-2\gamma}\right)
\fincal
since $\gamma<\frac{1}{3}$ and $i\le \kappa$ so that $\mai\geq \mukappa$. The estimate of Step \ref{prop:claimD}.1 is thus proved in this case. This ends Step \ref{prop:claimD}.1.1.

\medskip\noindent{\bf Step \ref{prop:claimD}.1.2:} Assume now that $\Rkappaeps\left(\xa\right)<\Rkappauneps\left(\xa\right)$ so that
$\Rkappaeps\left(\xa\right)^2 = \left\vert \xkappae-\xa\right\vert^2 +\mukappa^2$. Then (\ref{eqD.1.2}) becomes
\begin{equation}\label{eqD.1.3}
\left\vert \xkappae-\xa\right\vert^2 +\mukappa^2=O\bigl(\mukappa\left\vert \xai-\xa\right\vert\bigr)\hskip.1cm.
\end{equation}
If $i=\kappa$, we then get that $\left\vert \xai-\xa\right\vert = O\left(\mai\right)$. Since $\Rkappauneps\left(\xa\right) \geq \Rkappaeps (\xa)\geq  \mai$ in this case, we can deduce from Proposition \ref{prop:claimB}, (iii), that
$$ \left\vert \xai-\xa\right\vert^2 \left\vert \ua(\xa) - \bue - \sum_{j=1}^\kappa \Vaj(\xa)\right\vert^{2^\star-2}\to 0$$
as $\epstozero$, thus contradicting the fact that $\xa\in \tomegaie$. If $i<\kappa$, we write thanks to (\ref{eqD.5bis}) and to the fact that $\Psiie\left(\xa\right)\geq \psikappae\left(\xa\right)$ (since $\xa\in \Om_{i,\alpha}$) that
\begincal
&&\Uke^{\frac{2}{n-2}}\left(\xa\right) \left(\frac{\left\vert \xkappae-\xa\right\vert}{\mukappa}+\frac{\mukappa}{\left\vert \xkappae-\xa\right\vert}\right)^{2\gamma}\\
&&\quad =O\left( \Uie\left(\xa\right)^{\frac{2}{n-2}} \left(\frac{\left\vert \xai-\xa\right\vert}{\mai}+\frac{\mai}{\left\vert \xai-\xa\right\vert}\right)^{2\gamma}\right)
\fincal
which gives thanks to (\ref{eqD.1.3}) that
\begincal
&&\left(\frac{\left\vert \xkappae-\xa\right\vert}{\mukappa}+\frac{\mukappa}{\left\vert \xkappae-\xa\right\vert}\right)^{2\gamma}\\
&&\quad = O\left( \left\vert \xai-\xa\right\vert \Uie\left(\xa\right)^{\frac{2}{n-2}} \left(\frac{\left\vert \xai-\xa\right\vert}{\mai}+\frac{\mai}{\left\vert \xai-\xa\right\vert}\right)^{2\gamma}\right)\\
&&\quad = O\left(\left(\frac{\left\vert \xai-\xa\right\vert}{\mai}+\frac{\mai}{\left\vert \xai-\xa\right\vert}\right)^{2\gamma-1}\right)\hskip.1cm.
\fincal
Since $\gamma<\frac{1}{2}$, this leads clearly to
$$C^{-1} \le \frac{\left\vert \xai-\xa\right\vert}{\mai}\le C \hbox{ and } C^{-1}\le \frac{\left\vert \xkappae-\xa\right\vert}{\mukappa}\le C$$
for some $C>0$ independent of $\alpha$. This implies that $\mukappa=o\left(\mai\right)$ thanks to Proposition \ref{prop:claimB}, (ii). One then easily deduces from (\ref{eqD.5bis}) that
$$\frac{\Psiie\left(\xa\right)}{\psikappae\left(\xa\right)}= O\left(\left(\frac{\mukappa}{\mai}\right)^{\frac{n-2}{2}}\right)=o(1)\hskip.1cm,$$
which contradicts the fact that $\xa\in \Om_{i,\alpha}$. This ends Step \ref{prop:claimD}.1.2, and therefore this proves Step \ref{prop:claimD}.1.\hfill $\Box$

\medskip\noindent{\bf Step \ref{prop:claimD}.2:} We claim that there exists $A_3>0$ such that
\begin{eqnarray}\label{est:claimD2}
\ua(x)&\le & A_3\left(\sum_{i=1}^\kappa \Psiie(x) + \nukappa^{\frac{n-2}{2}\left(1-2\gamma\right)} \petitrkappauneps(x)^{\left(2-n\right)\left(1-\gamma\right)}\right.\nonumber\\
&&\qquad \left. + \max\left\{\bue\, ;\, \ma^{\frac{n-2}{2}\left(1-2\gamma\right)}\right\}\ra(x)^{\left(2-n\right)\gamma}\right)
\end{eqnarray}
for all $x\in \Omega\setminus \bigcup_{i=\kappa+1}^N B_{R_0\mai}(\xai)$.

\medskip\noindent{\it Proof of Step \ref{prop:claimD}.2:} As in the proof of Step  \ref{prop:claimC}.1, the proof of Step  \ref{prop:claimD}.2 requires to distinguish whether we consider points in the interior or on the boundary of $\Omega$. We only prove the estimate for interior points and we refer to Step  \ref{prop:claimC}.1.2 for the extension of the proof to the boundary. We fix $x_0\in\Omega$ and $\delta_0>0$ such that $B_{\delta_0}(x_0)\subset\subset\Omega$. We let $\xa\in \overline{B_{\delta_0}(x_0)}\setminus \bigcup_{i=\kappa+1}^N B_{R_0\mai}(\xai)$ be such that
\begin{equation}\label{eqD.2.1}
\begin{split}
&\frac{\ua\left(\xa\right)}{\sum_{i=1}^\kappa \Psiie\left(\xa\right) + \nukappa^{\frac{n-2}{2}\left(1-2\gamma\right)} \Thetaeps\left(\xa\right)+ \psie\left(\xa\right)} \\
&\quad
=\sup_{ \overline{B_{\delta_0}(x_0)}\setminus \bigcup_{i=\kappa+1}^N B_{R_0\mai}(\xai)} \frac{\ua}{\sum_{i=1}^\kappa \Psiie + \nukappa^{\frac{n-2}{2}\left(1-2\gamma\right)} \Thetaeps+ \psie}\hskip.1cm.
\end{split}
\end{equation}
and we assume by contradiction that
\begin{equation}\label{eqD.2.2}
\frac{\ua\left(\xa\right)}{\sum_{i=1}^\kappa \Psiie\left(\xa\right) + \nukappa^{\frac{n-2}{2}\left(1-2\gamma\right)} \Thetaeps\left(\xa\right)+ \psie\left(\xa\right)}\to +\infty \hbox{ as }\epstozero\hskip.1cm.
\end{equation}
Thanks to the definition (\ref{eqD.6}) of $\psie$ and to the fact that $\left({\mathcal I}_\kappa\right)$ holds, it is clear that
\begin{equation}\label{eqD.2.3}
\ra\left(\xa\right)\to 0\hbox{ as }\epstozero\hskip.1cm.
\end{equation}
We claim that
\begin{equation}\label{eqD.2.4}
\frac{\left\vert \xai-\xa\right\vert}{\mai}\to +\infty \hbox{ as }\epstozero\hbox{ for all } \kappa+1\le i \le N\hskip.1cm.
\end{equation}
Assume on the contrary that there exists $\kappa+1\le i\le N$ such that $\left\vert \xai-\xa\right\vert=O\left(\mai\right)$. Since $\left\vert \xai-\xa\right\vert\geq R_0\mai$ and by the definition (\ref{eq2.5}) of $R_0$, we then get thanks to Proposition \ref{prop:claimB}, (iii), that
$\ua\left(\xa\right)=O\left( \mai^{1-\frac{n}{2}}\right)$. But, thanks to (\ref{eqD.10}) and to \eqref{estAB}, we also have that
\begincal
\nukappa^{\frac{n-2}{2}\left(1-2\gamma\right)} \Thetaeps\left(\xa\right)&\geq &D_\gamma^{-1}
\mukappaun^{\frac{n-2}{2}\left(1-2\gamma\right)} \left\vert \xai-x\right\vert^{\left(2-n\right)\left(1-\gamma\right)} \\
&\geq & D_\gamma^{-1} \mai^{\frac{n-2}{2}\left(1-2\gamma\right)}  \left\vert \xai-x\right\vert^{\left(2-n\right)\left(1-\gamma\right)}
\fincal
so that
$$\frac{\ua\left(\xa\right)}{\nukappa^{\frac{n-2}{2}\left(1-2\gamma\right)} \Thetaeps\left(\xa\right)} = O\left(\left(\frac{\left\vert \xai-x\right\vert}{\mai}\right)^{\left(n-2\right)\left(1-\gamma\right)}\right)=O(1)\hskip.1cm,$$
thus contradicting (\ref{eqD.2.2}). So we have proved that (\ref{eqD.2.4}) holds. With the same argument performed with $\Psiie$, we also know that, for any $1\le i\le \kappa$,
\begin{equation}\label{eqD.2.5}
\hbox{ either }\left\vert \xai-\xa\right\vert\le R_0\mai \hbox{ or }\frac{\left\vert \xai-\xa\right\vert}{\mai}\to +\infty\hbox{ as }\epstozero\hskip.1cm.
\end{equation}
In particular, we can write thanks to (\ref{eqD.2.1}), to (\ref{eqD.2.4}) and to (\ref{eqD.2.5}) (which ensures that the $\Psiie$'s are smooth in a small neighbourhood of $\xa$, see the remark following (\ref{eqD.5})) that$$\frac{\Delta \ua\left(\xa\right)}{\ua\left(\xa\right)} \geq \frac{\Delta \left(\sum_{i=1}^\kappa \Psiie + \nukappa^{\frac{n-2}{2}\left(1-2\gamma\right)} \Thetaeps+ \psie\right)}{\sum_{i=1}^\kappa \Psiie + \nukappa^{\frac{n-2}{2}\left(1-2\gamma\right)} \Thetaeps+ \psie}\left(\xa\right)$$
for $\alpha$ large. We write thanks to (\ref{eq1}) that
$$\frac{\Delta \ua\left(\xa\right)}{\ua(\xa)} \le n(n-2) \ua\left(\xa\right)^{2^\star-2}$$
so that the above becomes
\begincal
&&\Delta \left(\sum_{i=1}^\kappa \Psiie + \nukappa^{\frac{n-2}{2}\left(1-2\gamma\right)} \Thetaeps+ \psie\right)\left(\xa\right)\\
&&\quad  \le n(n-2) \ua\left(\xa\right)^{2^\star-2}\left(\sum_{i=1}^\kappa \Psiie + \nukappa^{\frac{n-2}{2}\left(1-2\gamma\right)} \Thetaeps+ \psie\right)\left(\xa\right)\hskip.1cm.
\fincal
Writing thanks to (A), (B) that
$$\Delta \Psiie\left(\xa\right) \geq \left(\frac{1}{D_\gamma+D_{1-\gamma}} \left\vert \xai-\xa\right\vert^{-2}-D_\gamma - D_{1-\gamma}\right) \Psiie\left(\xa\right)$$
for all $1\le i\le \kappa$, that
$$\Delta \Thetaeps\left(\xa\right) \geq \left(\frac{1}{N}D_\gamma^{-3} \petitrkappauneps\left(\xa\right)^{-2} -N D_\gamma\right)\Thetaeps\left(\xa\right)$$
and that
$$\Delta \psie\left(\xa\right) \geq \left(\frac{1}{N}D_{1-\gamma}^{-3} \ra\left(\xa\right)^{-2} -N D_{1-\gamma}\right)\psie\left(\xa\right) \hskip.1cm,$$
we get that
\begin{equation}\label{eqD.2.6}
\begin{split}
0\geq & \sum_{i=1}^\kappa \left(\left\vert \xai-\xa\right\vert^{-2}-C_\gamma-C_\gamma \ua\left(\xa\right)^{2^\star-2}\right) \Psiie\left(\xa\right)\\
& +  \left( \petitrkappauneps\left(\xa\right)^{-2} -C_\gamma-C_\gamma \ua\left(\xa\right)^{2^\star-2}\right)\nukappa^{\frac{n-2}{2}\left(1-2\gamma\right)}\Thetaeps\left(\xa\right)\\
&+  \left( \ra\left(\xa\right)^{-2} -C_\gamma - C_\gamma \ua\left(\xa\right)^{2^\star-2}\right)\psie\left(\xa\right)
\end{split}
\end{equation}
where $C_\gamma>0$ is large enough and independent of $\alpha$ and $\delta$. We let in the following $1\le i\le \kappa$ be such that $\xa\in \Om_{i,\alpha}$. We can then deduce from (\ref{eqD.2.6}) that
\begin{equation}\label{eqD.2.7}
\begin{split}
0\geq & \left(\left\vert \xai-\xa\right\vert^{-2}-\kappa C_\gamma- \kappa C_\gamma \ua\left(\xa\right)^{2^\star-2}\right) \Psiie\left(\xa\right)\\
& +  \left( \petitrkappauneps\left(\xa\right)^{-2} -C_\gamma-C_\gamma \ua\left(\xa\right)^{2^\star-2}\right)\nukappa^{\frac{n-2}{2}\left(1-2\gamma\right)}\Thetaeps\left(\xa\right)\\
&+  \left( \ra\left(\xa\right)^{-2} -C_\gamma - C_\gamma \ua\left(\xa\right)^{2^\star-2}\right)\psie\left(\xa\right)
\end{split}
\end{equation}
Thanks to (\ref{eqD.6}) and to (\ref{eqD.2.2}), we know that
\begin{equation}\label{eqD.2.8}
\bue=o\bigl(\ua\left(\xa\right)\bigr)\hskip.1cm.
\end{equation}
We also know thanks to (\ref{eqD.2.2}) that
\begin{equation}\label{eqD.2.9}
\Uje\left(\xa\right) = o\bigl(\ua\left(\xa\right)\bigr)
\end{equation}
for all $1\le j\le \kappa$ since
$$\Uje\left(\xa\right) = O\bigl(\Psiie\left(\xa\right)\bigr)$$
for all $1\le j\le \kappa$. Note also that, thanks to (\ref{eqD.2.4}), we have that
\begin{equation}\label{eqD.2.10}
\Rkappauneps\left(\xa\right)^2 \Uje\left(\xa\right)^{2^\star-2}\to 0\hbox{ as }\epstozero\hbox{ for all }\kappa+1\le j\le N\hskip.1cm.
\end{equation}
Thus we can deduce from Proposition \ref{prop:claimB}, (iv), together with (\ref{eqD.2.9}) and (\ref{eqD.2.10}) that
\begin{equation}\label{eqD.2.11}
\raa\left(\xa\right)^2 \ua\left(\xa\right)^{2^\star-2} \to 0\hbox{ as }\epstozero\hskip.1cm.
\end{equation}
Thanks to (\ref{eqD.2.3}) and to this last equation, we can transform (\ref{eqD.2.7}) into
\begin{equation}\label{eqD.2.12}
\begin{split}
0\geq & \left(\left\vert \xai-\xa\right\vert^{-2}-\kappa C_\gamma- \kappa C_\gamma \ua\left(\xa\right)^{2^\star-2}\right) \Psiie\left(\xa\right)\\
& +  \left( \petitrkappauneps\left(\xa\right)^{-2} -C_\gamma-C_\gamma \ua\left(\xa\right)^{2^\star-2}\right)\nukappa^{\frac{n-2}{2}\left(1-2\gamma\right)}\Thetaeps\left(\xa\right)\\
&+  \bigl(1+o(1)\bigr)\ra\left(\xa\right)^{-2}\psie\left(\xa\right)
\end{split}
\end{equation}
Since $\left({\mathcal I}_\kappa\right)$ holds, we get thanks to (\ref{eqD.2.8}) and (\ref{eqD.2.9}) that
\begin{equation}\label{bnd:ue:mk:rk}
\ua\left(\xa\right)^{2^\star-2} = O\left(\mukappa^2 \Rkappaeps\left(\xa\right)^{-4}\right)\hskip.1cm.
\end{equation}
We claim that we then have that
\begin{equation}\label{eqD.2.13}
\ua\left(\xa\right)^{2^\star-2} = O\left(\mukappa^2 \Rkappauneps\left(\xa\right)^{-4}\right)\hskip.1cm.
\end{equation}
Indeed, if \eqref{eqD.2.13} does not hold, then $\Rkappauneps\left(\xa\right)=o(\Rkappaeps\left(\xa\right))$ when $\epstozero$ and then $\Rkappaeps\left(\xa\right)=\sqrt{\mukappa^2+|\xa-\xkappae|^2}$, which contradicts \eqref{eqD.2.9} and \eqref{bnd:ue:mk:rk}. This proves \eqref{eqD.2.13}.

\smallskip\noindent We claim that this implies that
\begin{equation}\label{eqD.2.14}
\Rkappauneps\left(\xa\right)^2 \ua\left(\xa\right)^{2^\star-2}\to 0\hbox{ as }\epstozero\hskip.1cm.
\end{equation}
Indeed, if not, (\ref{eqD.2.13}) would imply that
$$\Rkappauneps\left(\xa\right) = O\left(\mukappa\right)$$
while (\ref{eqD.2.11}) would imply that $\raa\left(\xa\right) = o\left(\Rkappauneps\left(\xa\right)\right)$, which would in turn imply that there exists $1\le j\le \kappa$ such that
$$\left\vert \xaj-\xai\right\vert^2 +\maj^2 = \raa\left(\xa\right)^2 = o\left(\Rkappauneps\left(\xa\right)^2\right)=o\left(\mukappa^2\right)$$
which turns out to be absurd since $\maj\geq \mukappa$. Thus (\ref{eqD.2.14}) holds. Note that (\ref{eqD.2.13}) together with (\ref{eqD.2.8}) also implies that
\begin{equation}\label{eqD.2.15}
\Rkappauneps\left(\xa\right)\to 0\hbox{ as }\epstozero
\end{equation}
thanks to (\ref{eq2.4bis}). Thanks to (\ref{eqD.2.14}) and (\ref{eqD.2.15}), we can transform (\ref{eqD.2.12}) into
\begin{equation}\label{eqD.2.16}
\begin{split}
0\geq & \left(\left\vert \xai-\xa\right\vert^{-2}-\kappa C_\gamma- \kappa C_\gamma \ua\left(\xa\right)^{2^\star-2}\right) \Psiie\left(\xa\right)\\
& +  \bigl(1+o(1)\bigr) \petitrkappauneps\left(\xa\right)^{-2} \nukappa^{\frac{n-2}{2}\left(1-2\gamma\right)}\Thetaeps\left(\xa\right)
+  \bigl(1+o(1)\bigr)\ra\left(\xa\right)^{-2}\psie\left(\xa\right) \hskip.1cm.
\end{split}
\end{equation}
If $\xa\not\in \tomegaie$, we can transform this into
\begincal
0&\geq & \left(1+o(1) -  \kappa C_\gamma A_2 -\kappa C_\gamma \left\vert \xai-\xa\right\vert^{2}\right) \left\vert \xai-\xa\right\vert^{-2}\Psiie\left(\xa\right)\\
&& +  \bigl(1+o(1)\bigr) \petitrkappauneps\left(\xa\right)^{-2} \nukappa^{\frac{n-2}{2}\left(1-2\gamma\right)}\Thetaeps\left(\xa\right)
+  \bigl(1+o(1)\bigr)\ra\left(\xa\right)^{-2}\psie\left(\xa\right)
\fincal
thanks to (\ref{eqD.2.8}) and (\ref{eqD.2.9}). Up to taking $A_2>0$ small enough, this leads to
$$\ra\left(\xa\right)^{-2}\psie\left(\xa\right) = O\left(\mai^{\frac{n-2}{2}\left(1-2\gamma\right)}\right)\hskip.1cm.$$
Thanks to (\ref{eqD.2.3}), (\ref{eqD.6}) and \eqref{estAB}, this is clearly absurd. Thus we have that $\xa\in \tomegaie$. Coming back to (\ref{eqD.2.16}), we have that
$$\nukappa^{\frac{n-2}{2}\left(1-2\gamma\right)}\Thetaeps\left(\xa\right) = O\left( \left(\ua\left(\xa\right)^{2^\star-2}+1\right) \petitrkappauneps\left(\xa\right)^2 \Psiie\left(\xa\right)\right)\hskip.1cm.$$
Using (\ref{eqD.2.14}) and (\ref{eqD.2.15}), this leads to
$$\nukappa^{\frac{n-2}{2}\left(1-2\gamma\right)}\Thetaeps\left(\xa\right) = o\bigl(\Psiie\left(\xa\right)\bigr)\hskip.1cm,$$
which clearly contradicts the definition (\ref{eqD.10}) of $\nukappa$ since $\xa\in \tomegaie$. We have thus proved that (\ref{eqD.2.2}) leads to a contradiction. Using \eqref{estAB}, this proves \eqref{est:claimD2} and permits to end the proof of Step \ref{prop:claimD}.2. \hfill $\Box$

\medskip\noindent{\bf Step \ref{prop:claimD}.3:} We claim that  there exists $A_4>0$ such that
\begin{equation}\label{est:claimD3}
\ua(x) \le A_4\left(\sum_{i=1}^\kappa \Uie(x) + \bue +\nukappa^{\frac{n-2}{2}} \Rkappauneps(x)^{2-n}\right)
\end{equation}
for all $x\in \bar{\Om}$ and all $\epspositif$.

\medskip\noindent{\it Proof of Step \ref{prop:claimD}.3:} We let $\left(\xa\right)$ be a sequence of points in $\bar{\Om}$. We aim at proving that
\begin{equation}\label{eqD.3.1}
\ua\left(\xa\right) = O\left(\sum_{i=1}^\kappa \Uie(\xa) + \bue +\nukappa^{\frac{n-2}{2}} \Rkappauneps(\xa)^{2-n}\right)\hskip.1cm.
\end{equation}
Since $\left({\mathcal I}_\kappa\right)$ holds, and distinguishing whether $\Rkappaeps(\xa)=o(\Rkappauneps(\xa))$ or not, we already know that (\ref{eqD.3.1}) holds if
$$\mukappa^{\frac{n-2}{2}}\Rkappauneps\left(\xa\right)^{2-n} = O\bigl(\Uie\left(\xa\right)\bigr)$$
for some $1\le i\le \kappa$. Thus we can assume from now on that
\begin{equation}\label{eqD.3.2}
\Rkappauneps\left(\xa\right)^2 = o\left(\mai\mukappa\right) + o\left(\frac{\mukappa}{\mai} \left\vert \xai-\xa\right\vert^2\right)
\end{equation}
for all $1\le i\le \kappa$. This implies in particular that, for $\alpha$ large,
\begin{equation}\label{eqD.3.3}
\raa\left(\xa\right) = \Rkappauneps\left(\xa\right)=o(1)\hskip.1cm.
\end{equation}
Using Step \ref{prop:claimD}.1 and (iv) of Proposition \ref{prop:claimB}, we also get that (\ref{eqD.3.1}) holds as soon as $\Rkappauneps\left(\xa\right) = O\left(\nukappa\right)$. Thus we can assume from now on that
\begin{equation}\label{eqD.3.4}
\frac{\Rkappauneps\left(\xa\right)}{\nukappa}\to +\infty\hbox{ as }\epstozero\hskip.1cm.
\end{equation}
We now use the Green representation formula to estimate $\ua\left(\xa\right)$. As in Step \ref{prop:claimC}.2.2, we write that
$$\ua\left(\xa\right) \le \bue +n(n-2)\int_{\Om} (G\left(\xa,x\right)+m(\Omega))\ua\left(x\right)^{2^\star-1}\, dx $$
since $\ua$ satisfies equation (\ref{eq1}). This leads to
\begin{equation}\label{eqD.3.5}
\ua\left(\xa\right)-\bue \le  C_0 n(n-2)\int_{\Om} \left\vert \xa-x\right\vert^{2-n} \ua\left(x\right)^{2^\star-1}\, dx\hskip.1cm.
\end{equation}
Noting that $\petitrkappauneps(\xa)\asymp\Rkappauneps\left(\xa\right)$, we write thanks to (\ref{eqD.3.4}) and to Step \ref{prop:claimD}.1 that
\begin{equation}\label{eq:added}
\int_{\left\{x\in \Om,\, \petitrkappauneps\left(x\right)\le R_0\nukappa\right\}} \left\vert \xa-x\right\vert^{2-n} \ua\left(x\right)^{2^\star-1}\, dx =O\left( \nukappa^{\frac{n-2}{2}}\Rkappauneps\left(\xa\right)^{2-n} \right) \end{equation}
using H\"older's inequalities and (\ref{eq2}), where $R_0$ is as in Step \ref{prop:claimD}.2. Noting that
$$\bigcup_{i=\kappa+1}^N B_{R_0\mai}(\xai)\subset \left\{x\in \Om,\, \petitrkappauneps\left(x\right)\le R_0\nukappa\right\},$$
we write now thanks to this last inclusion, to \eqref{eq:added} and to \eqref{est:claimD2} that
\begincal
&&\int_{\Om} \left\vert \xa-x\right\vert^{2-n} \ua\left(x\right)^{2^\star-1}\, dx \\
&&\quad = O\left(\sum_{i=1}^\kappa \int_{\Om} \left\vert \xa-x\right\vert^{2-n} \Psiie\left(x\right)^{2^\star-1}\, dx\right)\\
&&\qquad + O\left(\nukappa^{\frac{n+2}{2}\left(1-2\gamma\right)}\int_{\left\{x\in \Om,\, \petitrkappauneps\left(x\right)\geq R_0\nukappa\right\}}
\left\vert \xa-x\right\vert^{2-n} \petitrkappauneps\left(x\right)^{-(n+2)(1-\gamma)}\, dx \right)\\
&&\qquad + O\left(\max\left\{\bue^{2^\star-1}\,;\, \ma^{\frac{n+2}{2}\left(1-2\gamma\right)}\right\}\int_{\Om} \left\vert \xa-x\right\vert^{2-n} \ra(x)^{-\left(n+2\right)\gamma}\, dx\right)\\
&&\qquad + O\left( \nukappa^{\frac{n-2}{2}}\Rkappauneps\left(\xa\right)^{2-n} \right)\hskip.1cm.
\fincal
Here all the integrals have a meaning since $\gamma<\frac{2}{n+2}$. We write that
\begincal
&&\int_{\left\{x\in \Om,\, \petitrkappauneps\left(x\right)\geq \nukappa\right\}}
\left\vert \xa-x\right\vert^{2-n} \petitrkappauneps\left(x\right)^{-(n+2)(1-\gamma)}\, dx\\
&& \quad \le \sum_{i=\kappa+1}^N \int_{\left\{x\in \Om,\, \left\vert \xai-x\right\vert \geq \nukappa\right\}} \left\vert \xa-x\right\vert^{2-n} \left\vert \xai-x\right\vert^{-(n+2)(1-\gamma)}\, dx\\
&&\quad = O\left(\sum_{i=\kappa+1}^N \nukappa^{n-\left(n+2\right)\left(1-\gamma\right)}\left\vert \xa-\xai\right\vert^{2-n}\right)\\
&&\quad = O\left(\nukappa^{n-\left(n+2\right)\left(1-\gamma\right)} \Rkappauneps\left(\xa\right)^{2-n}\right)
\fincal
since $\gamma<\frac{2}{n+2}$ and thanks to (\ref{eqD.3.4}). We can also write since $\gamma<\frac{2}{n+2}$ that
$$\int_{\Om} \left\vert \xa-x\right\vert^{2-n} \ra(x)^{-\left(n+2\right)\gamma}\, dx = O(1)$$
and that
$$\bue^{2^\star-1}+\ma^{\frac{n+2}{2}\left(1-2\gamma\right)} =o\left(\bue\right)+ o\left(\ma^{\frac{n-2}{2}}\right) = o\left(\bue\right)$$
thanks to (\ref{eq4}) and (\ref{eq2.4bis}). Collecting these estimates, we arrive to
\begincal
&&\int_{\Om} \left\vert \xa-x\right\vert^{2-n} \ua\left(x\right)^{2^\star-1}\, dx \\
&&\quad = O\left(\sum_{i=1}^\kappa \int_{\Om} \left\vert \xa-x\right\vert^{2-n} \Psiie\left(x\right)^{2^\star-1}\, dx\right)\\
&&\qquad + o\left(\bue\right) +O\left( \nukappa^{\frac{n-2}{2}}\Rkappauneps\left(\xa\right)^{2-n} \right)\hskip.1cm.
\fincal
Since $\gamma<\frac{2}{n+2}$, we get that (see Step \ref{lemmaprojection}.2 in the proof of Proposition \ref{lemmaprojection} in Appendix B for the details)
$$\int_{\Om} \left\vert \xa-x\right\vert^{2-n} \Psiie\left(x\right)^{2^\star-1}\, dx = O\bigl(\Uie\left(\xa\right)\bigr)$$
for all $1\le i\le \kappa$. Thus we have obtained that
\begincal
&&\int_{\Om} \left\vert \xa-x\right\vert^{2-n} \ua\left(x\right)^{2^\star-1}\, dx \\
&&\quad = O\left(\sum_{i=1}^k\Uie\left(\xa\right)\right)+ o\left(\bue\right) +O\left( \nukappa^{\frac{n-2}{2}}\Rkappauneps\left(\xa\right)^{2-n} \right)\hskip.1cm.
\fincal
Coming back to (\ref{eqD.3.5}) with this last estimate, we obtain that (\ref{eqD.3.1}) holds. This ends the proof of Step \ref{prop:claimD}.3. \hfill $\Box$

\medskip\noindent{\bf Step \ref{prop:claimD}.4:} We claim that there exists $A_5>0$ such that for any sequence $\left(\xa\right)$ of points in $\bar{\Om}$, we have that
\begin{eqnarray}\label{est:claimD4}
&&\left\vert \ua\left(\xa\right)-\bue - \sum_{i=1}^\kappa \Vai\left(\xa\right)\right\vert \\
&&\quad \le A_5 \nukappa^{\frac{n-2}{2}}\Rkappauneps\left(\xa\right)^{2-n} +o\left(\bue\right) + o\left(\sum_{i=1}^\kappa \Uie\left(\xa\right)\right)\hskip.1cm.\nonumber
\end{eqnarray}

\medskip\noindent {\it Proof of Step \ref{prop:claimD}.4:} Let $\left(\xa\right)$ be a sequence of points in $\bar{\Om}$.

\smallskip\noindent{\bf Step \ref{prop:claimD}.4.1:} Assume first that
\begin{equation}\label{eqD.4.1}
\Rkappauneps\left(\xa\right)=O\left(\nukappa\right)\hbox{ when }\alpha\to +\infty\hbox{ and }\Rkappauneps\left(\xa\right)=\raa\left(\xa\right)\hskip.1cm.
\end{equation}
We can apply Proposition \ref{prop:claimB}, (iv), to write that
$$\Rkappauneps\left(\xa\right)^{\frac{n}{2}-1} \left\vert \ua\left(\xa\right)-\bue-\sum_{i=1}^N \Uie\left(\xa\right)\right\vert = o(1)\hskip.1cm.$$
This leads to
$$\left\vert \ua\left(\xa\right)-\bue - \sum_{i=1}^\kappa \Uie\left(\xa\right)\right\vert \le
\sum_{i=\kappa+1}^N \Uie\left(\xa\right) + o\left(\Rkappauneps\left(\xa\right)^{1-\frac{n}{2}}\right)\hskip.1cm.$$
Noting that, for any $\kappa+1\le i\le N$,
\begincal
\Uie\left(\xa\right) &=& \mai^{\frac{n-2}{2}} \left(\left\vert \xai-\xa\right\vert^2 +\mai^2\right)^{1-\frac{n}{2}}\\
&\le &  \mai^{\frac{n-2}{2}}\Rkappauneps\left(\xa\right)^{2-n} \\
&\le & \mukappaun^{\frac{n-2}{2}}\Rkappauneps\left(\xa\right)^{2-n} \\
&\le & \nukappa^{\frac{n-2}{2}}\Rkappauneps\left(\xa\right)^{2-n}
\fincal
thanks to (\ref{eqD.10}), we then get that
$$\left\vert \ua\left(\xa\right)-\bue - \sum_{i=1}^\kappa \Uie\left(\xa\right)\right\vert \le
N \nukappa^{\frac{n-2}{2}}\Rkappauneps\left(\xa\right)^{2-n}+ o\left(\Rkappauneps\left(\xa\right)^{1-\frac{n}{2}}\right)\hskip.1cm.$$
Thanks to (\ref{eqD.4.1}), we also know that
$$\Rkappauneps\left(\xa\right)^{1-\frac{n}{2}}=O\left(\nukappa^{\frac{n-2}{2}}\Rkappauneps\left(\xa\right)^{2-n}\right)\hskip.1cm.$$
We then get that
\begin{equation}\label{ineq:diff:1}
\left\vert \ua\left(\xa\right)-\bue - \sum_{i=1}^\kappa \Vai\left(\xa\right)\right\vert \le
C \nukappa^{\frac{n-2}{2}}\Rkappauneps\left(\xa\right)^{2-n}+\sum_{i=1}^\kappa |\Uie\left(\xa\right)-\Vai(\xa)|\hskip.1cm.
\end{equation}
We are left with estimating $|\Uie\left(\xa\right)-\Vai(\xa)|$ when $\epstozero$ for all $i\in \{1,...,\kappa\}$. We use the estimates of Proposition \ref{lemmaprojection} and we let $i\in\{1,...,\kappa\}$. We have that
\begin{eqnarray*}
\Uie\left(\xa\right)-\Vai(\xa)&=&O(\Uie(\xa))=O\left(\left(\frac{\mai}{\mai^2+|\xa-\xai|^2}\right)^{\frac{n-2}{2}}\right)\\
&=& O\left(\min \left\{\frac{\mai^{\frac{n-2}{2}}}{\raa(\xa)^{n-2}},\mai^{-\frac{n-2}{2}}\right\}\right)\\
&=& O\left(\min \left\{\frac{\mai^{\frac{n-2}{2}}}{\nukappa^{\frac{n-2}{2}}},\frac{\raa(\xa)^{n-2}}{\nukappa^{\frac{n-2}{2}}\mai^{\frac{n-2}{2}}}\right\}\frac{\nukappa^{\frac{n-2}{2}}}{\raa(\xa)^{n-2}}\right)\\
\end{eqnarray*}
Using \eqref{eqD.4.1} and $\raa(\xa)=\Rkappauneps(\xa)$, we get that
\begin{eqnarray}
\Uie\left(\xa\right)-\Vai(\xa)&=&O\left(\min \left\{\left(\frac{\mai}{\nukappa}\right)^{\frac{n-2}{2}},\left(\frac{\nukappa}{\mai}\right)^{\frac{n-2}{2}}\right\} \frac{\nukappa^{\frac{n-2}{2}}}{\Rkappauneps(\xa)^{n-2}}\right)\nonumber\\
\Uie\left(\xa\right)-\Vai(\xa)&=&O\left(\frac{\nukappa^{\frac{n-2}{2}}}{\Rkappauneps(\xa)^{n-2}}\right)\label{ineq:diff:2}
\end{eqnarray}
Plugging \eqref{ineq:diff:2} into \eqref{ineq:diff:1} yields \eqref{est:claimD4} up to take $A_5$ large enough if \eqref{eqD.4.1} holds. This ends Step \ref{prop:claimD}.4.1.

\medskip\noindent{\bf Step \ref{prop:claimD}.4.2:} Assume now that
\begin{equation}\label{eqD.4.1:bis}
\Rkappauneps\left(\xa\right)=O\left(\nukappa\right)\hbox{ when }\alpha\to +\infty\hbox{ and }\raa\left(\xa\right)<\Rkappauneps\left(\xa\right)\hskip.1cm.
\end{equation}
Then there exists $1\le i\le \kappa$ such that
$$\left\vert \xai-\xa\right\vert^2 + \mai^2 \le \Rkappauneps\left(\xa\right)^2 = O\left(\mukappa^2\right)$$
thanks to Step \ref{prop:claimD}.1. This implies that $\mai= O\left(\mukappa\right)$ and that $\left\vert \xai-\xa\right\vert=O\left(\mukappa\right)$ when $\alpha\to +\infty$. This also implies that $\Rkappauneps\left(\xa\right)\geq \mai$. Since we have that $\mukappa\leq \mai$, using Proposition \ref{prop:claimB}, (ii) and (iii), we then obtain that
$$\left\vert \ua\left(\xa\right) - \Uie\left(\xa\right)\right\vert = o\bigl(\Uie\left(\xa\right)\bigr)$$
and that
\begin{equation}\label{ineq:Ui}
\mai^{1-\frac{n}{2}} = O\bigl(\Uie\left(\xa\right)\bigr)\hskip.1cm.
\end{equation}
This leads in particular to
$$\left\vert \ua\left(\xa\right)-\bue - \sum_{j=1}^\kappa \Uje\left(\xa\right)\right\vert
= o\bigl(\Uie\left(\xa\right)\bigr) + O\left(\sum_{1\le j\le \kappa,\, j\neq i}\Uje\left(\xa\right)\right)\hskip.1cm.$$
Now, for any $1\le j\le \kappa$, $j\neq i$, we have that
\begin{eqnarray}
\mai^{\frac{n-2}{2}}\Uje\left(\xa\right) &=&\left(\mai\maj\right)^{\frac{n-2}{2}} \left(\left\vert \xaj-\xa\right\vert^2 +\maj^2\right)^{1-\frac{n}{2}}\nonumber\\
&\le & \left(\frac{\left\vert \xaj-\xa\right\vert^2}{\mai\maj} +\frac{\maj}{\mai}\right)^{1-\frac{n}{2}}=o(1)\label{ineq:added:2}
\end{eqnarray}
thanks to Proposition \ref{prop:claimB}, (ii), since $\mai = O\left(\mukappa\right)$ and $\mukappa\le \maj$.  In particular, \eqref{ineq:Ui} and \eqref{ineq:added:2} yield
\begin{equation}\label{ineq:Uj:Ui}
\Uje(\xa)=o(\Uie(\xa))
\end{equation}
when $\epstozero$ for all $1\le j\le \kappa$, $j\neq i$. Thus we arrive in this case to
\begin{equation}\label{est:diff:1}
\left\vert \ua\left(\xa\right)-\bue - \sum_{j=1}^\kappa \Uje\left(\xa\right)\right\vert
= o\bigl(\Uie\left(\xa\right)\bigr)\hskip.1cm.
\end{equation}
To obtain \eqref{est:claimD4}, we need to remark that, thanks to Proposition \ref{lemmaprojection} and \eqref{ineq:Uj:Ui}, we have that
\begin{equation}\label{eq:diff}
\Uje(\xa)-\Vaj(\xa)=O(\Uje(\xa))=o(\Uie(\xa))
\end{equation}
when $\epstozero$ for all $1\le j\le \kappa$, $j\neq i$. Concerning $\Uie(\xa)$, we refer again to Proposition \ref{lemmaprojection}: if $\xai$ is such that case (i) or (ii) holds, then $\Uie(\xa)-\Vai(\xa)=o(\Uie(\xa))$ when $\epstozero$. In case (iii) of Proposition \ref{lemmaprojection}, we get with \eqref{ineq:Ui} that
\begin{eqnarray*}
\Uie(\xa)-\Vai(\xa)&=&O\left(\left(\frac{\mai}{\mai^2+d(\xai,\partial\Omega)^2}\right)^{\frac{n-2}{2}}\right)+o(\Uie(\xa))+O(\mai^{\frac{n-2}{2}})\\
&=&o(\mai^{-\frac{n-2}{2}})+o(\Uie(\xa))=o(\Uie(\xa))
\end{eqnarray*}
when $\epstozero$. Therefore \eqref{eq:diff} holds for all $j\in\{1,...,\kappa\}$: associating this equation with \eqref{est:diff:1} yields \eqref{est:claimD4} for any choice of $A_5>0$ if \eqref{eqD.4.1:bis} holds. This ends Step \ref{prop:claimD}.4.2.

\medskip\noindent{\bf Step \ref{prop:claimD}.4.3:} From now on, we assume that
\begin{equation}\label{eqD.4.3}
\frac{\Rkappauneps\left(\xa\right)}{\nukappa}\to +\infty \hbox{ as }\epstozero\hskip.1cm.
\end{equation}
As a preliminary remark, let us note that
\begin{equation}\label{eq:r:R}
r_{\kappa+1,\alpha}(\alpha)\asymp R_{\kappa+1,\alpha}(\xa)
\end{equation}
for large $\alpha$'s (the argument goes by contradiction). We use Green's representation formula and (\ref{eq2.4bis}) to write that
\begin{equation}\label{eqD.4.4}
\begin{split}
&\left\vert \ua\left(\xa\right) - \bue -\sum_{i=1}^\kappa \Vai\left(\xa\right) \right\vert \\
&\quad \le n(n-2)C_0 \int_{\Om} \left\vert \xa-x\right\vert^{2-n} \left\vert \ua\left(x\right)^{2^\star-1}-\sum_{i=1}^\kappa \Uie\left(x\right)^{2^\star-1}\right\vert\, dx \\
&\qquad + C_0\ea \int_{\Om} \left\vert \xa-x\right\vert^{2-n} \ua\left(x\right)\, dx + o\left(\bue\right)\hskip.1cm.
\end{split}
\end{equation}
Let us write thanks to \eqref{est:claimD3} that
\begincal
&&\int_{\Om} \left\vert \xa-x\right\vert^{2-n} \ua\left(x\right)\, dx \leq A_4\sum_{i=1}^\kappa \mai^{\frac{n-2}{2}}\int_{\Om} \left\vert \xa-x\right\vert^{2-n} \left(\left\vert \xai-x\right\vert^2+\mai^2\right)^{1-\frac{n}{2}}\, dx\\
&&+ A_4 \bue \int_{\Om}\left\vert \xa-x\right\vert^{2-n}\, dx+ A_4  \nukappa^{\frac{n-2}{2}} \sum_{i=\kappa+1}^N\int_{\Om} \left\vert \xa-x\right\vert^{2-n}\left(\left\vert \xai-x\right\vert^2 +\mai^2\right)^{1-\frac{n}{2}} \, dx\\
&&= O\left(\sum_{i=1}^\kappa \mai^{\frac{n-2}{2}}\left(\left\vert\xai-\xa\right\vert^2+\mai^2\right)^{1-\frac{n}{2}} \right) + O\left(\bue\right)\\
&&+ O\left(\nukappa^{\frac{n-2}{2}} \sum_{i=\kappa+1}^N \left(\left\vert\xai-\xa\right\vert^2+\mai^2\right)^{1-\frac{n}{2}} \right)\\
&&= O\left(\sum_{i=1}^\kappa \Uie\left(\xa\right)\right) + O\left(\bue\right) + O\left(\nukappa^{\frac{n-2}{2}}\Rkappauneps\left(\xa\right)^{2-n}\right)\hskip.1cm.
\fincal
Thus (\ref{eqD.4.4}) becomes
\begin{equation}\label{eqD.4.5}
\begin{split}
&\left\vert \ua\left(\xa\right) - \bue -\sum_{i=1}^\kappa \Vai\left(\xa\right) \right\vert \\
&\quad \le n(n-2)C_0 \int_{\Om} \left\vert \xa-x\right\vert^{2-n} \left\vert \ua\left(x\right)^{2^\star-1}-\sum_{i=1}^\kappa \Uie\left(x\right)^{2^\star-1}\right\vert\, dx \\
&\qquad + o\left(\bue\right) + o\left(\sum_{i=1}^\kappa \Uie\left(\xa\right)\right)+o\left(\nukappa^{\frac{n-2}{2}}\Rkappauneps\left(\xa\right)^{2-n}\right)\hskip.1cm.
\end{split}
\end{equation}
Thanks to Proposition \ref{prop:claimB}, (ii) and (iii), there exists a sequence $\Leps\to +\infty$ as $\epstozero$ such that, for any $1\le i\le \kappa$,
$$\left\Vert \frac{\ua-\Uie}{\Uie}\right\Vert_{L^\infty\left(\Omega_{i,\alpha}\cap \Om\right)} \to 0\hbox{ as }\epstozero$$
and
$$\left\Vert \frac{\sum_{1\le j\le \kappa,\, j\neq i} \Uje}{\Uie}\right\Vert_{L^\infty\left(\Omega_{i,\alpha}\cap \Om\right)} \to 0\hbox{ as }\epstozero$$
where
$$\Omega_{i,\alpha} := B_{\Leps\mai}(\xai) \setminus \bigcup_{i+1\le j \le N} B_{\frac{1}{\Leps}\mai}(\xaj)$$
and such that these sets are disjoint for $\alpha$ large enough. Then we can write that
\begincal
&&\int_{\Om \cap \Omega_{i,\alpha}}  \left\vert \xa-x\right\vert^{2-n} \left\vert \ua\left(x\right)^{2^\star-1}-\sum_{j=1}^\kappa \Uje\left(x\right)^{2^\star-1}\right\vert\, dx \\
&&\quad = o\left(\int_{\Om \cap \Omai}  \left\vert \xa-x\right\vert^{2-n} \Uie(x)^{2^\star-1}\, dx\right)\\
&&\quad = o\bigl(\Uie\left(\xa\right)\bigr)
\fincal
for all $1\le i\le \kappa$. We also remark that
$$\int_{\Om\setminus \Omega_{i,\alpha}} \left\vert \xa-x\right\vert^{2-n}\Uie(x)^{2^\star-1}\, dx =  o\bigl(\Uie\left(\xa\right)\bigr)$$
for all $1\le i\le \kappa$. Thus, using \eqref{est:claimD3}, we transform (\ref{eqD.4.5}) into
\begin{equation}\label{eqD.4.6}
\begin{split}
&\left\vert \ua\left(\xa\right) - \bue -\sum_{i=1}^\kappa \Vai\left(\xa\right) \right\vert \\
&\quad \le C\nukappa^{\frac{n+2}{2}} \int_{\Om\cap\{\petitrkappauneps(x)\geq \nukappa\}} \left\vert \xa-x\right\vert^{2-n} \Rkappauneps\left(x\right)^{-(n+2)}\, dx \\
&\qquad + O\left(\int_{\Om\cap\{\petitrkappauneps(x)< \nukappa\}} \left\vert \xa-x\right\vert^{2-n}\ua(x)^{2^\star-1}\, dx\right)\\
&\qquad+ o\left(\bue\right) + o\left(\sum_{i=1}^\kappa \Uie\left(\xa\right)\right)+o\left(\nukappa^{\frac{n-2}{2}}\Rkappauneps\left(\xa\right)^{2-n}\right)\hskip.1cm.
\end{split}
\end{equation}
Following the proof of Step \ref{prop:claimD}.3, it remains to notice that
\begincal
&&\int_{\Om\cap\{\petitrkappauneps(x)\geq \nukappa\}} \left\vert \xa-x\right\vert^{2-n} \Rkappauneps\left(x\right)^{-(n+2)}\, dx \\
&&\quad \le \sum_{i=\kappa+1}^N \int_{B_R(\xai)\cap\{|x-\xai|\geq \nukappa\}} \left\vert \xa-x\right\vert^{2-n}\left(\left\vert \xai-x\right\vert^2+\mai^2\right)^{-1-\frac{n}{2}}\, dx
\fincal
Fix $i\geq \kappa+1$. Assume first that
$$\lim_{\epstozero}\frac{|\xai-\xa|}{\mai}=+\infty.$$
Then, we get with changes of variables that
\begin{eqnarray*}
&&\int_{B_R(\xai)\cap\{|x-\xai|\geq \nukappa\}} \left\vert \xa-x\right\vert^{2-n}\left(\left\vert \xai-x\right\vert^2+\mai^2\right)^{-1-\frac{n}{2}}\, dx\\
&& \leq  \int_{B_R(\xai)\cap\{|x-\xai|\geq \nukappa\}} \left\vert \xa-x\right\vert^{2-n}\left\vert \xai-x\right\vert^{-(n+2)}\, dx\\
&&= |\xai-\xa|^{2-n}\nukappa^{-2}\int_{1< |z|<\frac{R}{\nukappa}}|z|^{-2-n}\left|\frac{\xai-\xa}{|\xai-\xa|}+\frac{\nukappa}{|\xai-\xa|}z\right|^{2-n}\, dz\\
&&=O(|\xai-\xa|^{2-n}\nukappa^{-2})=O\left(\nukappa^{-2}\left(|\xai-\xa|^2+\mai^2\right)^{1-\frac{n}{2}}\right)
\end{eqnarray*}
when $\epstozero$. Assume now that
\begin{equation}\label{cas:2}
|\xai-\xa|=O(\mai)
\end{equation}
when $\epstozero$. With the change of variables $x:=\xai+\mai z$, we get that
$$\int_{B_R(\xai)\cap\{|x-\xai|\geq \nukappa\}} \left\vert \xa-x\right\vert^{2-n}\left(\left\vert \xai-x\right\vert^2+\mai^2\right)^{-1-\frac{n}{2}}\, dx=O(\mai^{-n})$$
when $\epstozero$. It follows from \eqref{cas:2} that $R_{\kappa+1,\alpha}(\xa)=O(\mai)$, and then, with \eqref{eqD.4.3}, we get that $\nukappa=o(\mai)$ and then
\begin{eqnarray*}
&&\int_{B_R(\xai)\cap\{|x-\xai|\geq \nukappa\}} \left\vert \xa-x\right\vert^{2-n}\left(\left\vert \xai-x\right\vert^2+\mai^2\right)^{-1-\frac{n}{2}}\, dx\\
&&=O(\mai^{2-n}\nukappa^{-2})=O\left(\nukappa^{-2}\left(|\xai-\xa|^2+\mai^2\right)^{1-\frac{n}{2}}\right).
\end{eqnarray*}
In all the cases, we have then proved that
\begin{eqnarray}
&&\int_{B_R(\xai)\cap\{|x-\xai|\geq \nukappa\}} \left\vert \xa-x\right\vert^{2-n}\left(\left\vert \xai-x\right\vert^2+\mai^2\right)^{-1-\frac{n}{2}}\, dx\label{eq:add:1}\\
&&=O\left(\nukappa^{-2}\left(|\xai-\xa|^2+\mai^2\right)^{1-\frac{n}{2}}\right)\nonumber
\end{eqnarray}
when $\epstozero$ for all $i\geq \kappa+1$.

\smallskip\noindent independently, using H\"older's inequality and \eqref{eq:r:R}, we have that
\begin{equation}\label{eq:add:2}
\int_{\Om\cap\{\petitrkappauneps(x)<\nukappa\}} \left\vert \xa-x\right\vert^{2-n}\ua(x)^{2^\star-1}\, dx=O(\nukappa^{\frac{n-2}{2}}\petitrkappauneps(\xa)^{2-n})
\end{equation}
when $\epstozero$.

\smallskip\noindent Plugging \eqref{eq:add:1} and \eqref{eq:add:2} into (\ref{eqD.4.6}), we get that \eqref{est:claimD4} holds up to take $A_5$ large enough if \eqref{eqD.4.3} holds. This ends Step \ref{prop:claimD}.4.3.

\medskip\noindent Plugging together Steps  \ref{prop:claimD}.4.1 to  \ref{prop:claimD}.4.3, we get that \eqref{est:claimD4} holds up to taking $A_5$ large enough. This ends the proof of Step \ref{prop:claimD}.4.\hfill$\Box$

\medskip\noindent{\bf Step \ref{prop:claimD}.5:} We claim that $\nukappa=O\left(\mukappaun\right)$ when $\alpha\to +\infty$.

\medskip\noindent{\it Proof of Step \ref{prop:claimD}.5:} We proceed by contradiction and thus assume that, up to a subsequence, there exists $1\le i\le \kappa$ and $\xa\in \tomegaie$ such that (see the definition \eqref{eqD.10})
\begin{equation}\label{eqD.5.1}
\nukappa^{\frac{n-2}{2}\left(1-2\gamma\right)}\Thetaeps\left(\xa\right) = \Psiie\left(\xa\right)\hskip.1cm.
\end{equation}
Since $\xa\in \tomegaie$, we also have that
\begin{equation}\label{eqD.5.2}
\left\vert \xai-\xa\right\vert^2 \left\vert \ua\left(\xa\right)-\bue -\sum_{j=1}^\kappa \Vaj\left(\xa\right)\right\vert^{2^\star-2} \geq A_2\hskip.1cm.
\end{equation}
At last, since $\xa\in \Om_{i,\alpha}$, we have that
\begin{equation}\label{eqD.5.3}
\Psije\left(\xa\right)\le \Psiie\left(\xa\right)
\end{equation}
for all $1\le j\le \kappa$. In particular, we can write thanks to (\ref{eqD.5bis}) that
\begincal
&&\left\vert \xai-\xa\right\vert^2 \Uje\left(\xa\right)^{2^\star-2}\le C A_1^{2^\star-2} \left\vert \xai-\xa\right\vert^2 \Psiie\left(\xa\right)^{2^\star-2}\\
&&\leq C A_1^{2\left(2^\star-2\right)} \left\vert \xai-\xa\right\vert^2 \Uie\left(\xa\right)^{2^\star-2}
\left(\frac{\left\vert \xai-\xa\right\vert}{\mai}+\frac{\mai}{\left\vert \xai-\xa\right\vert}\right)^{4\gamma}\\
&&\leq C A_1^{2\left(2^\star-2\right)}\left(\frac{\left\vert \xai-\xa\right\vert}{\mai}+\frac{\mai}{\left\vert \xai-\xa\right\vert}\right)^{4\gamma-2}\leq C A_1^{2\left(2^\star-2\right)}
\fincal
for all $1\le j\le \kappa$ since $\gamma<\frac{1}{2}$. Applying \eqref{est:claimD4} to the sequence $\left(\xa\right)$ and coming back to (\ref{eqD.5.2}), we thus obtain that
$$A_2 \le A_5^{2^\star-2} \left\vert \xai-\xa\right\vert^2 \nukappa^2 \Rkappauneps\left(\xa\right)^{-4} + o(1)\hskip.1cm.$$
This leads to
\begin{equation}\label{eqD.5.4}
\Rkappauneps\left(\xa\right)^2 = O\left(\nukappa\left\vert \xai-\xa\right\vert\right)\hskip.1cm.
\end{equation}
Using \eqref{estAB}, we can write thanks to (\ref{eqD.5.1}) that
\begin{equation}\label{eqD.5.5}
\nukappa^{1-2\gamma} = O\left(\Psiie\left(\xa\right)^{\frac{2}{n-2}} \Rkappauneps\left(\xa\right)^{2(1-\gamma)}\right)
\end{equation}
which leads with (\ref{eqD.5.4}) to
$$\nukappa^{-\gamma} = O\left(\Psiie\left(\xa\right)^{\frac{2}{n-2}}\left\vert \xai-\xa\right\vert^{1-\gamma}\right)\hskip.1cm.$$
It is easily checked thanks to (\ref{eqD.5bis}) that this leads to $\left\vert \xai-\xa\right\vert = O\left(\nukappa\right)$. Since $\nukappa=O\left(\mai\right)$ thanks to Step \ref{prop:claimD}.1 and since $\xai-\xa\neq o(\mai)$ when $\epstozero$, this leads in turn to
$$\mai= O\left(\left\vert \xai-\xa\right\vert^{1-\gamma}\nukappa^\gamma\right)=O\left(\nukappa\right)=O\left(\mukappa\right)\hskip.1cm.$$
Thanks to (\ref{eqD.5.5}), we have obtained so far that $\left\vert \xai-\xa\right\vert = O\left(\mai\right)$, that $\mai=O\left(\mukappa\right)$ and at last that
$\mai= O\left(\Rkappauneps\left(\xa\right)\right)$ using again \eqref{eqD.5.5}. Note that since $\mai\asymp\mukappa$, we have that $\mai=O(\maj)$ for $j\leq\kappa$ when $\alpha\to +\infty$. Using Proposition \ref{prop:claimB}, (ii) and (iii), we then get that
$$\left\vert \xai-\xa\right\vert^2 \left\vert \ua\left(\xa\right) - \bue-\sum_{j=1}^\kappa \Vaj\left(\xa\right)\right\vert^{2^\star-2}\to 0\hbox{ as }\epstozero\hskip.1cm,$$
thus contradicting (\ref{eqD.5.2})  This ends the proof of Step \ref{prop:claimD}.5. \hfill $\Box$

\medskip\noindent Steps \ref{prop:claimD}.4 and \ref{prop:claimD}.5 give that, if $\left({\mathcal I}_\kappa\right)$ holds for some $1\le \kappa\le N-1$, then $\left({\mathcal I}_{\kappa+1}\right)$ holds. Since we know that $\left({\mathcal I}_1\right)$ holds thanks to Proposition \ref{prop:claimC}, we have proved that $\left({\mathcal I}_N\right)$ holds and thus we have proved Proposition \ref{prop:claimD}.\hfill$\Box$

\section{Asymptotic estimates in $C^1\left(\Omega\right)$}\label{sec:asymp:c1}

In this section, we prove the following:

\begin{prop}\label{prop:claimE} There exists a sequence $\betae$ as $\epstozero$ such that
\begin{equation}\label{est:pointwise:V}
\left\vert \ua-\bue - \sum_{i=1}^N \Vai\right\vert \le \betae \left(\bue + \sum_{i=1}^N \Uie\right)
\end{equation}
for all $x\in \bar{\Om}$ and all $\epspositif$. In addition, there exists $A_6>0$ such that
\begin{equation}\label{est:grad}
\left\vert \nabla \ua(x)\right\vert \le o(\bua)+A_6 \sum_{i=1}^N \mai^{\frac{n-2}{2}} \left(\mai^2 +\left\vert \xai-x\right\vert^2\right)^{-\frac{n-1}{2}}
\end{equation}
for all $x\in \bar{\Om}$ and all $\epspositif$.
\end{prop}
\medskip\noindent{\it Proof of Proposition \ref{prop:claimE}:} We first prove the pointwise estimate on $\ua$. Then we will prove the pointwise estimate in $C^1\left(\bar{\Om}\right)$.

\medskip\noindent{\bf Step \ref{prop:claimE}.1:} We claim that there exists a sequence $\betae\to 0$ as $\epstozero$ such that \eqref{est:pointwise:V} holds. In particular, there exists $C>0$ such that
\begin{equation}\label{est:pointwise}
\ua(x)\leq C\left(\bua+\sum_{i=1}^N\left(\frac{\mai^2}{\mai^2+|x-\xai|^2}\right)^{\frac{n-2}{2}}\right)\end{equation}
for all $x\in\Omega$ and for all $\alpha\in\nn$.

\smallskip\noindent{\it Proof of Step \ref{prop:claimE}.1:} The proof of \eqref{est:pointwise:V} goes as in Step \ref{prop:claimD}.4. We omit the details. The estimate \eqref{est:pointwise} is a consequence of \eqref{est:pointwise:V} and the inequality \eqref{ineq:Ua:Va} of Proposition \ref{lemmaprojection}.

\medskip\noindent{\bf Step \ref{prop:claimE}.2:} We claim that \eqref{est:grad} holds.

\smallskip\noindent{\it Proof of Step \ref{prop:claimE}.2:} Green's representation formula yields
$$\ua(x)-\bua=\int_\Omega G(x,\cdot)\left(\ua^{\crit-1}-\ea\ua\right)\, dy$$
for all $x\in\Omega$ and all $\alpha\in\nn$. Differentiation with respect to $x$ yields
\begin{eqnarray*}
|\nabla\ua(x)|&\leq &\left|\int_\Omega \nabla_xG(x,\cdot)\left(\ua^{\crit-1}-\ea\ua\right)\, dy\right|\\
&\leq &\int_\Omega |\nabla_xG(x,\cdot)|\left(\ua^{\crit-1}+\ea\ua\right)\, dy
\end{eqnarray*}
Plugging \eqref{est:pointwise} and the estimate of $\nabla_x G$ of in \eqref{app:est:der:green} of Proposition \ref{app:prop:upp:bnd} yield \eqref{est:grad}: we omit the details.

\smallskip\noindent These two steps prove Proposition \ref{prop:claimE}.\hfill$\Box$

\section{Convergence to singular harmonic functions}\label{sec:cv:sing}
\subsection{Convergence at general scale}
We prove the following general convergence result
\begin{prop}\label{th:cv:ra} Let $(\ua)_{\alpha\in\nn}\in C^2(\omegabar)$ such that \eqref{eq:ua} and \eqref{bnd:nrj} hold. Let $(x_{1, \alpha}),...,(x_{N, \alpha})\in\omegabar$ and  $(\mu_{1, \alpha}),...,(\mu_{N, \alpha})\in (0,+\infty)$ such that \eqref{est:pointwise:V} holds. Let $(\xa)_{\alpha\in\nn}\in\omegabar$ and $(\ma)_{\alpha\in\nn}, (\ra)_{\alpha\in\nn}\in (0,+\infty)$ be sequences such that

\smallskip(i) $\lim_{\alpha\to +\infty}\ra=0$ and $\ma=o(\ra)$ when $\alpha\to +\infty$,

\smallskip(ii) $\ra\not\asymp\mai$ when $\alpha\to +\infty$ for all $i\in\{1,...,N\}$ such that $\xai-\xa=O(\ra)$ when $\alpha\to +\infty$,

\smallskip(iii) $\ra^{n-2}\bua=O(\ma^{\frac{n-2}{2}})$ when $\alpha\to +\infty$,

\smallskip(iv) $\mai=O(\ma)$ when $\alpha\to +\infty$ for all $i\in I$ where
$$I:=\{i\in \{1,...,N\}/ \xai-\xa=O(\ra)\hbox{ and }\mai=o(\ra)\hbox{ when }\alpha\to +\infty\},$$

\smallskip(v) $\ra^2\mai=O(\ma (\mai^2+|\xa-\xai|^2))$ when $\alpha\to +\infty$ for all $i\in I^c$.

\medskip\noindent Then we distinguish two cases:

\smallskip\noindent$\bullet$ {\bf Case \ref{th:cv:ra}.1:} Assume that
$$\lim_{\alpha\to +\infty}\frac{d(\xa,\partial\Omega)}{\ra}=+\infty.$$
We define
\begin{equation}\label{def:va}
\va(x):=\frac{\ra^{n-2}}{\ma^{\frac{n-2}{2}}}\ua(\xa+\ra x)\hbox{ for all }x\in \frac{\Omega-\xa}{\ra}.
\end{equation}
Then,
\begin{equation}\label{lim:va}
\lim_{\alpha\to +\infty}\va(x)=K+\sum_{i\in I'}\lambda_i|x-\theta_i|^{2-n}\hbox{ in }C^2_{loc}(\rn\setminus\{\theta_i/\, i\in I\})
\end{equation}
where
\begin{equation}\label{def:Iprime}
I':=\{i\in I/\, \mai\asymp\ma\}
\end{equation}
and
\begin{equation}\label{def:thetai:th}
\theta_i:=\lim_{\alpha\to +\infty}\frac{\xai-\xa}{\ra}\hbox{ for all }i\in I\hbox{ and }\lambda_i>0\hbox{ for all }i\in I'.
\end{equation}
And
\begin{equation}\label{lower:K}
K= \lim_{\alpha\to +\infty}\frac{\ra^{n-2}\bua}{\ma^{\frac{n-2}{2}}}+\sum_{i\in I^c}\tau_i\lim_{\alpha\to +\infty}\left(\frac{\ra^2\mai}{\ma(\mai^2+|\xa-\xai|^2)}\right)^{\frac{n-2}{2}}
\end{equation}
with $\tau_i\geq 1$ for all $i\in I^c$

\medskip\noindent$\bullet$ {\bf Case \ref{th:cv:ra}.2:} Assume that
$$\lim_{\alpha\to +\infty}\frac{d(\xa,\partial\Omega)}{\ra}=\rho\in [0,+\infty).$$Then there exists $x_0\in\partial\Omega$ such that $\lim_{\alpha\to +\infty}\xa=x_0$. We take $\varphi$, $U_{x_0}$ and the extension $\tua$ as in Lemma \ref{lem:ext}. We define
\begin{equation}\label{def:tva}
\tva(x):=\frac{\ra^{n-2}}{\ma^{\frac{n-2}{2}}}\tua\circ\varphi(\varphi^{-1}(\xa)+\ra x)\hbox{ for all }x\in\frac{\varphi^{-1}(U_{x_0})-\varphi^{-1}(\xa)}{\ra}.
\end{equation}
Then,
\begin{equation}\label{val:lim:tva}
\lim_{\alpha\to +\infty}\tva(x)=K+\sum_{i\in I'}\lambda_i(|x-\tilde{\theta}_i|^{2-n}+|x-\sigma(\tilde{\theta}_i)|^{2-n})\hbox{ in }C^1_{loc}(\rn\setminus\{\tilde{\theta}_i,\, \sigma(\tilde{\theta}_i)/\, i\in I\})
\end{equation}
where $K\geq 0$, $I'$, $\lambda_i$ are as above,
$$\tilde{\theta}_i:=\lim_{\alpha\to +\infty}\frac{\varphi^{-1}(\xai)-\varphi^{-1}(\xa)}{\ra}\hbox{ for all }i\in I$$
and $\sigma:\rn\to\rn$ is the orthogonal symmetry with respect to the hyperplane $\{x_1=\rho\}$, that is
\begin{equation}\label{def:sigma}
\sigma(x_1,x')=(2\rho-x_1, x')\hbox{ for all }(x_1,x')\in\rn.
\end{equation}
\end{prop}
\medskip\noindent{\it Proof of Proposition \ref{th:cv:ra}:} As in the statement of the proposition, we distinguish two cases.

\medskip\noindent{\bf Case \ref{th:cv:ra}.1:} we assume that
\begin{equation}\label{lim:da:ma:1}
\lim_{\alpha\to +\infty}\frac{d(\xa,\partial\Omega)}{\ra}=+\infty.
\end{equation}
We let $R>0$ so that, for $\alpha\in\nn$ large enough, it follows from \eqref{def:va} that $\va(x)$ makes sense for all $x\in B_R(0)$. We fix $x\in B_R(0)$. It follows from \eqref{est:pointwise:V} that
\begin{equation}\label{est:va:1}
\va(x)=(1+o(1))\left(\frac{\ra^{n-2}\bua}{\ma^{\frac{n-2}{2}}}+\sum_{i=1}^N\frac{\ra^{n-2}\Vai(\xa+\ra x)}{\ma^{\frac{n-2}{2}}}\right)
\end{equation}
when $\alpha\to +\infty$. We estimate the right-hand-side with Proposition \ref{lemmaprojection}. We have to distinguish whether $i\in I$ or not

\medskip\noindent{\bf Step \ref{th:cv:ra}.1.1:} Let $i\in I$. We define
$$\theta_{i,\alpha}:=\frac{\xai-\xa}{\ra}$$
for all $\alpha\in\nn$. In particular, $\lim_{\alpha\to +\infty}\theta_{i, \alpha}=\theta_i$ where $\theta_i$ is defined    in \eqref{def:thetai:th}. Therefore
\begin{eqnarray}
\frac{\ra^{n-2}\Uai(\xa+\ra x)}{\ma^{\frac{n-2}{2}}}&=&\left(\frac{\mai\ra^2}{\ma(\mai^2+|\xa-\xai+\ra x|^2)}\right)^{\frac{n-2}{2}}\nonumber\\
&= & \left(\frac{\mai}{\ma(\left(\frac{\mai}{\ra}\right)^2+|x-\theta_{i, \alpha}|^2)}\right)^{\frac{n-2}{2}}\nonumber\\
&= & \left(\lim_{\alpha\to +\infty}\left(\frac{\mai}{\ma}\right)^{\frac{n-2}{2}} \right)|x-\theta_{i}|^{2-n}+o(1)\label{bnd:va:1}
\end{eqnarray}
for all $x\in B_R(0)\setminus \{\theta_{ i}\}$ when $\alpha\to\ +\infty$. Note that these quantities are well-defined due to point (iv) of the hypothesis of Proposition \ref{th:cv:ra}.

\medskip\noindent{\bf Step \ref{th:cv:ra}.1.2:} Let $i\in I^c$ such that
$$\lim_{\alpha\to +\infty}\frac{|\xai-\xa|}{\ra}=+\infty.$$
Let $\alpha_0\in\nn$ be large enough such that $|\xai-\xa|\geq 2R \ra$ for all $\alpha\geq \alpha_0$. Then
$$\left||\xa-\xai+\ra x|-|\xai-\xa|\right|\leq \ra|x|=O(\ra)=o(|\xa-\xai|)$$
when $\alpha\to +\infty$ and uniformly for all $x\in B_R(0)$. Therefore, we have that
\begin{eqnarray}
\frac{\ra^{n-2}\Uai(\xa+\ra x)}{\ma^{\frac{n-2}{2}}}&=&\left(\frac{\mai\ra^2}{\ma(\mai^2+|\xa-\xai+\ra x|^2)}\right)^{\frac{n-2}{2}}\nonumber\\
&= & (1+o(1))\left(\frac{\mai\ra^2}{\ma(\mai^2+|\xa-\xai|^2)}\right)^{\frac{n-2}{2}}\label{bnd:va:2}
\end{eqnarray}
for all $x\in B_R(0)$ and all $\alpha\in\nn$.

\medskip\noindent{\bf Step \ref{th:cv:ra}.1.3:} Let $i\in I^c$ such that
$$|\xai-\xa|=O(\ra)\hbox{ when }\alpha\to +\infty.$$
Since $i\not\in I$ and points (ii) and (iv) of the hypothesis of Proposition \ref{th:cv:ra} hold, we then have that $\ra=o(\mai)$ when $\alpha\to +\infty$: in particular, $|\xa-\xai|=o(\mai)$ when $\alpha\to +\infty$. We then get that
\begin{eqnarray}
\frac{\ra^{n-2}\Uai(\xa+\ra x)}{\ma^{\frac{n-2}{2}}} &=& \left(\frac{\mai\ra^2}{\ma(\mai^2+o(\mai^2))}\right)^{\frac{n-2}{2}}\nonumber\\
&=&(1+o(1))\left(\frac{\mai\ra^2}{\ma(\mai^2+|\xai-\xa|^2)}\right)^{\frac{n-2}{2}}\label{bnd:va:3}
\end{eqnarray}
for all $x\in B_R(0)$ and all $\alpha\in\nn$.

\medskip\noindent{\bf Step \ref{th:cv:ra}.1.4:} Let $i\in\{1,..., N\}$. We claim that
\begin{equation}\label{bnd:va:4}
\frac{\ra^{n-2}\mai^{\frac{n-2}{2}}}{\ma^{\frac{n-2}{2}}}=o(1)
\end{equation}
when $\alpha\to +\infty$. Indeed, it follows from Point (iii) of Proposition \ref{prop:claimB} that
\begin{equation}\label{comp:ea:ma}
\mai^{\frac{n-2}{2}}\leq C\int_{\Omega\cap B_{\mai}(\xai)}\ua^{\crit-1}\, dx\leq C\int_{\Omega}\ua^{\crit-1}\, dx=C\ea\int_\Omega\ua\, dx=o(\bua)
\end{equation}
when $\alpha\to +\infty$. Therefore, \eqref{bnd:va:4} follows from point (iii) of the hypothesis of Proposition \ref{th:cv:ra}.

\medskip\noindent{\bf Step \ref{th:cv:ra}.1.5:} We let $i\in\{1,..., N\}$ such that the hypothesis of point (iii) of Proposition \ref{lemmaprojection} hold. Since $\pip^{-1}(\xai)\not\in \Omega$, we have that $|\xa-\pip^{-1}(\xai)|\geq d(\xa,\partial\Omega)$. Moreover, since \eqref{lim:da:ma:1} holds, we have that
\begin{eqnarray*}
\frac{\ra^{n-2}\tUai(\xa+\ra x)}{\ma^{\frac{n-2}{2}}} &=& \left(\frac{\mai\ra^2}{\ma(\mai^2+|\xa+\ra x-\pip^{-1}(\xai)|^2)}\right)^{\frac{n-2}{2}}\\
&=&(1+o(1))\left(\frac{\mai\ra^2}{\ma(\mai^2+|\xa-\pip^{-1}(\xai)|^2)}\right)^{\frac{n-2}{2}}
\end{eqnarray*}
Assume that $i\in I$: in this case, we have that $\mai=O(\ma)$ when $\alpha\to +\infty$. Since in addition $|\xa-\pip^{-1}(\xai)|\geq d(\xa,\partial\Omega)$ and \eqref{lim:da:ma:1} holds, we have that
\begin{equation}\label{bnd:va:5}
\lim_{\alpha\to +\infty}\frac{\ra^{n-2}\tUai(\xa+\ra x)}{\ma^{\frac{n-2}{2}}}=0\hbox{ if }i\in I.
\end{equation}
Assume that $i\not\in I$. Since $|\xa-\xai|=O(|\xa-\pip^{-1}(\xai)|)$ when $\alpha\to +\infty$, we have that
\begin{equation}\label{bnd:va:6}
\frac{\ra^{n-2}\tUai(\xa+\ra x)}{\ma^{\frac{n-2}{2}}}=O\left(\left(\frac{\mai\ra^2}{\ma(\mai^2+|\xa-\xai|^2)}\right)^{\frac{n-2}{2}}\right)
\end{equation}
when $\alpha\to +\infty$.

\medskip\noindent Plugging \eqref{bnd:va:1}-\eqref{bnd:va:6} into  \eqref{est:va:1} and using Proposition \ref{lemmaprojection}, we get that
\begin{equation}\label{lim:va:bis}
\lim_{\alpha\to +\infty}\va(x)=K+\sum_{i\in I'}\lambda_i|x-\theta_i|^{2-n}
\end{equation}
for all $x\in\rn\setminus\{\theta_i/\, i\in I\}$, where $K$, $I'$, $\theta_i$ and $\lambda_i$ are as in \eqref{def:Iprime}, \eqref{def:thetai:th} and \eqref{lower:K}. Moreover, as easily checked, this convergence is uniform on every compact subset of $\rn\setminus\{\theta_i/\, i\in I\}$.

\medskip\noindent{\bf Step \ref{th:cv:ra}.1.6:} We claim that \eqref{lim:va} holds. We prove the claim. It follows from equation \eqref{eq:ua} that
\begin{equation}\label{eq:va}
\Delta\va+\ra^2\ea\va= n(n-2)\left(\frac{\ma}{\ra}\right)^2\va^{\crit-1}\hbox{ in }B_R(0)
\end{equation}
for all $\alpha\in\nn$. Since $\ma=o(\ra)$ when $\alpha\to +\infty$, it follows from \eqref{lim:va:bis} and standard elliptic theory that \eqref{lim:va} holds. This proves the claim.

\medskip\noindent This ends the proof of Proposition \ref{th:cv:ra} in Case \ref{th:cv:ra}.2.

 \medskip\noindent{\bf Case \ref{th:cv:ra}.2:} We assume that
\begin{equation*}
\lim_{\alpha\to +\infty}\frac{d(\xa,\partial\Omega)}{\ra}=\rho
\end{equation*}
with $\rho\in [0,+\infty)$. In particular, $\lim_{\alpha\to +\infty}\xa=x_0\in\partial\Omega$. We consider the domain $U_{x_0}$, the extension $\tg$ of the Euclidean metric $\xi$, the chart $\varphi$ and the extension $\tua$ defined in Lemma \ref{lem:ext}. Let $R>0$ and let $\alpha>0$ large enough such that
$$B_R(0)\subset \ra^{-1}(\varphi^{-1}(U_{x_0})-\varphi^{-1}(\xa)).$$
Let us define $(x_{1, \alpha},\xa'):=\varphi^{-1}(\xa)$ with $x_{1, \alpha}\leq 0$ and $\xa'\in \rr^{n-1}$. Therefore, as is easily checked, we have that for any $x\in B_R(0)$,
$$\varphi(\varphi^{-1}(\xa)+\ra x)\in \omegabar\; \Leftrightarrow \;x_1\leq \frac{|x_{1, \alpha}|}{\ra}.$$
We consider the extension $\tua$ of $\ua$ defined as in Lemma \ref{lem:ext}. In particular, the maps $\varphi$, $\pi$, $\tilde{\pi}$, $\pip$, $\tpip$ refer to the point $x_0$. Given $i\in \{1,...,N\}$, it follows from the properties of the $\Vai$'s (see Proposition \ref{lemmaprojection}) that
\begin{equation}\label{diff:sym:vai}
\Vai(\pip^{-1}(x))=(1+o(1))\Vai(x)+O(\mai^{\frac{n-2}{2}})
\end{equation}
when $\alpha\to +\infty$ uniformly for $x\in U_{x_0}$ (up to taking $U_{x_0}$ a sufficiently small neighborhood of $x_0$ in $\rn$). Therefore, it follows from \eqref{est:pointwise:V} that
$$\tua(x)=(1+o(1))\left(\bua+\sum_{i=1}^N\Vai(x)\right)$$
when $\alpha\to +\infty$ uniformly for $x\in U_{x_0}\cap\Omega$. Consequently, using \eqref{diff:sym:vai} and \eqref{bnd:va:4}, for $x\in B_R(0)$, we have that
\begin{equation}\label{eq:tva:1}
\tva(x)=(1+o(1))\left(\frac{\ra^{n-2}\bua}{\ma^{\frac{n-2}{2}}}+\sum_{i=1}^N\frac{\ra^{n-2}\Vai(\varphi((\xaun, \xa')+\ra x))}{\ma^{\frac{n-2}{2}}}\right)
\end{equation}
when $\alpha\to +\infty$ uniformly for $x\in B_R(0)$. Here again, we distinguish whether $i\in I$ or not.

\medskip\noindent{\bf Step \ref{th:cv:ra}.2.1:} we fix $i\in I^c$. Then there exists $\tau_i\geq 1$ such that
\begin{equation}
\frac{\ra^{n-2}\Vai(\varphi((\xaun, \xa')+\ra x))}{\ma^{\frac{n-2}{2}}}=(1+o(1))\tau_i\left(\frac{\mai\ra^2}{\ma(\mai^2+|\xai-\xa|^2)}\right)^{\frac{n-2}{2}}\label{lim:uai:0}
\end{equation}
for all $x\in B_R(0)$ and all $\alpha\in\nn$. The proof goes as in Case \ref{th:cv:ra}.1 above and we omit it.

\medskip\noindent{\bf Step \ref{th:cv:ra}.2.2:} We fix $i\in I$. Mimicking what was done in Case \ref{th:cv:ra}.1.1, we define
$$\tilde{\theta}_{i,\alpha}:=\frac{\varphi^{-1}(\xai)-\varphi^{-1}(\xa)}{\ra}\hbox{ and }\tilde{\theta}_i:=\lim_{\alpha\to +\infty}\tilde{\theta}_{i,\alpha}$$
for all $i\in I$. Using that $d\varphi_0$ is an orthogonal transformation and proceeding as in Step \ref{th:cv:ra}.1.1, we get that
\begin{equation}
\frac{\ra^{n-2}\Uai(\varphi((\xaun, \xa')+\ra x))}{\ma^{\frac{n-2}{2}}}=  \left(\lim_{\alpha\to +\infty}\left(\frac{\mai}{\ma}\right)^{\frac{n-2}{2}} \right)|x-\tilde{\theta}_{i}|^{2-n}+o(1)\label{lim:uai:00}
\end{equation}
for all $x\in B_R(0)\setminus \{\tilde{\theta}_{ i}\}$ when $\alpha\to\ +\infty$. Here again we omit the proof and we refer to Step \ref{th:cv:ra}.1.1.

\medskip\noindent{\bf Step \ref{th:cv:ra}.2.3:} We fix $i\in I$. In particular, $\lim_{\alpha\to +\infty}\xai=x_0$. We assume that $\xai\not\in\partial\Omega$ for all $\alpha\in\nn$. We then have that
\begin{eqnarray}
\frac{\ra^{n-2}\tUai(\varphi((\xaun, \xa')+\ra x))}{\ma^{\frac{n-2}{2}}} &= &\left(\frac{\mai\ra^2}{\ma(\mai^2+|\varphi((\xaun, \xa')+\ra x)-\pip^{-1}(\xai)|^2)}\right)^{\frac{n-2}{2}}\label{lim:tuai:1}
\end{eqnarray}
for all $\alpha\in\nn$ and all $x\in B_R(0)$. Here, note that since we work in a neighborhood of $x_0$, we use the maps $\varphi$, $\pi$ defined above. We define $((\xai)_1,\xai'):=\varphi^{-1}(\xai)$ for all $\alpha\in\nn$. We have that
\begin{eqnarray}
&&|\varphi((\xaun, \xa')+\ra x)-\pip^{-1}(\xai)|= |\varphi((\xaun, \xa')+\ra x)-\varphi\circ\pi^{-1}((\xaun,\xa')+\ra\tilde{\theta}_{i,\alpha})|\nonumber\\
&&= (1+o(1))|(\xaun, \xa')+\ra x-\pi^{-1}(\xaun,\xa')-\ra\pi^{-1}(\tilde{\theta}_{i,\alpha})|\nonumber\\
&&=  (1+o(1))\ra \left|\left(2\frac{\xaun}{\ra}, 0\right)+ x-\pi^{-1}(\tilde{\theta}_{i,\alpha})\right|\label{lim:tuai:2}
\end{eqnarray}
independently, since $d\varphi_0$ is an orthogonal transformation (this is due to the choice made in Lemma \ref{lem:ext}), we have that
\begin{equation}\label{lim:tuai:3}
d(\xa, \partial\Omega)=(1+o(1))|x_{1, \alpha}|
\end{equation}
when $\alpha\to +\infty$. In particular,
$$\lim_{\alpha\to +\infty}\frac{|x_{1, \alpha}|}{\ra}=\rho.$$
Since $\xaun<0$, plugging together \eqref{lim:tuai:1}, \eqref{lim:tuai:2} and \eqref{lim:tuai:3}, we have that
\begin{equation}\label{lim:tuai:4}
\frac{\ra^{n-2}\tUai(\varphi((\xaun, \xa')+\ra x))}{\ma^{\frac{n-2}{2}}} =\left(\lim_{\alpha\to +\infty}\left(\frac{\mai}{\ma}\right)^{\frac{n-2}{2}} \right)|x-\sigma(\tilde{\theta}_{i})|^{2-n}+o(1)
\end{equation}
when $\alpha\to +\infty$ uniformly on compact subsets of $\rn\setminus\{\sigma(\tilde{\theta}_{i})\}$.

\medskip\noindent{\bf Step \ref{th:cv:ra}.2.4:} Given $i\in I'$, we define
$$\lambda_i:=\lim_{\alpha\to +\infty}\left(\frac{\mai}{\ma}\right)^{\frac{n-2}{2}}\cdot \left\{\begin{array}{ll}
1 & \hbox{ if }\xai\not\in\partial\Omega\hbox{ for all }\alpha\in\nn\\
\frac{1}{2} & \hbox{ if }\xai\in\partial\Omega\hbox{ for all }\alpha\in\nn
\end{array}\right.$$
It then follows from \eqref{eq:tva:1}, \eqref{lim:uai:0}, \eqref{lim:uai:00}, \eqref{lim:tuai:4}, Step \ref{th:cv:ra}.1.4 and Proposition \ref{lemmaprojection} that
\begin{equation}\label{lim:tva:1}
\lim_{\alpha\to +\infty}\tva(x)=K+\sum_{i\in I'}\lambda_i(|x-\tilde{\theta}_i|^{2-n}+|x-\sigma(\tilde{\theta}_i)|^{2-n})
\end{equation}
uniformly for all $x$ in compact subsets of $\rn\setminus\{\tilde{\theta}_i,\sigma(\tilde{\theta}_i)/\, i\in I\}$, where $K$ is defined in \eqref{lower:K}. We define the metric $g_\alpha:=(\varphi^\star \tg)(\varphi^{-1}(\xa)+\ra x)$ for $x\in \ra^{-1}(\varphi^{-1}(U_{x_0})-\varphi^{-1}(\xa))$. With a change of variables, equation \eqref{eq:ua} rewrites
\begin{equation*}
\Delta_{\tilde{g}_\alpha}\tva+\ea\ra^2\tva=\left(\frac{\ma}{\ra}\right)^2\tva^{\crit-1}
\end{equation*}
weakly in $B_R(0)$. It then follows from standard elliptic theory that \eqref{lim:tva:1} holds in $C^1_{loc}$. This proves \eqref{val:lim:tva}, and this concludes the proof of Proposition \ref{th:cv:ra} in Case \ref{th:cv:ra}.2.

\medskip\noindent Proposition \ref{th:cv:ra} is a direct consequence of Cases \ref{th:cv:ra}.1 and \ref{th:cv:ra}.2.\hfill$\Box$

\subsection{Convergence at appropriate scale}
We fix $i\in\{1,...,N\}$. We define
$$J_i:=\{j\neq i/\, \mai=O(\maj)\hbox{ when }\alpha\to +\infty\}.$$
We define also

\begin{equation}\label{def:sai}
\sai:=\left\{\begin{array}{ll}
\min\left\{\frac{\mai^{\frac{1}{2}}}{\bua^{\frac{1}{n-2}}}, \min_{j\in J_i}\left(\frac{\mai}{\maj}(\maj^2+|\xaj-\xai|^2)\right)^{\frac{1}{2}}\right\} & \hbox{ if }\xai\in\partial\Omega\\
\min\left\{\frac{\mai^{\frac{1}{2}}}{\bua^{\frac{1}{n-2}}}, \min_{j\in J_i}\left(\frac{\mai}{\maj}(\maj^2+|\xaj-\xai|^2)\right)^{\frac{1}{2}}, d(\xai,\partial\Omega)\right\} & \hbox{ if }\xai\not\in\partial\Omega
\end{array}\right.
\end{equation}
Applying Proposition \ref{th:cv:ra}, we get the two following propositions:

\begin{prop}\label{prop:cv:sai:1} Let $i\in\{1,...,N\}$ and assume that
$$\lim_{\alpha\to +\infty}\frac{d(\xai,\partial\Omega)}{\sai}=+\infty.$$
For $x\in \sai^{-1}(\Omega-\xai)$, we define
$$\vai(x):=\frac{\sai^{n-2}}{\mai^{\frac{n-2}{2}}}\ua(\xai+\sai x).$$
We define
$$I_i:=\{j\in\{1,...,N\}/\, \xaj-\xai=O(\sai)\hbox{ and }\maj=o(\sai)\hbox{ when }\alpha\to +\infty\}$$
and
$$\theta_j:=\lim_{\alpha\to +\infty}\frac{\xaj-\xai}{\sai}\hbox{ for all }j\in I_i.$$
Then there exists $v_i\in C^2(\rn\setminus\{\theta_j/\, j\in I_i\})$ such that
\begin{equation}\label{cv:vai:prop}
\lim_{\alpha\to +\infty}\vai=v_i\hbox{ in }C^2_{loc}(\rn\setminus\{\theta_j/\, j\in I_i\}).
\end{equation}
In addition, there exists $K\geq 0$ and $\lambda_j>0$ for all $j\in I'_i:=\{j\in I_i/\, \mai \asymp \maj\}$ such that
\begin{equation}\label{explicit:vai:prop}
v_i(x)=K+\sum_{j\in I_i'}\lambda_j|x-\theta_j|^{2-n}\hbox{ for all }x\in \rn\setminus\{\theta_j/\, j\in I_i\}.
\end{equation}
Moreover, there exists $\delta>0$, there exists $\lambda_i'>0$ and $\psi_i\in C^2(B_{2\delta}(0))$ harmonic such that
\begin{equation}\label{cons:prop:1}
v_i(x):=\frac{\lambda_i'}{|x|^{n-2}}+\psi_i(x)\hbox{ for all }x\in B_{2\delta}(0)\setminus\{0\}\hbox{ with }\psi_i(0)>0.
\end{equation}
\end{prop}

\begin{prop}\label{prop:cv:sai:2} Let $i\in\{1,...,N\}$ and assume that
$$\lim_{\alpha\to +\infty}\frac{d(\xai,\partial\Omega)}{\sai}=\rho\in [0,+\infty).$$
In particular $\lim_{\alpha\to +\infty}\xai=x_0\in \partial\Omega$. We let $\varphi$ be a chart around $x_0$ as in Lemma \ref{lem:ext}. For $x\in \sai^{-1}(\Omega-\xai)$, we define
$$\tvai(x):=\frac{\sai^{n-2}}{\mai^{\frac{n-2}{2}}}\tua\circ\varphi(\varphi^{-1}(\xai)+\sai x).$$
We define
$$I_i:=\{j\in\{1,...,N\}/\, \xaj-\xai=O(\sai)\hbox{ and }\maj=o(\sai)\hbox{ when }\alpha\to +\infty\}$$
and
\begin{equation}\label{def:tthetaj}
\tilde{\theta}_j:=\lim_{\alpha\to +\infty}\frac{\varphi^{-1}(\xaj)-\varphi^{-1}(\xai)}{\sai}\hbox{ for all }j\in I_i.
\end{equation}
We define $\sigma(x_1,x'):=(2\rho-x_1,x')$ for all $(x_1,x')\in\rn$. Then there exists $\tv_i\in C^2(\rn\setminus\{\tilde{\theta}_j, \sigma(\tilde{\theta}_j)/\, j\in I_i\})$ such that
\begin{equation}\label{cv:tvai:prop}
\lim_{\alpha\to +\infty}\tvai=\tv_i\hbox{ in }C^1_{loc}(\rn\setminus\{\tilde{\theta}_j, \sigma(\tilde{\theta}_j)/\, j\in I_i\}).
\end{equation}
In addition, there exists $K\geq 0$ and $\lambda_j>0$ for all $j\in I'_i:=\{j\in I_i/\, \mai \asymp \maj\}$ such that
\begin{equation}\label{explicit:tvai:prop}
\tv_i(x)=K+\sum_{j\in I_i'}\lambda_j\left(|x-\tilde{\theta}_j|^{2-n}+|x-\sigma(\tilde{\theta}_j)|^{2-n}\right)\hbox{ for all }x\in \rn\setminus\{\theta_j, \sigma(\tilde{\theta}_j)/\, j\in I_i\}.
\end{equation}
Moreover, there exists $\delta>0$, there exists $\lambda_i'>0$ and $\tilde{\psi}_i\in C^2(B_{2\delta}(0))$ harmonic such that
\begin{equation}\label{cons:prop:2}
\tv_i(x):=\frac{\lambda_i'}{|x|^{n-2}}+\tilde{\psi}_i(x)\hbox{ for all }x\in B_{2\delta}(0)\setminus\{0\}\hbox{ with }\tilde{\psi}_i(0)>0.
\end{equation}
\end{prop}

\medskip\noindent{\it Proof of Propositions \ref{prop:cv:sai:1} and \ref{prop:cv:sai:2}:} We apply Proposition \ref{th:cv:ra}.

\medskip\noindent{\bf Step \ref{prop:cv:sai:1}.1:} we claim that points (i) to (v) of Proposition \ref{th:cv:ra} hold with
$$\ma:=\mai\hbox{ and }\ra:=\sai\hbox{ for all }\alpha\in\nn.$$
We prove the claim.

\medskip\noindent{\bf Step \ref{prop:cv:sai:1}.1.1} We claim that (i) holds.

\smallskip\noindent We prove this claim via two claims. We first claim that
\begin{equation}\label{lim:sai}
\lim_{\alpha\to +\infty}\sai=0.
\end{equation}
We prove the claim. Indeed, it follows from the estimate \eqref{comp:ea:ma} and the definition \eqref{def:sai} of $\sai$ that
$$\sai^{n-2}\leq \frac{\mai^{\frac{n-2}{2}}}{\bua}\leq C\ea=o(1)\hbox{ when }\alpha\to +\infty.$$
This proves \eqref{lim:sai}. This proves the claim.

\medskip\noindent We claim that
\begin{equation}\label{lim:sai:mai}
\mai=o(\sai)\hbox{ when }\alpha\to +\infty.
\end{equation}
We prove the claim by contradiction. Assume that $\sai=O(\mai)$ when $\alpha\to +\infty$. Since $\lim_{\alpha\to +\infty}\mai^{-1}d(\xai,\partial\Omega)=+\infty$ if $\xai\not\in\partial\Omega$ (see Proposition \ref{prop:claimB}), it then follows from the definition of $\sai$ that there exists $j\in J_i$ such that
\begin{equation}\label{ineq:proof:sai:mai}
\maj^2+|\xaj-\xai|^2=O(\mai\maj)\hbox{ when }\alpha\to +\infty.
\end{equation}
In particular, $\maj=O(\mai)$ when $\alpha\to +\infty$. Since $j\in J_i$, we then get that $\maj\asymp\mai$ when $\alpha\to +\infty$. It then follows from \eqref{ineq:proof:sai:mai} that $\xaj-\xai=O(\mai)$ when $\alpha\to +\infty$. A contradiction with point (ii) of Proposition \ref{prop:claimB}.  This proves that \eqref{lim:sai:mai} holds. This proves the claim.

\smallskip\noindent These two claims prove that (i) holds. This ends Step \ref{prop:cv:sai:1}.1.1.

\medskip\noindent{\bf Step \ref{prop:cv:sai:1}.1.2:} Let $k\in\{1,...,N\}$. We assume that $\xak-\xai=O(\sai)$ when $\alpha\to +\infty$. We claim that
\begin{equation}\label{claim:iii}
\sai\not\asymp \mak\hbox{ when }\alpha\to +\infty.
\end{equation}
We prove the claim by contradiction and we assume that
\begin{equation}\label{claim:iii:contrap}
\sai\asymp \mak\hbox{ when }\alpha\to +\infty.
\end{equation}
Since $\mai=o(\sai)$ when $\alpha\to +\infty$, we then get that $\mai=o(\mak)$ when $\alpha\to +\infty$, and therefore $k\in J_i$. It then follows from the definiiton of $\sai$ that
$$\sai^2\leq \frac{\mai}{\mak}(\mak^2+|\xai-\xak|^2)=o(\mak^2)+o(\sai^2)$$
when $\alpha\to +\infty$, and then $\sai=o(\mak)$ when $\alpha\to +\infty$: a contradiction with \eqref{claim:iii:contrap}. Then \eqref{claim:iii} holds and the claim is proved. This ends Step \ref{prop:cv:sai:1}.1.2.

\medskip\noindent{\bf Step \ref{prop:cv:sai:1}.1.3:} Point (iii) is a straightforward consequence of the definition \eqref{def:sai} of $\sai$.

\medskip\noindent{\bf Step \ref{prop:cv:sai:1}.1.4:} We let $j\in \{1,...,N\}$ be such that $\xaj-\xai=O(\sai)$ and $\maj=o(\sai)$ when $\alpha\to +\infty$. We claim that
\begin{equation}\label{claim:iv}
\maj=O(\mai)\hbox{ when }\alpha\to +\infty.
\end{equation}
We prove the claim by contradiction and we assume that
\begin{equation}\label{claim:iv:contrap}
\mai=o(\maj)\hbox{ when }\alpha\to +\infty.
\end{equation}
Therefore, $j\in J_i$ and we have with \eqref{lim:sai:mai} that
$$\sai^2\leq\frac{\mai}{\maj}(\maj^2+|\xai-\xaj|^2)=o(\maj^2)+o(\sai^2)=o(\sai^2)$$
when $\alpha\to +\infty$. A contradiction. Then \eqref{claim:iv:contrap} does not hold and \eqref{claim:iv} holds. This proves the claim and ends Step \ref{prop:cv:sai:1}.1.4.

\medskip\noindent{\bf Step \ref{prop:cv:sai:1}.1.5:} Let $j\in \{1,...,N\}$ be such that $\lim_{\alpha\to +\infty}\frac{|\xai-\xaj|}{\sai}=+\infty$. We claim that
\begin{equation}\label{claim:v:1}
\frac{\sai^2\maj}{\mai(\maj^2+|\xai-\xaj|^2)}=O(1)\hbox{ when }\alpha\to +\infty.
\end{equation}
We prove the claim. Assume first that $\maj=o(\mai)$ when $\alpha\to +\infty$: we then get that
$$\frac{\sai^2\maj}{\mai(\maj^2+|\xai-\xaj|^2)}=O\left(\frac{\maj}{\mai}\cdot\frac{\sai^2}{|\xai-\xaj|)}\right)=o(1)$$
when $\alpha\to +\infty$. This proves \eqref{claim:v:1}, and the claim is proved in this case.\par
\smallskip\noindent Assume that $\mai=O(\maj)$ when $\alpha\to +\infty$. Then $j\in J_i$ and \eqref{claim:v:1} follows from the definition of $\sai$.\par
\smallskip\noindent In the two cases, we have proved \eqref{claim:v:1}. This proves the claim and ends Step \ref{prop:cv:sai:1}.1.5.

\medskip\noindent{\bf Step \ref{prop:cv:sai:1}.1.6:} Let $j\in \{1,...,N\}$ be such that $\xai-\xaj=O(\sai)$ and $\sai=o(\maj)$ when $\alpha\to +\infty$. We claim that
\begin{equation}\label{claim:v:2}
\frac{\sai^2\maj}{\mai(\maj^2+|\xai-\xaj|^2)}=O(1)\hbox{ when }\alpha\to +\infty.
\end{equation}
 We prove the claim. We first assume that $\maj=o(\mai)$ when $\alpha\to +\infty$. We then get that
$$\frac{\sai^2\maj}{\mai(\maj^2+|\xai-\xaj|^2)}=O\left(\frac{\maj}{\mai}\cdot\frac{\sai^2}{\maj^2}\right)=o(1)$$
when $\alpha\to +\infty$. Then \eqref{claim:v:2} holds in this case. The case $\mai=O(\maj)$ when $\alpha\to +\infty$ is dealt as in Step \ref{prop:cv:sai:1}.1.5. This proves \eqref{claim:v:2} and then the claim. This ends Step \ref{prop:cv:sai:1}.1.6.

\medskip\noindent{\bf Step \ref{prop:cv:sai:1}.1.7:} point (v) is a consequence of Steps \ref{prop:cv:sai:1}.1.5 and \ref{prop:cv:sai:1}.1.6.

\medskip\noindent Therefore, points (i) to (v) of the hypothesis of Proposition \ref{th:cv:ra} are satisfied with $\ma:=\mai$ and $\ra:=\sai$. This ends Step 1.\hfill$\Box$

\medskip\noindent Then we can apply Proposition \ref{th:cv:ra} with $\ra:=\sai$ and $\ma:=\mai$.

\medskip\noindent{\bf Step \ref{prop:cv:sai:1}.2:} we assume that
\begin{equation}\label{lim:d:sai:infty}
\lim_{\alpha\to +\infty}\frac{d(\xai,\partial\Omega)}{\sai}=+\infty.
\end{equation}
It then follows from Proposition \ref{th:cv:ra} that there exists $v_i$ as in Proposition \ref{prop:cv:sai:1} such that \eqref{cv:vai:prop} and \eqref{explicit:vai:prop} hold. Moreover, there exists $(\tau_j)_{j}$ such that

\begin{equation}\label{inf:K}
K=\lim_{\alpha\to +\infty}\frac{\sai^{n-2}\bua}{\mai^{\frac{n-2}{2}}}+\sum_{j\not\in I_i}\tau_j\lim_{\alpha\to +\infty}\left(\frac{\sai^2\maj}{\mai(\maj^2+|\xai-\xaj|^2)}\right)^{\frac{n-2}{2}}.
\end{equation}

\medskip\noindent{\bf Step \ref{prop:cv:sai:1}.2.1:} We claim that
\begin{equation}\label{claim:K:theta}
K>0\hbox{ or }\exists j\in I_i'\hbox{ such that }\theta_j\neq 0.
\end{equation}
We prove the claim. If $K>0$, then \eqref{claim:K:theta} holds. We assume that $K=0$. It then follows from \eqref{inf:K} that
\begin{equation}\label{cons:K:0}
\sai=o\left(\frac{\mai^{\frac{1}{2}}}{\bua^{\frac{1}{n-2}}}\right)\hbox{ and }\sai^2=o\left(\frac{\mai}{\maj}(\maj^2+|\xai-\xaj|^2)\right)\hbox{ for all }j\not\in I_i
\end{equation}
when $\alpha\to +\infty$. The definition \eqref{def:sai} of $\sai$, \eqref{lim:d:sai:infty} and \eqref{cons:K:0} yield the existence of $j\in I_i\cap J_i$ such that
\begin{equation}\label{sai:contra}
\sai^2=\frac{\mai}{\maj}(\maj^2+|\xaj-\xai|^2)
\end{equation}
for all $\alpha\in\nn$. Since $j\in J_i$, we have that
\begin{equation}\label{comp:mai:maj}
\mai=O(\maj)\hbox{ when }\alpha\to +\infty\hbox{ and }j\neq i.
\end{equation}
Moreover, since $j\in I_i$, we have that
\begin{equation}\label{app:Ii}
\xaj-\xai=O(\sai)\hbox{ and }\maj=o(\sai)
\end{equation}
When $\alpha\to +\infty$. It then follows from \eqref{sai:contra}, \eqref{comp:mai:maj} and \eqref{app:Ii} that
\begin{equation}\label{eq:mai:maj}
\mai\asymp\maj\hbox{ and }|\xai-\xaj|\asymp\sai \hbox{ when }\alpha\to +\infty.
\end{equation}
In particular, $j\in I_i'$ and $\theta_j\neq 0$. This proves \eqref{claim:K:theta} when $K=0$. This proves the claim and ends Step \ref{prop:cv:sai:1}.2.1.\hfill$\Box$

\medskip\noindent We set
\begin{equation*}
\delta:=\frac{1}{2}\min\{|\theta_j|/\, j\in I_i\hbox{ and }\theta_j\neq 0\}.
\end{equation*}
We define
$$\psi_i(x):=K+\sum_{j\in I_i''}\lambda_j|x-\theta_j|^{2-n}$$
for all $x\in B_{2\delta}(0)$ where $I_i":=\{j\in I_i/\, \theta_j\neq 0\}$. Clearly $\psi_i$ is smooth and harmonic on $B_\delta(0)$. We define $\lambda_i'=\sum_{j\in I_i'\setminus I_i"}\lambda_j$, so that one has that
$$v_i(x)=\frac{\lambda_i'}{|x|^{n-2}}+\psi_i(x)\hbox{ for all }x\in B_{2\delta}(0)\setminus\{0\}.$$
Note that $\lambda'_i\geq\lambda_i>0$.

\medskip\noindent{\bf Step \ref{prop:cv:sai:1}.2.2:}  We claim that
$$\psi_i(0)>0.$$
We prove the claim. Indeed, if $K>0$, the claim is clear. If $K=0$, it follows from \eqref{claim:K:theta} that there exists $j\in I_i"$, and then $\psi_i(0)\geq \lambda_j|\theta_j|^{2-n}>0$. This proves the claim.

\medskip\noindent Proposition \ref{prop:cv:sai:1} is a consequence of Steps \ref{prop:cv:sai:1}.1 and \ref{prop:cv:sai:1}.2.\hfill$\Box$

\medskip\noindent{\bf Step \ref{prop:cv:sai:1}.3:} we assume that
\begin{equation}\label{lim:d:sai:finite}
\lim_{\alpha\to +\infty}\frac{d(\xai,\partial\Omega)}{\sai}=\rho\geq 0.
\end{equation}
In this case, the proof of Proposition \ref{prop:cv:sai:2} goes basically as the proof of Proposition \ref{prop:cv:sai:1}. We stress here on the differences.

\smallskip\noindent It follows from Proposition \ref{th:cv:ra} that there exists $\tv_i$ as in Proposition \ref{prop:cv:sai:2} such that \eqref{cv:tvai:prop} and \eqref{explicit:tvai:prop} holds. We define
$$\delta:=\frac{1}{2}\min\{|\tilde{\theta}_j|/\, j\in I_i\hbox{ and }\tilde{\theta}_j\neq 0\}.$$
 We define
$$\tilde{\psi}_i(x):=K+\sum_{j\in I_i''}\lambda_j(|x-\tilde{\theta}_j|^{2-n}+|x-\sigma(\tilde{\theta}_j)|^{2-n})+\left\{\begin{array}{ll}
\lambda'_i|x-\sigma(\tilde{\theta}_i)|^{2-n}&\hbox{ if }\sigma(\tilde{\theta}_i)\neq 0\\
0&\hbox{ if }\sigma(\tilde{\theta}_i)=0
\end{array}\right.$$
for all $x\in B_{2\delta}(0)$ where $I_i":=\{j\in I_i"/\, \theta_j\neq 0\}$ and $\lambda_i'>0$ is as in Step \ref{prop:cv:sai:1}.2.1. In particular, as in Step \ref{prop:cv:sai:1}.2, we have that
$$\tv_i(x)=\frac{\lambda_i'}{|x|^{n-2}}+\tilde{\psi}_i(x)$$
for all $x\in B_{2\delta}(0)$.

\medskip\noindent We claim that
\begin{equation}\label{sgn:psi}
\tilde{\psi}_i(0)>0.
\end{equation}
We prove the claim. As in Step \ref{prop:cv:sai:1}.2.2, \eqref{sgn:psi} holds if $K>0$. Assume that  $K=0$. Arguing as in Step \ref{prop:cv:sai:1}.2.1, we get that
$$\left\{\begin{array}{l}
\hbox{either }\sai= d(\xai,\partial\Omega)\hbox{ and }\xai\not\in\partial\Omega\\
\hbox{or there exists }j\in I_i\cap J_i\hbox{ such that }\sai^2=\frac{\mai}{\maj}(\maj^2+|\xaj-\xai|^2)\end{array}\right.$$

\medskip\noindent{\bf Step \ref{prop:cv:sai:1}.3.1:} we assume that
$$\sai:=d(\xai,\partial\Omega)$$
for all $\alpha\in\nn$. In particular, it follows from \eqref{lim:d:sai:finite} that that $\rho=1>0$ and then $\sigma(\tilde{\theta}_i)=\sigma(0)=(2\rho, 0)\neq 0$ and then $\tilde{\psi}_i(0)\geq \lambda'_i|\sigma(\tilde{\theta}_i)|^{2-n}=\lambda'_i (2\rho)^{2-n}>0$.

\medskip\noindent{\bf Step \ref{prop:cv:sai:1}.3.2:} we assume that there exists $j\in I_i\cap J_i$ such that
$$\sai^2=\frac{\mai}{\maj}(\maj^2+|\xaj-\xai|^2)$$
for all $\alpha\in\nn$. Mimicking what was done in Step \ref{prop:cv:sai:1}.2.2, we get again that $\tilde{\psi}_i(0)>0$.

\medskip\noindent In all the cases, we have proved that $\tilde{\psi}_i(0)>0$. This proves \eqref{sgn:psi}, and then ends Step \ref{prop:cv:sai:1}.3.\hfill$\Box$

\medskip\noindent Proposition \ref{prop:cv:sai:2} is a consequence of Steps \ref{prop:cv:sai:1}.1 and \ref{prop:cv:sai:1}.3.\hfill$\Box$

\section{Estimates of the interior blow-up rates}
This section is devoted to the analysis of the concentration at the points $\xai$ away from the boundary.
\begin{thm}\label{th:asymp:int} Let $i\in\{1,...,N\}$. We assume that
\begin{equation}\label{lim:d:mai:infty}
\lim_{\alpha\to +\infty}\frac{d(\xai,\partial\Omega)}{\mai}=+\infty.
\end{equation}
Then $n\geq 4$ (equation \eqref{lim:d:mai:infty} does not hold in dimension $n=3$). Concerning the blow-up rate, there exists $c_i>0$ such that
\begin{equation}\label{rate:sai:int}
\lim_{\alpha\to +\infty}\frac{\ea\sai^{n-2}}{\mai^{n-4}}=c_i\hbox{ if }n\geq 5,
\end{equation}
\begin{equation}\label{rate:sai:int:4}
\lim_{\alpha\to +\infty}\ea\sai^2\ln\frac{1}{\mai}=c_i\hbox{ if }n=4.
\end{equation}
and
\begin{equation}\label{d:s:bord}
\sai=o(d(\xai,\partial\Omega))
\end{equation}
when $\alpha\to +\infty$. Moreover, when $n\geq 7$, we have the following additional information:
\begin{equation}
\sai=o\left(\frac{\mai^{\frac{1}{2}}}{\bua^{\frac{1}{n-2}}}\right)\hbox{ when }\alpha\to +\infty,
\end{equation}
and there exists $j\in\{1,...,N\}$ such that $\mai=o(\maj)$ when $\alpha\to +\infty$ and
$$\sai=\left(\frac{\mai}{\maj}(\maj^2+|\xai-\xaj|^2)\right)^{\frac{1}{2}}$$
for all $\alpha\in\nn$.
\end{thm}

\medskip\noindent{\it Proof of Theorem \ref{th:asymp:int}:}

\medskip\noindent For $x\in\sai^{-1}(\Omega-\xai)$, we define
$$\vai(x):=\frac{\sai^{n-2}}{\mai^{\frac{n-2}{2}}}\ua(\xai+\sai x).$$

\medskip\noindent {\bf Step \ref{th:asymp:int}.1:} We claim that there exists $\delta>0$ such that $\vai$ is well defined on $B_\delta(0)$ and such that there exists $v_i\in C^2(B_\delta(0)\setminus \{0\})$ such that
\begin{equation}\label{lim:vai:vi}
\lim_{\alpha\to +\infty}\vai=v_i\hbox{ in }C^2_{loc}(B_{2\delta}(0)\setminus \{0\})
\end{equation}
where there exists $\lambda_i'>0$ and $\psi_i\in C^2(B_{2\delta}(0))$ such that $\Delta\psi_i=0$ and
\begin{equation}\label{vi:explicit}
v_i(x)=\frac{\lambda_i'}{|x|^{n-2}}+\psi_i(x)\hbox{ for all }x\in B_{2\delta}(0)\setminus \{0\}\hbox{ with }\psi_i(0)>0.
\end{equation}
We prove the claim. Indeed, since $\xai\not\in\partial\Omega$, it follows from the definition of $\sai$ that
\begin{equation}\label{min:dist:sai}
\frac{d(\xai,\partial\Omega)}{\sai}\geq 1
\end{equation}
for all $\alpha\in\nn$. In particular, $\vai$ is well defined on $B_{1/2}(0)$.\par
\smallskip\noindent Assume that $\lim_{\alpha\to +\infty}\frac{d(\xai,\partial\Omega)}{\sai}=+\infty$: then \eqref{lim:vai:vi} and \eqref{vi:explicit} are direct consequences of Proposition \ref{prop:cv:sai:1}.\par
\smallskip\noindent Assume that $\lim_{\alpha\to +\infty}\frac{d(\xai,\partial\Omega)}{\sai}=\rho\geq 0$:  it follows from \eqref{min:dist:sai} that $\rho\geq 1$ and that $\lim_{\alpha\to +\infty}\xai=x_0\in\partial\Omega$. Using that the chart $\varphi$ around $x_0$ is such that $d\varphi_0$ is an orthogonal transformation and that $\tua$ coincides with $\ua$ on $\Omega$, we get  \eqref{lim:vai:vi} and \eqref{vi:explicit} thanks to Proposition \ref{prop:cv:sai:2}.\par
\smallskip\noindent This proves the claim and therefore ends Step \ref{th:asymp:int}.1.

\medskip\noindent Taking $\delta>0$ smaller if needed, for any $j\in \{1,...,N\}$, we have that
\begin{equation}\label{ppty:delta}
\xaj-\xai\neq o(\sai)\hbox{ when }\alpha\to +\infty\,\Rightarrow \, |\xaj-\xai|\geq 2\delta\sai\hbox{ for all }\alpha\in\nn.
\end{equation}

\medskip\noindent{\bf Step \ref{th:asymp:int}.2:} Let $U$ be a  smooth bounded domain of $\rn$, let $x_0\in\rn$ be a point and let $u\in C^2(\overline{U})$. We claim that
\begin{eqnarray}\label{id:poho}
&&\int_U(x-x_0)^k\partial_k u \Delta u\, dx+\frac{n-2}{2}\int_U u\Delta u\, dx\\
&&=\int_{\partial U}\left((x-x_0,\nu)\frac{|\nabla u|^2}{2}-\partial_\nu u\left((x-x_0)^k\partial_k u+\frac{n-2}{2}u\right)\right)\, d\sigma\nonumber
\end{eqnarray}
We prove the claim. Indeed, this is the celebrated Pohozaev identity \cite{pohozaev}. We sketch a proof here for convenience for the reader. We have that
\begin{eqnarray*}
&&\int_U(x-x_0)^k\partial_k u \Delta u\, dx+\frac{n-2}{2}\int_U u\Delta u\, dx\\
&&=\int_U -\partial_j\partial_j u \left((x-x_0)^k\partial_k u+\frac{n-2}{2}u\right)\, dx\\
&&= \int_U \partial_j u\partial_j \left((x-x_0)^k\partial_k u+\frac{n-2}{2}u\right)\, dx -\int_{\partial U}\partial_\nu u\left((x-x_0)^k\partial_k u+\frac{n-2}{2}u\right)\, d\sigma\\
&&= \frac{n}{2}\int_U |\nabla u|^2\, dx+\frac{1}{2}\int_U(x-x_0)^k\partial_k|\nabla u|^2\, dx -\int_{\partial U}\partial_\nu u\left((x-x_0)^k\partial_k u+\frac{n-2}{2}u\right)\, d\sigma\\
&&=\int_U \partial_k \left((x-x_0)^k\frac{|\nabla u|^2}{2}\right)\, dx-\int_{\partial U}\partial_\nu u\left((x-x_0)^k\partial_k u+\frac{n-2}{2}u\right)\, d\sigma\\
&&=\int_{\partial U}\left((x-x_0,\nu)\frac{|\nabla u|^2}{2}-\partial_\nu u\left((x-x_0)^k\partial_k u+\frac{n-2}{2}u\right)\right)\, d\sigma.
\end{eqnarray*}
This proves \eqref{id:poho}, and therefore the claim. This ends Step \ref{th:asymp:int}.2.

\medskip\noindent As a consequence, differentiating \eqref{id:poho} with respect to $x_0$, we get that
\begin{equation}\label{id:poho:der}
\int_U\partial_k u \Delta u\, dx=\int_{\partial U}\left(\nu_k\frac{|\nabla u|^2}{2}-\partial_\nu u\partial_k u\right)\, d\sigma
\end{equation}

\medskip\noindent Taking $u:=\ua$, using equation \eqref{eq:ua} and integrating by parts, we get that
\begin{eqnarray}\label{id:poho:ua}
\ea\int_U\ua^2\, dx&=&\int_{\partial U}\left((x-x_0,\nu)\left(\frac{|\nabla \ua|^2}{2}-c_n\frac{\ua^{\crit}}{\crit}+\frac{\ea\ua^2}{2}\right)\right.\\
&&\left.-\partial_\nu \ua\left((x-x_0)^k\partial_k \ua+\frac{n-2}{2}\ua\right)\right)\, d\sigma\nonumber
\end{eqnarray}
where here and in the sequel, we define $c_n:=n(n-2)$. Taking $i\in\{1,...,N\}$ such that \eqref{lim:d:mai:infty} holds, and $\delta>0$ as in Step  \ref{th:asymp:int}.1, we let $U:=B_{\delta\sai}(\xai)\subset\subset \Omega$ and $x_0:=\xai$ in \eqref{id:poho:ua}. This yields
\begin{eqnarray}\label{id:poho:ua:int}
\ea\int_{B_{\delta\sai}(\xai)}\ua^2\, dx&=&\int_{\partial B_{\delta\sai}(\xai)}\left((x-\xai,\nu)\left(\frac{|\nabla \ua|^2}{2}-c_n\frac{\ua^{\crit}}{\crit}+\frac{\ea\ua^2}{2}\right)\right.\nonumber\\
&&\left.-\partial_\nu \ua\left((x-\xai)^k\partial_k \ua+\frac{n-2}{2}\ua\right)\right)\, d\sigma.
\end{eqnarray}

\medskip\noindent We now estimate the LHS and the RHS separately.

\medskip\noindent{\bf Step \ref{th:asymp:int}.3:} We claim that there exists $c>0$ such that
\begin{equation}\label{lhs:L2}
\int_{B_{\delta\sai}(\xai)}\ua^2\, dx=(c+o(1))\mai^2\cdot\left\{\begin{array}{ll}
1&\hbox{ if }n\geq 5\\
\ln\frac{\sai}{\mai}&\hbox{ if }n=4
\end{array}\right.
\end{equation}
when $\alpha\to +\infty$.\par
\medskip\noindent We prove the claim. We assume here that $n\geq 4$. It follows from \eqref{est:pointwise:V} and the estimate \eqref{ineq:Ua:Va} that
\begin{eqnarray}
\int_{B_{\delta\sai}(\xai)}\ua^2\, dx&\geq &C\int_{B_{\delta\sai}(\xai)}U_{i,\alpha}^2\, dx=C\mai^2\int_{B_{\delta\sai/\mai}(0)}\frac{1}{(1+|z|^2)^{n-2}}\, dz\nonumber\\
&\geq &C\mai^2\cdot\left\{\begin{array}{ll}
1&\hbox{ if }n\geq 5\\
\ln\frac{\sai}{\mai}&\hbox{ if }n=4
\end{array}\right.\label{lower:L2:ua}
\end{eqnarray}
for all $\alpha\in\nn$.\par

\smallskip\noindent We now deal with the upper estimate. With the upper bound \eqref{est:pointwise}, we get that
\begin{eqnarray}\label{upper:L2:ua}
&&\int_{B_{\delta\sai}(\xai)}\ua^2\, dx\\
&&\leq  C\int_{B_{\delta\sai}(\xai)}\bua^2\, dx+C\sum_{j=1}^N\int_{B_{\delta\sai}(\xai)}\left(\frac{\maj}{\maj^2+|x-\xaj|^2}\right)^{n-2}\, dx\nonumber
\end{eqnarray}
We deal with the different terms separately.

\medskip\noindent{\bf Step \ref{th:asymp:int}.3.1:} We claim that
\begin{equation}\label{est:L2:1}
\int_{B_{\delta\sai}(\xai)}\bua^2\, dx=O(\mai^2)\hbox{ when }n\geq 4
\end{equation}
when $\alpha\to +\infty$. We prove the claim. Indeed, with the definition \eqref{def:sai} of $\sai$, we have that
$$\int_{B_{\delta\sai}(\xai)}\bua^2\, dx=O(\sai^n\bua^2)=O(\mai^{\frac{n}{2}}\bua^{\frac{n-4}{n-2}})=o(\mai^2)$$
when $\alpha\to +\infty$ since $n\geq 4$. This proves \eqref{est:L2:1} and ends Step \ref{th:asymp:int}.3.1.

\medskip\noindent{\bf Step \ref{th:asymp:int}.3.2:} We let $j\in\{1,...,N\}$ such that
\begin{equation}\label{hyp:mai:maj:1}
\maj=O(\mai)
\end{equation}
when $\alpha\to +\infty$. We claim that
\begin{equation}\label{est:L2:2}
\int_{B_{\delta\sai}(\xai)}\left(\frac{\maj}{\maj^2+|x-\xaj|^2}\right)^{n-2}\, dx\leq C\mai^2\cdot\left\{\begin{array}{ll}
1&\hbox{ if }n\geq 5\\
\ln\frac{\sai}{\mai}&\hbox{ if }n=4
\end{array}\right.
\end{equation}
when $\alpha\to +\infty$. We first assume that $n\geq 5$. Estimating roughly the integral, we get with the change of variable $x=\xaj+\maj z$ and with \eqref{hyp:mai:maj:1} that
\begin{eqnarray*}
&&\int_{B_{\delta\sai}(\xai)}\left(\frac{\maj}{\maj^2+|x-\xaj|^2}\right)^{n-2}\, dx\leq \int_{\rn}\left(\frac{\maj}{\maj^2+|x-\xaj|^2}\right)^{n-2}\, dx\\
&&=\maj^2\int_{\rn}\frac{dz}{(1+|z|^2)^{n-2}}=O(\maj^2)=O(\mai^2)
\end{eqnarray*}
when $\alpha\to +\infty$ since $n\geq 5$. This proves \eqref{est:L2:2} when $n\geq 5$. When $n=4$, we must be a little more precise. Assume first that $\xai-\xaj=O(\sai)$ when $\alpha\to +\infty$. Then we have that
\begin{eqnarray*}
&&\int_{B_{\delta\sai}(\xai)}\left(\frac{\maj}{\maj^2+|x-\xaj|^2}\right)^{2}\, dx\leq \int_{B_{R\sai}(\xaj)}\left(\frac{\maj}{\maj^2+|x-\xaj|^2}\right)^{n-2}\, dx\\
&&=\maj^2\int_{B_{\delta\sai\mai^{-1}}(0)}\frac{dz}{(1+|z|^2)^{2}}=O\left(\maj^2\ln\frac{\sai}{\maj}\right)=O\left(\mai^2\ln\frac{\sai}{\mai}\right)
\end{eqnarray*}
when $\alpha\to +\infty$. Assume now that $\sai^{-1}|\xai-\xaj|\to +\infty$ when $\alpha\to +\infty$. Then for any $x\in B_{\delta\sai}(\xai)$, we have that $|x-\xaj|\geq \sai$ and then
\begin{eqnarray*}
&&\int_{B_{\delta\sai}(\xai)}\left(\frac{\maj}{\maj^2+|x-\xaj|^2}\right)^{2}\, dx\leq C\frac{\sai^4\maj^2}{\sai^4}=O(\maj^2)=O\left(\mai^2\ln\frac{\sai}{\mai}\right)
\end{eqnarray*}
when $\alpha\to +\infty$. These estimates prove \eqref{est:L2:2} in case $n=4$. This ends Step \ref{th:asymp:int}.3.2.

\medskip\noindent{\bf Step \ref{th:asymp:int}.3.3:} We let $j\in\{1,...,N\}$ such that
\begin{equation}\label{hyp:mai:maj:2}
\mai=o(\maj)\hbox{ and }\xai-\xaj\neq o(\sai)
\end{equation}
when $\alpha\to +\infty$. We claim that when $n\geq 4$, we have that
\begin{equation}\label{est:L2:3}
\int_{B_{\delta\sai}(\xai)}\left(\frac{\maj}{\maj^2+|x-\xaj|^2}\right)^{n-2}\, dx=O(\mai^2)
\end{equation}
when $\alpha\to +\infty$. We prove the claim. It follows from \eqref{hyp:mai:maj:2} and the definition \eqref{ppty:delta} of $\delta$ that $|\xai-\xaj|\geq 2\delta\sai$ for all $\alpha\in\nn$. In particular,
$$x\in B_{\delta\sai}(\xai)\,\Rightarrow\, |x-\xaj|\geq \frac{|\xai-\xaj|}{2}$$
and therefore
\begin{equation}\label{eq:step3.3:1}
\int_{B_{\delta\sai}(\xai)}\left(\frac{\maj}{\maj^2+|x-\xaj|^2}\right)^{n-2}\, dx=O\left(\sai^n\left(\frac{\maj}{\maj^2+|\xai-\xaj|^2}\right)^{n-2}\right)
\end{equation}
when $\alpha\to +\infty$. Moreover, it follows from \eqref{hyp:mai:maj:2} that $j\in J_i$, and then
\begin{equation}\label{eq:step3.3:2}
\sai^2\leq\frac{\mai}{\maj}(\maj^2+|\xai-\xaj|^2)
\end{equation}
for all $\alpha\in\nn$. It then follows from \eqref{eq:step3.3:1}, \eqref{eq:step3.3:2} and \eqref{lim:sai:mai} that
\begin{eqnarray*}
&&\int_{B_{\delta\sai}(\xai)}\left(\frac{\maj}{\maj^2+|x-\xaj|^2}\right)^{n-2}\, dx=O\left(\sai^n\left(\frac{\mai}{\sai^2}\right)^{n-2}\right)\\
&&=O\left(\mai^2\left(\frac{\mai}{\sai}\right)^{n-4}\right)=O(\mai^2)
\end{eqnarray*}
when $\alpha\to +\infty$ since $n\geq 4$. This proves \eqref{est:L2:3} and ends Step \ref{th:asymp:int}.3.3.

\medskip\noindent{\bf Step \ref{th:asymp:int}.3.4:} We let $j\in\{1,...,N\}$ such that
\begin{equation}\label{hyp:mai:maj:3}
\mai=o(\maj)\hbox{ and }\xai-\xaj= o(\sai)
\end{equation}
when $\alpha\to +\infty$. We claim that
\begin{equation}\label{est:L2:4}
\int_{B_{\delta\sai}(\xai)}\left(\frac{\maj}{\maj^2+|x-\xaj|^2}\right)^{n-2}\, dx=O(\mai^2)\hbox{ when }n\geq 4
\end{equation}
when $\alpha\to +\infty$. We prove the claim. As in Step 3.3, it follows from \eqref{hyp:mai:maj:3} that $j\in J_i$. In particular, using the definition \eqref{def:sai} of $\sai$ and the second assertion of \eqref{hyp:mai:maj:3}, we get that
\begin{equation*}
\sai^2\leq\frac{\mai}{\maj}(\maj^2+|\xai-\xaj|^2)\leq \mai\maj+o(\sai^2)
\end{equation*}
when $\alpha\to +\infty$, and then $\sai^2=O(\mai\maj)$ when $\alpha\to +\infty$. Consequently, we get that
\begin{eqnarray*}
&&\int_{B_{\delta\sai}(\xai)}\left(\frac{\maj}{\maj^2+|x-\xaj|^2}\right)^{n-2}\, dx=O\left(\frac{\sai^n}{\maj^{n-2}}\right)=O\left(\frac{\mai^{\frac{n}{2}}}{\maj^{\frac{n}{2}-2}}\right)\\
&&=O\left(\left(\frac{\mai}{\maj}\right)^{\frac{n}{2}-2}\mai^2\right)=O(\mai^2)
\end{eqnarray*}
when $\alpha\to +\infty$ since $n\geq 4$. This proves \eqref{est:L2:4} and ends Step \ref{th:asymp:int}.3.4.

\medskip\noindent Plugging together \eqref{est:L2:1}, \eqref{est:L2:2}, \eqref{est:L2:3} and \eqref{est:L2:4} into \eqref{upper:L2:ua} and combining this with \eqref{lower:L2:ua}, we get \eqref{lhs:L2}. This proves the claim and ends Step \ref{th:asymp:int}.3.

\medskip\noindent We define
\begin{eqnarray}
A_{i, \alpha}:&=&\int_{\partial B_{\delta\sai}(\xai)}\left((x-\xai,\nu)\left(\frac{|\nabla \ua|^2}{2}-c_n\frac{\ua^{\crit}}{\crit}+\frac{\ea\ua^2}{2}\right)\right.\nonumber\\
&&\left.-\partial_\nu \ua\left((x-\xai)^k\partial_k \ua+\frac{n-2}{2}\ua\right)\right)\, d\sigma.\label{def:Aai}
\end{eqnarray}
for all $\alpha\in\nn$.

\medskip\noindent{\bf Step \ref{th:asymp:int}.4:} Assume that $n\geq 3$. We claim  that
\begin{equation}\label{est:A}
A_{i, \alpha}=\left(\frac{(n-2)^2\omega_{n-1}\lambda_i'\psi_i(0)}{2}+o(1)\right)\cdot\left(\frac{\mai}{\sai}\right)^{n-2}
\end{equation}
when $\alpha\to +\infty$. Here, $\omega_{n-1}$ denotes the volume of the unit $(n-1)-$sphere of $\rn$.\par
\medskip\noindent We prove the claim. With the change of variable $x=\xai+\sai z$ and using the definition of $\vai$, we get that
\begin{eqnarray*}
A_{i, \alpha}&=&\left(\frac{\mai}{\sai}\right)^{n-2}\int_{\partial B_{\delta}(0)}\left((z,\nu)\left(\frac{|\nabla \vai|^2}{2}-c_n\left(\frac{\mai}{\sai}\right)^2\frac{\vai^{\crit}}{\crit}+\frac{\ea\sai^2\ua^2}{2}\right)\right.\\
&&\left.-\partial_\nu \vai\left( x^k\partial_k \vai+\frac{n-2}{2}\vai\right)\right)\, d\sigma
\end{eqnarray*}
for all $\alpha\in\nn$. Since $\vai\to v_i$ in $C^2_{loc}(B_{2\delta}(0)\setminus \{0\})$ when $\alpha\to +\infty$, passing to the limit, we get that
\begin{eqnarray}\label{est:A:1}
A_{i, \alpha}&=&\left(\frac{\mai}{\sai}\right)^{n-2}\left(\int_{\partial B_{\delta}(0)}\left((z,\nu)\left(\frac{|\nabla v_i|^2}{2}\right)\right.\right.\\
&&\left.\left.-\partial_\nu v_i\left( x^k\partial_k v_i+\frac{n-2}{2}v_i\right)\right)\, d\sigma +o(1)\right)\nonumber
\end{eqnarray}
when $\alpha\to +\infty$. We let $\epsilon\in (0,\delta)$ and we apply the Pohozaev identity \eqref{id:poho} to $v_i$ on $B_\delta(0)\setminus \overline{B}_\epsilon(0)$ with $x_0=0$. Since $\Delta v_i=0$, we get that the map
$$\epsilon\mapsto \int_{\partial B_{\epsilon}(0)}\left((z,\nu)\left(\frac{|\nabla v_i|^2}{2}\right)-\partial_\nu v_i\left( x^k\partial_k v_i+\frac{n-2}{2}v_i\right)\right)\, d\sigma$$
is constant on $(0,\delta]$. With the explicit expression \eqref{vi:explicit} of $v_i$, we have the asymptotic expansion
$$(z,\nu)\left(\frac{|\nabla v_i|^2}{2}\right)-\partial_\nu v_i\left( x^k\partial_k v_i+\frac{n-2}{2}v_i\right)=\frac{(n-2)^2 \lambda_i'\psi_i(0)}{2}|x|^{1-n}+O(|x|^{2-n})$$
when $|x|\to 0$. Consequently, we get that
$$\lim_{\epsilon\to 0}\int_{\partial B_{\epsilon}(0)}\left((z,\nu)\left(\frac{|\nabla v_i|^2}{2}\right)-\partial_\nu v_i\left( x^k\partial_k v_i+\frac{n-2}{2}v_i\right)\right)\, d\sigma=\frac{(n-2)^2 \lambda_i'\psi_i(0)\omega_{n-1}}{2},$$
and then
$$\int_{\partial B_{\delta}(0)}\left((z,\nu)\left(\frac{|\nabla v_i|^2}{2}\right)-\partial_\nu v_i\left( x^k\partial_k v_i+\frac{n-2}{2}v_i\right)\right)\, d\sigma=\frac{(n-2)^2 \lambda_i'\psi_i(0)\omega_{n-1}}{2}.$$
Plugging this equality in \eqref{est:A:1} yields \eqref{est:A}. This ends Step \ref{th:asymp:int}.4.

\medskip\noindent{\bf Step \ref{th:asymp:int}.5:} We claim that there exists $c_i>0$ such that
\begin{equation}\label{asymp:sai}
\lim_{\alpha\to +\infty}\frac{\ea\sai^{n-2}}{\mai^{n-4}}=c_i\hbox{ if }n\geq 5\hbox{ and }\lim_{\alpha\to +\infty}\ea\sai^{2}\ln\frac{1}{\mai}=c_i\hbox{ if }n=4.
\end{equation}
Indeed, plugging \eqref{lhs:L2} and \eqref{est:A} into \eqref{id:poho:ua:int} yields
$$(c+o(1))\ea\mai^2=\left(\frac{(n-2)^2\omega_{n-1}\lambda_i'\psi_i(0)}{2}+o(1)\right)\cdot\left(\frac{\mai}{\sai}\right)^{n-2}$$
when $\alpha\to +\infty$ when $n\geq 5$. Since $c,\lambda_i',\psi_i(0)>0$, we get that
$$\lim_{\alpha\to +\infty}\frac{\ea\sai^{n-2}}{\mai^{n-4}}=\frac{(n-2)^2\omega_{n-1}\lambda_i'\psi_i(0)}{2 c}>0.$$
This proves the claim when $n\geq 5$. The proof is similar when $n=4$.

\medskip\noindent{\bf Step \ref{th:asymp:int}.6:} we claim that
\begin{equation}\label{comp:sai:bua}
\sai=o\left(\frac{\mai^{\frac{1}{2}}}{\bua^{\frac{1}{n-2}}}\right)\hbox{ when }n\geq 7.
\end{equation}
when $\alpha\to +\infty$. We prove the claim by contradiction. Indeed, if \eqref{comp:sai:bua} does not hold, it follows from the definition \eqref{def:sai} of $\sai$ that
$$\sai\asymp \frac{\mai^{\frac{1}{2}}}{\bua^{\frac{1}{n-2}}}$$
when $\alpha\to +\infty$. Plugging this identity into \eqref{asymp:sai} yields
$$\ea\asymp \mai^{\frac{n-6}{2}}\bua$$
when $\alpha\to +\infty$. With \eqref{upp:bua}, we then get that
$$1=O\left(\mai^{\frac{n-6}{2}}\ea^{\frac{n-6}{4}}\right),$$
a contradiction since $n\geq 7$. Then \eqref{comp:sai:bua} holds and the claim is proved. This ends Step \ref{th:asymp:int}.6.

\medskip\noindent{\bf Step \ref{th:asymp:int}.7:} Assume that $n\geq 3$. We claim that
\begin{equation}\label{lim:d:sai}
\lim_{\alpha\to +\infty}\frac{d(\xai,\partial\Omega)}{\sai}=+\infty.
\end{equation}
We prove the claim. We argue by contradiction and we assume that
\begin{equation*}
\lim_{\alpha\to +\infty}\frac{d(\xai,\partial\Omega)}{\sai}=\rho\geq 0.
\end{equation*}
It follows from the definition \eqref{def:sai} of $\sai$ that $\rho\geq 1>0$. We adopt the notations of Proposition \ref{prop:cv:sai:2}. We let $j_0\in I_i'$ such that
\begin{equation}\label{def:j0}
\tilde{\theta}_{j_0,1}=\min_{j\in I_i'}\{\tilde{\theta}_{j,1}\}.
\end{equation}
Here, $\tilde{\theta}_{j,1}$ denotes the first coordinate of $\tilde{\theta}_j$.

\medskip\noindent{\bf Step \ref{th:asymp:int}.7.1:} We claim that there exists $\epsilon_0>0$ such that
\begin{equation}\label{lower:j0}
d(x_{j_0, \alpha},\partial\Omega)\geq \epsilon_0\sai
\end{equation}
for all $\alpha\in\nn$. We prove the claim by contradiction and we assume that $d(x_{j_0, \alpha},\partial\Omega)=o(\sai)$ when $\alpha\to +\infty$. In particular,  via the chart $\varphi$, we get that
$$\lim_{\alpha\to +\infty}\frac{(\varphi^{-1}(x_{j_0, \alpha}))_1}{\sai}=0\hbox{ and }\lim_{\alpha\to +\infty}\frac{(\varphi^{-1}(\xai))_1}{\sai}=-\rho<0.$$
Coming back to the definition \eqref{def:tthetaj} of $\tilde{\theta}_{j_0}$, we get that
$$\tilde{\theta}_{j_0,1}=\lim_{\alpha\to +\infty}\frac{(\varphi^{-1}(x_{j_0, \alpha})-\varphi^{-1}(\xai))_1}{\sai}=\rho>0.$$
A contradiction since $\tilde{\theta}_{j_0, 1}\leq \tilde{\theta}_{i,1}=0$. This proves \eqref{lower:j0} and ends Step \ref{th:asymp:int}.7.1.

\medskip\noindent{\bf Step \ref{th:asymp:int}.7.2:} We let $\delta_0>0$ such that
$$\delta_0<\frac{\epsilon_0}{2}\hbox{ and }\tilde{\theta}_j\neq \tilde{\theta}_{j_0}\,\Rightarrow \, |\tilde{\theta}_{j}-\tilde{\theta}_{j_0}|\geq 2\delta_0.$$
Taking the Pohozaev identity \eqref{id:poho:ua} with $U:=B_{\delta_0\sai}(x_{j_0, \alpha})\subset\subset\Omega$ and differentiating with respect to $x_0$, we get that
\begin{equation}\label{id:der:ua}
\int_{\partial B_{\delta_0\sai}(x_{j_0, \alpha})}\left(\nu_k\left(\frac{|\nabla \ua|^2}{2}-c_n\frac{\ua^{\crit}}{\crit}+\frac{\ea\ua^2}{2}\right)-\partial_\nu \ua\partial_k \ua)\right)\, d\sigma=0
\end{equation}
for all $\alpha\in\nn$ and all $k\in \{1,...,n\}$.  With the change of variable $x=\xai+\sai z$ and using the function $\vai$, we get that
\begin{equation}\label{id:der:vai}
\int_{\partial B_{\delta_0}(\tilde{\theta}_{j_0, \alpha})}\left(\nu_k\left(\frac{|\nabla \vai|^2}{2}-c_n\left(\frac{\mai}{\sai}\right)^2\frac{\vai^{\crit}}{\crit}+\frac{\ea\sai^2\vai^2}{2}\right)-\partial_\nu \vai\partial_k \vai)\right)\, d\sigma=0
\end{equation}
for all $\alpha\to 0$. Letting $\alpha\to 0$, we get with \eqref{lim:vai:vi} that
\begin{equation}\label{id:der:vi}
\int_{\partial B_{\delta_0}(\tilde{\theta}_{j_0})}\left(\nu_k\frac{|\nabla v_i|^2}{2}-\partial_\nu v_i\partial_k  v_i)\right)\, d\sigma=0
\end{equation}
for all $k\in\nn$. It follows from \eqref{explicit:tvai:prop} that
\begin{eqnarray*}
v_i(x)&=&K+\sum_{j\in I_i'}\lambda_j(|x-\theta_j|^{2-n}+|x-\sigma(\theta_j)|^{2-n}\\
&=&\frac{\lambda'_{i, j_0}}{|x-\theta_{j_0}|^{n-2}}+\psi_{i, j_0}(x)
\end{eqnarray*}
where $\lambda'_{i, j_0}>0$ and
$$\psi_{i, j_0}(x):=K+\lambda_{j_0}|x-\sigma(\theta_{j_0})|^{2-n}+\sum_{j\in I_i" }\lambda_j(|x-\theta_j|^{2-n}+|x-\sigma(\theta_j)|^{2-n}$$
where $I_i":=\{j\in I_i'/\, \theta_j\neq\theta_{j_0}\}$
Arguing as in Step \ref{th:asymp:int}.4, we get that \eqref{id:der:vi} holds on balls with arbitrary small positive radius and then we get that
$$\partial_k\psi_{i, j_0}(\theta_{j_0})=0.$$
Taking $k=1$, we get that
\begin{equation}\label{id:poho:der:0}
\lambda_{j_0}\frac{(\theta_{j_0}-\sigma(\theta_{j_0}))_1}{|\theta_{j_0}-\sigma(\theta_{j_0})|^n}+\sum_{j\in I_i"}\lambda_j\left(\frac{(\theta_{j_0}-\theta_{j})_1}{|\theta_{j_0}-\theta_{j}|^n}+\frac{(\theta_{j_0}-\sigma(\theta_{j}))_1}{|\theta_{j_0}-\sigma(\theta_{j})|^n}\right)=0.
\end{equation}
Recall that if $\theta_j=(\theta_{j, 1},\theta_j')$, then $\sigma(\theta_j)=(2\rho-\theta_{j, 1},\theta_j')$. In particular, since $x_{j,\alpha}\in\omegabar$, we have that $\theta_{j}\in\{x_1\leq \rho\}$ and then for all $j\in I_i"$, we have that
\begin{equation}\label{ineq:poho:1}
\theta_{j_0,1}\leq \theta_{j,1}\leq (\sigma(\theta_j))_1.
\end{equation}
In addition, we have that
\begin{equation}\label{ineq:poho:2}
(\theta_{j_0}-\sigma(\theta_{j_0}))_1=2(\theta_{j_0, 1}-\rho)=-2(|\theta_{j_0,1}|+\rho)<0.
\end{equation}
Plugging \eqref{ineq:poho:1} and \eqref{ineq:poho:2} into \eqref{id:poho:der:0} yields a contradiction. This proves that \eqref{lim:d:sai} holds. This ends Step \ref{th:asymp:int}.7.

\medskip\noindent{\bf Step \ref{th:asymp:int}.8:} We assume that $n\geq 3$. We claim that
\begin{equation}\label{d:xai:xaj:sai}
\xaj-\xai=o(\sai)\hbox{ when }\alpha\to +\infty\hbox{ for all }j\in I_i'.
\end{equation}
We prove the claim. Since \eqref{lim:d:sai} holds, we define $\vai$ and $v_i$ as in Proposition \ref{prop:cv:sai:1}. In particular, we have that
$$v_i(x)=K+\sum_{j\in I_i'}\lambda_j|x-\theta_j|^{2-n}$$
for all $x\in\rn\setminus \{\theta_j/\, j\in I_i\}$. We fix $k\in\{1,...,n\}$ and we let $j_0\in I_i'$ such that
$$\theta_{j_0,k}=\min \{\theta_{j, k}/\, j\in I_i'\}.$$
We let $I_i":=\{j\in I_i'/\, \theta_j\neq \theta_{j_0}\}$. Therefore, there exists $\lambda_{i, j_0}'>0$ such that
$$v_i(x)=\frac{\lambda_{i, j_0}'}{|x-\theta_{j_0}|^{n-2}}+\psi_{i, j_0}(x)$$
where
$$\psi_{i, j_0}(x):=K+\sum_{j\in I_i"}\lambda_j|x-\theta_j|^{2-n}.$$
Taking $\delta<\min\{|\theta_j|/\, \theta_j\neq \theta_{j_0}\}$, we use the identity \eqref{id:der:ua} as in Step \ref{th:asymp:int}.7. Performing the change of variable $x=\xai+\sai z$, we get again that
$$\partial_k \psi_{i, j_0}(\theta_{j_0})=0.$$
With the explicit expression of $\psi_{i, j_0}$, this yields
$$\sum_{j\in I_i"}\lambda_j\frac{(\theta_j-\theta_{j_0})_k}{|\theta_j-\theta_{j_0}|^n}=0.$$
Since $(\theta_j-\theta_{j_0})_k\geq 0$ for all $j\in I_i"$ by definition, we get that $\theta_{j,k}=\theta_{j_0,k}$ for all $j\in I_i"$, and therefore for all $j\in I_i'$. In particular, $\theta_{j, k}=\theta_{i, k}$ for all $k\in\nn$, and therefore $\theta_j=\theta_i=0$ for all $j\in I_i'$. Coming back to the definition \eqref{def:tthetaj} of $\theta_j$, we get that \eqref{d:xai:xaj:sai} holds. This ends the proof of the claim and of Step \ref{th:asymp:int}.8.

\medskip\noindent{\bf Step \ref{th:asymp:int}.9:} Assume that $n\geq 7$. We claim that there exists $j_0\in J_i$ such that
\begin{equation}\label{ppty:sai:j}
\sai=\left(\frac{\mai}{\maj}(\maj^2+|\xai-\xaj|^2)\right)^{\frac{1}{2}}\hbox{ and }\mai=o(\maj)
\end{equation}
when $\alpha\to +\infty$. We prove the claim. Indeed, it follows from the definition \eqref{def:sai} of $\sai$ and \eqref{comp:sai:bua} of Step \ref{th:asymp:int}.6 and \eqref{lim:d:sai} of Step \ref{th:asymp:int}.7 that there exists $j\in J_i$ such that
\begin{equation}\label{id:sai}
\sai=\left(\frac{\mai}{\maj}(\maj^2+|\xai-\xaj|^2)\right)^{\frac{1}{2}}
\end{equation}
for all $\alpha\in\nn$ (up to a subsequence, of course). Since $j\in J_i$, we have that $\mai=O(\maj)$ when $\alpha\to +\infty$. Assume that $\mai\asymp\maj$ when $\alpha\to +\infty$: then it follows from \eqref{id:sai} that $\xaj-\xai=O(\sai)$ when $\alpha\to +\infty$, and then $j\in I_i'$. It then follows from \eqref{d:xai:xaj:sai} of Step \ref{th:asymp:int}.8 that we have that $\xai-\xaj=o(\sai)$. Coming back to \eqref{id:sai}, we get that $\sai\asymp\mai$ when $\alpha\to +\infty$: a contradiction with \eqref{lim:sai:mai}. Therefore \eqref{ppty:sai:j} holds, and the claim is proved. This ends Step \ref{th:asymp:int}.9.

\medskip\noindent{\bf Step \ref{th:asymp:int}.10:} We assume that $n=3$. It follows from \eqref{est:A} and \eqref{id:poho:ua:int} that
\begin{equation}\label{asymp:3}
\int_{B_{\delta\sai}(\xai)}\ua^2\, dx\asymp\frac{\mai}{\sai}
\end{equation}
when $\alpha\to +\infty$. It follows from \eqref{est:pointwise:V} that
\begin{eqnarray*}
\int_{B_{\delta\sai}(\xai)}\ua^2\, dx&=&(1+o(1))\int_{B_{\delta\sai}(\xai)}\left(\bua+\sum_{j=1}^NV_{j,\alpha}(x)\right)^2\, dx\\
&\asymp& \sai^3\bua^2+\sum_{j=1}^N\maj\int_{B_{\delta\sai}(\xai)}(\maj^2+|x-\xaj|^2)^{-1}\, dx
\end{eqnarray*}
when $\alpha\to +\infty$. We distinguish three cases to get a contradiction.

\smallskip\noindent{\bf Step \ref{th:asymp:int}.10.1:} we assume that
\begin{equation}\label{3d:1}
\int_{B_{\delta\sai}(\xai)}\ua^2\, dx\asymp \sai^3\bua^2
\end{equation}
when $\alpha\to +\infty$. It then follows from \eqref{asymp:3} that $\ea\sai^4\bua^2\asymp\mai$ when $\alpha\to +\infty$. Moreover, since $\sai\leq \mai^{1/2}\bua^{-1}$ by the definition \eqref{def:sai}, we get that $\bua^2=o(\mai)$ when $\alpha\to +\infty$. This is a contradiction with \eqref{comp:ea:ma}. Then \eqref{3d:1} does not hold.

\smallskip\noindent{\bf Step \ref{th:asymp:int}.10.2:} we assume that there exists $j\in\{1,...,N\}$ such that $\sai=O(|\xai-\xaj|)$ and
\begin{equation}\label{3d:2}
\int_{B_{\delta\sai}(\xai)}\ua^2\, dx\asymp \maj\int_{B_{\delta\sai}(\xai)}(\maj^2+|x-\xaj|^2)^{-1}\, dx
\end{equation}
when $\alpha\to +\infty$. Here again, since $|x-\xaj|\asymp |\xai-\xaj|$ for all $x\in B_{\delta\sai}(\xai)$, it follows from \eqref{3d:2} and \eqref{asymp:3} that
\begin{equation}\label{3d:2:bis}
\frac{\ea\maj\sai^3}{\maj^2+|\xai-\xaj|^2}\asymp\frac{\mai}{\sai}
\end{equation}
when $\alpha\to +\infty$. In particular, since $\sai=O(|\xai-\xaj|)$, we get that $\mai=o(\maj)$ when $\alpha\to +\infty$, and then $j\in J_i$. Therefore, we have that
$$\sai^2\leq\frac{\mai}{\maj}(\maj^2+|\xai-\xaj|^2)$$
for all $\alpha\in\nn$, and it then follows from \eqref{3d:2:bis} that $1=O(\ea\sai^2)=o(1)$. A contradiction. Therefore, \eqref{3d:2} does not hold.

\smallskip\noindent{\bf Step \ref{th:asymp:int}.10.3:} we assume that there exists $j\in\{1,...,N\}$ such that $|\xai-\xaj|=o(\sai)$ and
\begin{equation}\label{3d:3}
\int_{B_{\delta\sai}(\xai)}\ua^2\, dx\asymp \maj\int_{B_{\delta\sai}(\xai)}(\maj^2+|x-\xaj|^2)^{-1}\, dx
\end{equation}
when $\alpha\to +\infty$. A change of variable then yields
$$\int_{B_{\delta\sai}(\xai)}\ua^2\, dx\asymp \maj\sai^3\int_{B_{\delta}\left(\frac{\xai-\xaj}{\sai}\right)}(\maj^2+\sai^2|z|^2)^{-1}$$
when $\alpha\to +\infty$. Therefore,
$$\int_{B_{\delta\sai}(\xai)}\ua^2\, dx\asymp \maj\sai^3 \max\{\maj,\sai\}^{-2}$$
when $\alpha\to +\infty$. It then follows from \eqref{asymp:3} that
\begin{equation}\label{eq:dim3}
\ea\maj\sai^4\asymp \mai\max\{\maj,\sai\}^2
\end{equation}
when $\alpha\to +\infty$. In particular, we have that $\mai=o(\maj)$, and then $j\in J_i$. Therefore, we have that
$$\sai^2\leq \frac{\mai}{\maj}(\maj^2+|\xai-\xaj|^2)\leq \mai\maj+o(\sai^2)$$
and then $\sai=O(\sqrt{\mai\maj})=o(\maj)$ when $\alpha\to +\infty$. Then \eqref{eq:dim3} becomes $\ea\sai^4\asymp\mai\maj$ when $\alpha\to +\infty$, a contradiction since $\sai^2=O(\mai\maj)$ when $\alpha\to +\infty$. Therefore, \eqref{3d:3} does not hold.

\medskip\noindent In all the situations, we have proved a contradiction. Therefore the hypothesis \eqref{lim:d:mai:infty} of Theorem \ref{th:asymp:int} does not hold in dimension $n=3$. This ends Step \ref{th:asymp:int}.10.

\medskip\noindent{\bf Step \ref{th:asymp:int}.10:} Theorem \ref{th:asymp:int} is a direct consequence of Steps \ref{th:asymp:int}.5, \ref{th:asymp:int}.6, \ref{th:asymp:int}.7, \ref{th:asymp:int}.8 and \ref{th:asymp:int}.10. This ends the proof of Theorem \ref{th:asymp:int}.\hfill$\Box$

\bigskip\noindent In the sequel, we need to translate slightly the boundary concentration points: we fix $\theta\in\rr^{n-1}$ and for all $i\in\{1,...,N\}$ such that $\xai\in\partial\Omega$, we define $\tilde{x}_{i,\alpha}:=\varphi(\varphi^{-1}(\xai)+\mai\theta)\in\partial\Omega$ for all $\alpha\in\nn$. The parameter $\theta$ is chosen such that there exists $\epsilon_0>0$ such that
\begin{equation}\label{zero:S}
|\tilde{x}_{i,\alpha}-\tilde{x}_{j,\alpha}|\geq \epsilon_0\mai
\end{equation}
for all $i,j\in\{1,...,N\}$ distincts such that $\tilde{x}_{i,\alpha},\tilde{x}_{j,\alpha}\in\partial\Omega$ and all $\alpha\in\nn$. We define $\tilde{s}_{i,\alpha}$ as $\sai$ with replacing $\xai$ by $\tilde{x}_{i,\alpha}$: as easily checked, for any $i\in\{1,...,N\}$ such that $\xai\in\partial\Omega$, we have that $\tilde{s}_{i,\alpha}\asymp\sai$ when $\alpha\to +\infty$. From now on, we replace $\xai$ by $\tilde{x}_{i,\alpha}$. As easily checked, the convergence Propositions \ref{prop:cv:sai:1} and \ref{prop:cv:sai:2} and the estimates \eqref{est:grad} and \eqref{est:pointwise} continue to hold with this new choice of points (with $\tau_i>0$ only in the propositions). Note that
the convergence \eqref{cv:ua:ma:2} of the $\tuai$ in Proposition \ref{prop:claimB} is changed as follows:
\begin{equation}\label{cv:ua:ma:2:bis}
\lim_{\alpha\to +\infty}\Vert\tuai-U_0(\cdot+\theta)\Vert_{C^1\left(K\cap \bar{\Om}_{i,\alpha}\right)}=0.
\end{equation}

\section{Estimates of the boundary blow-up rates}\label{sec:bnd}
In this section, we deal with the case when the concentration point is on the boundary.
\begin{thm}\label{th:asymp:bndy} Assume that $n\geq 3$. Let $i\in\{1,...,N\}$. We assume that
\begin{equation}\label{lim:d:mai:0}
\xai\in\partial\Omega
\end{equation}
for all $\alpha\in\nn$. We assume that for all $j\in\{1,...,N\}\setminus\{i\}$, we have that
\begin{equation}\label{ppty:cv:bord}
\xaj\in\partial\Omega\,\Rightarrow\, \xaj-\xai \neq o(\sai)\hbox{ when }\alpha\to +\infty
\end{equation}
when $\alpha\to +\infty$. Then there exists $c'_i>0$ such that
\begin{equation}
\begin{array}{ll}\label{rate:sai:bnd}
\lim_{\alpha\to +\infty}\frac{\mai^{n-3}}{\sai^{n-2}}=-c_i' H(x_0)&\hbox{ if }n\geq 4,\\
\lim_{\alpha\to +\infty}\frac{1}{\sai\ln\frac{1}{\mai}}=-c_i' H(x_0)&\hbox{ if }n=3,
\end{array}
\end{equation}
Where $x_0:=\lim_{\alpha\to +\infty}\xai$ and $H(x_0)$ denotes the mean curvature of $\partial\Omega$ at $x_0$. In particular, $H(x_0)\leq 0$.
\end{thm}

\medskip\noindent{\it Proof of Theorem \ref{th:asymp:bndy}:} As for Theorem \ref{th:asymp:int}, the proof relies on a Pohozaev identity. Here, we have to consider the boundary of $\Omega$. For any $\alpha\in\nn$, we define
\begin{equation}\label{def:Ualpha}
U_\alpha:=B_{\delta\sai}(\varphi^{-1}(\xai))
\end{equation}

\medskip\noindent{\bf Step \ref{th:asymp:bndy}.1:} we apply the Pohozaev identity \eqref{id:poho:ua} on $\varphi(U_\alpha)\cap\Omega=\varphi(U_\alpha\cap\rnm)$ with $x_0=\xai$. This yields
\begin{eqnarray}
&&\ea\int_{\varphi(U_\alpha\cap\rnm)}\ua^2\, dx=\int_{\partial \varphi(U_\alpha\cap\rnm)}F_\alpha\, d\sigma\label{id:poho:bord}\\
&&=\int_{\varphi((\partial U_\alpha)\cap\rnm)}F_\alpha\, d\sigma+\int_{\varphi(U_\alpha\cap\partial\rnm)}F_\alpha\, d\sigma\nonumber
\end{eqnarray}
where for convenience, we have defined
$$F_\alpha:=(\cdot-\xai,\nu)\left(\frac{|\nabla \ua|^2}{2}-c_n\frac{\ua^{\crit}}{\crit}+\frac{\ea\ua^2}{2}\right)-\partial_\nu \ua\left((\cdot-\xai)^k\partial_k \ua+\frac{n-2}{2}\ua\right)$$
for all $\alpha\in\nn$.

\medskip\noindent{\bf Step \ref{th:asymp:bndy}.2:} We deal with the LHS of \eqref{id:poho:bord}. We claim that
\begin{equation}\label{poho:bord:1}
\int_{\varphi(U_\alpha\cap\rnm)}\ua^2\, dx=\left\{\begin{array}{ll}
o(\mai)&\hbox{ if }n\geq 4\\
O(\mai)&\hbox{ if }n=3\end{array}\right.
\end{equation}
when $\alpha\to +\infty$. Indeed, the proof goes exactly as in the proof of \eqref{lhs:L2} of Step \ref{th:asymp:int}.3 of the proof of Theorem \ref{th:asymp:int}.

\medskip\noindent{\bf Step \ref{th:asymp:bndy}.3:} We deal with the first term of the RHS of \eqref{id:poho:bord}. When $n\geq 3$, we claim that there exists $c_i>0$ such that
\begin{equation}\label{poho:bord:2}
\int_{\varphi((\partial U_\alpha)\cap\rnm)}F_\alpha\, dx=\left(\frac{\mai}{\sai}\right)^{n-2}(c_i+o(1))
\end{equation}
when $\alpha\to +\infty$. \par
\smallskip\noindent We prove the claim.  The proof proceeds basically as in the proof of \eqref{est:A} of Step \ref{th:asymp:int}.4 of the proof of Theorem \ref{th:asymp:int}. Since $\xai\in\partial\Omega$, we have that $\lim_{\alpha\to +\infty}\xai=x_0\in\partial\Omega$. We take a domain $U_{x_0}$, a chart $\varphi$ and the extension $\tg$ of the metric and $\tua$ of $\ua$ as in Lemma \ref{lem:ext}. Therefore, there exists $\xa'\in\rr^{n-1}$ such that $\xai=\varphi(0,\xai')$ for all $\alpha\in\nn$ with $\lim_{\alpha\to +\infty}\xa'=0$. We define $\tvai$ as in Proposition \ref{prop:cv:sai:2}, that is
\begin{equation}\label{def:tvai:bord}
\tvai(x):=\frac{\sai^{n-2}}{\mai^{\frac{n-2}{2}}}\tua((0,\xai')+\sai x)
\end{equation}
for all $\alpha\in\nn$ and for all $x\in \sai^{-1}(\varphi^{-1}(U_{x_0})-(0, \xai'))$. Recall that it follows from Proposition \ref{prop:cv:sai:2} that there exists  $\tv_i\in C^2(B_{2\delta}(0)\setminus\{0\})$ such that
\begin{equation}\label{cv:tvai:tvi:bord}
\lim_{\alpha\to +\infty}\tvai=\tv_i\hbox{ in }C^1_{loc}(\rn\setminus\{0\}
\end{equation}
In addition, there exists $\tilde{\psi}_i\in C^2(B_\delta(0))$ harmonic such that
\begin{equation}\label{val:tvi}
\tv_i(x):=\frac{\lambda_i'}{|x|^{n-2}}+\tilde{\psi}_i(x)\hbox{ for all }x\in B_\delta(0)\setminus\{0\}\hbox{ with }\tilde{\psi}_i(0)>0.
\end{equation}
We define the metric $\tga(x):=(\varphi^\star \tg)((0,\xai')+\sai x)$ for all $x$. With the change of variable $x=\varphi((0,\xai')+\sai z)$, we get that
\begin{eqnarray*}
&&\int_{\varphi((\partial B_{\delta\sai}(\varphi^{-1}(\xai)))\cap\rnm)}F_\alpha\, dx \\
&&=\left(\frac{\mai}{\sai}\right)^{n-2}\int_{\partial B_{\delta}(0)\cap\rnm}\left((z,\nu)_{g_\alpha}\left(\frac{|\nabla \tvai|_{g_\alpha}^2}{2}-c_n\left(\frac{\mai}{\sai}\right)^2\frac{\tvai^{\crit}}{\crit}+\frac{\ea\sai^2\tvai^2}{2}\right)\right.\\
&&\left.-\partial_\nu \tvai\left( x^k\partial_k \tvai+\frac{n-2}{2}\tvai\right)\right)\, d\sigma_{\alpha}
\end{eqnarray*}
Passing to the limit $\alpha\to +\infty$ and using \eqref{cv:tvai:tvi:bord}, we get that
\begin{eqnarray*}
&&\int_{\varphi((\partial B_{\delta\sai}(\varphi^{-1}(\xai)))\cap\rnm)}F_\alpha\, dx\\
&&=\left(\frac{\mai}{\sai}\right)^{n-2}\left(\int_{\partial B_{\delta}(0)\cap\rnm}\left((z,\nu)\left(\frac{|\nabla v_i|^2}{2}\right)-\partial_\nu v_i\left( x^k\partial_k v_i+\frac{n-2}{2}v_i\right)\right)\, d\sigma +o(1)\right)\\
&&=\left(\frac{\mai}{\sai}\right)^{n-2}\left(\frac{1}{2}\int_{\partial B_{\delta}(0)}\left((z,\nu)\left(\frac{|\nabla v_i|^2}{2}\right)-\partial_\nu v_i\left( x^k\partial_k v_i+\frac{n-2}{2}v_i\right)\right)\, d\sigma +o(1)\right)
\end{eqnarray*}
when $\alpha\to +\infty$. Similarly to what was done in the proof of \eqref{est:A} of Step \ref{th:asymp:int}.4 in the proof of Theorem \ref{th:asymp:int}, and using \eqref{val:tvi}, we get that
\begin{equation*}
\int_{\varphi((\partial B_{\delta\sai}(\varphi^{-1}(\xai)))\cap\rnm)}F_\alpha\, dx=\left(\frac{\mai}{\sai}\right)^{n-2}\left(\frac{(n-2)^2 c_i\psi_i(0)\omega_{n-1}}{4}+o(1)\right)
\end{equation*}

\medskip\noindent This proves  \eqref{poho:bord:2} and ends Step \ref{th:asymp:bndy}.3.

\medskip\noindent We define
\begin{equation*}
L:=\{j\in \{1,...,N\}/\, \xaj-\xai=O(\mai)\hbox{ when }\alpha\to +\infty\}.
\end{equation*}
Given $R>0$ and $\alpha\in\nn$, we define
\begin{equation}\label{def:DRa}
{\mathcal D}_{R, \alpha}:=\varphi\left(B_{R\mai}(\varphi^{-1}(\xai))\setminus \bigcup_{k\in L}B_{R^{-1}\mai}(\varphi^{-1}(\xak))\cap\rnm\right).
\end{equation}

\medskip\noindent{\bf Step \ref{th:asymp:bndy}.4:} Assume that $n\geq 4$. We claim that
\begin{equation}\label{lim:bnd:Rc}
\lim_{R\to +\infty}\lim_{\alpha\to +\infty}\mai^{-1}\int_{\varphi(U_\alpha\cap\partial\rnm)\setminus {\mathcal D}_{R, \alpha}}(x-\xai,\nu)\left(\frac{|\nabla \ua|^2}{2}-c_n\frac{\ua^{\crit}}{\crit}+\frac{\ea\ua^2}{2}\right)\, d\sigma=0
\end{equation}
\smallskip\noindent We prove the claim. Indeed, it follows from \eqref{est:grad}  that
\begin{equation}\label{dom:integrand}
\left|\frac{|\nabla \ua|^2}{2}-c_n\frac{\ua^{\crit}}{\crit}+\frac{\ea\ua^2}{2}\right|\leq C\bua^2+ C\sum_{j=1}^N\frac{\maj^{n-2}}{\left(\maj^2+|x-\xaj|^2\right)^{n-1}}
\end{equation}
for all $x\in\Omega$ and all $\alpha\in\nn$.

\medskip\noindent{\bf Step \ref{th:asymp:bndy}.4.1:} We claim that
\begin{equation}\label{dom:ps:quadra}
|(x-\xai,\nu(x))|\leq C|x-\xai|^2
\end{equation}
for all $\alpha\in\nn$ and all $x\in\partial\Omega\cap\partial U_{x_0}$. We prove the claim. Indeed, for $x\in\rnm$ small enough, we get via the chart $\varphi$ that
\begin{eqnarray}
&&(\cdot-\xai,\nu)\circ\varphi((0,\xai')+x)\label{dl:ps}\\
&&= \left(\varphi((0, \xai')+x)-\varphi(0,\xai'), \nu\circ\varphi((0,\xai')+x)\right)\nonumber\\
&&= \left(d\varphi_{(0, \xai')}(x)+\frac{1}{2}d^2\varphi_{(0,\xai')}(x,x)+O(|x|^3), \nu\circ\varphi((0,\xai')+x)\right)\nonumber\\
&&=-\frac{1}{2}\left(d^2\varphi_{(0,\xai')}(x,x), \nu\circ\varphi(0,\xai')\right)+O(|x|^3).\nonumber
\end{eqnarray}
Inequality \eqref{dom:ps:quadra} is a straightforward consequence of \eqref{dl:ps}. This proves \eqref{dom:ps:quadra} and ends Step \ref{th:asymp:bndy}.4.1.

\medskip\noindent As a consequence of \eqref{dom:integrand} and \eqref{dom:ps:quadra}, we have that
\begin{eqnarray}\label{inter:Rc}
&&\left|\int_{\varphi(U_\alpha\cap\partial\rnm\setminus {\mathcal D}_{R, \alpha})}(x-\xai,\nu)\left(\frac{|\nabla \ua|^2}{2}-\frac{\ua^{\crit}}{\crit}+\frac{\ea\ua^2}{2}\right)\, d\sigma\right|\\
&&\leq C\int_{\varphi(U_\alpha\cap\partial\rnm\setminus {\mathcal D}_{R, \alpha})}|x-\xai|^2\bua^2\, d\sigma \nonumber\\
&&+C\sum_{j=1}^N \int_{\varphi(U_\alpha\cap\partial\rnm\setminus {\mathcal D}_{R, \alpha})}\frac{|x-\xai|^2\maj^{n-2}}{\left(\maj^2+|x-\xaj|^2\right)^{n-1}}\, d\sigma\nonumber
\end{eqnarray}
for all $\alpha\in\nn$ and all $R>0$. We are going to estimate these terms separately.

\medskip\noindent{\bf Step \ref{th:asymp:bndy}.4.2:} We claim that
\begin{equation}\label{est:bnd:Rc:1}
\int_{\varphi(U_\alpha\cap\partial\rnm\setminus {\mathcal D}_{R, \alpha})}|x-\xai|^2\bua^2\, d\sigma=o(\mai)\hbox{ when }\alpha\to +\infty
\end{equation}
We prove the claim. Indeed, using the definition \eqref{def:sai} of $\sai$, we get that
$$\int_{\varphi(U_\alpha\cap\partial\rnm\setminus {\mathcal D}_{R, \alpha})}|x-\xai|^2\bua^2\, d\sigma\leq C\sai^{n+1}\bua^2=O(\mai^{\frac{n+1}{2}}\bua^{\frac{n-5}{n-2}})$$
when $\alpha\to +\infty$. Moreover, since $\mai^{\frac{n-2}{2}}=o(\bua)$, we get that the above expression is $o(\mai\bua^{\frac{n-3}{n-2}})=o(\mai)$ when $\alpha\to +\infty$ since $n\geq 3$. This proves \eqref{est:bnd:Rc:1} and ends Step \ref{th:asymp:bndy}.4.2.

\medskip\noindent{\bf Step \ref{th:asymp:bndy}.4.3:} We claim that
\begin{equation}\label{est:bnd:Rc:2}
\lim_{R\to +\infty}\lim_{\alpha\to +\infty}\mai^{-1}\int_{\varphi(U_\alpha\cap\partial\rnm\setminus {\mathcal D}_{R, \alpha})}\frac{|x-\xai|^2\mai^{n-2}}{\left(\mai^2+|x-\xai|^2\right)^{n-1}}\, d\sigma=0\hbox{ if }n\geq 4.
\end{equation}
We prove the claim. Recall that for convenience, we let $\rr^{n-1}:=\partial\rnm$. Noting that
$$\varphi(U_\alpha\cap\partial\rnm\setminus {\mathcal D}_{R, \alpha})\subset \rr^{n-1}\setminus \left(B_{R\mai}(\xai)\setminus \cup_{k\in L}B_{R^{-1}\mai}(\xak)\right),$$
we get with the change of variables $x=\xai+\mai z$ that
\begin{eqnarray*}
&&\int_{\varphi(U_\alpha\cap\partial\rnm)\setminus {\mathcal D}_{R, \alpha}}\frac{|x-\xai|^2\mai^{n-2}}{\left(\mai^2+|x-\xai|^2\right)^{n-1}}\, d\sigma\\
&&\leq C\mai\int_{\rr^{n-1}\setminus \left(B_{R}(0)\setminus \cup_{k\in L}B_{R^{-1}}(\theta_{k,\alpha})\right)}\frac{|z|^2\, dz}{(1+|z|^2)^{n-1}}
\end{eqnarray*}
where $\theta_{k, \alpha}:=\mai^{-1}(\xak-\xai)$ for all $\alpha\in\nn$ and all $k\in L$. Letting $\theta_k:=\lim_{\alpha\to +\infty}\theta_{k, \alpha}$, we get that
\begin{eqnarray*}
&&\limsup_{\alpha\to +\infty}\mai^{-1}\int_{\varphi(U_\alpha\cap\partial\rnm\setminus {\mathcal D}_{R, \alpha})}\frac{|x-\xai|^2\mai^{n-2}}{\left(\mai^2+|x-\xai|^2\right)^{n-1}}\, d\sigma\\
&&\leq C\int_{\rr^{n-1}\setminus \left(B_{R}(0)\setminus \cup_{k\in L}B_{R^{-1}}(\theta_{k})\right)}\frac{|z|^2\, dz}{(1+|z|^2)^{n-1}}
\end{eqnarray*}
for all $R>0$. Then, letting $R\to +\infty$ and using that $n\geq 4$, we get \eqref{est:bnd:Rc:2}. This ends Step \ref{th:asymp:bndy}.4.3.

\medskip\noindent{\bf Step \ref{th:asymp:bndy}.4.4:} Let $j\in\{1,...,N\}$ such that
\begin{equation*}
\xaj-\xai\neq o(\sai)\hbox{ when }\alpha\to +\infty.
\end{equation*}
Then
\begin{equation}\label{est:bnd:Rc:3}
\int_{\varphi(U_\alpha\cap\partial\rnm)\setminus {\mathcal D}_{R, \alpha}}\frac{|x-\xai|^2\maj^{n-2}}{\left(\maj^2+|x-\xaj|^2\right)^{n-1}}\, d\sigma=\left\{\begin{array}{ll}
o(\mai)&\hbox{ if }n\geq 4\\
O(\mai)&\hbox{ if }n=3
\end{array}\right.
\end{equation}
when $\alpha\to +\infty$. We prove the claim. Taking $\delta>0$ smaller if necessary,  we have that
\begin{equation}\label{inf:xai:xaj}
|\xaj-\xai|\geq 2\delta\sai
\end{equation}
for all $\alpha\in\nn$. In particular, for all $x\in {\mathcal D}_{R, \alpha}\subset\varphi(U_\alpha\cap\rnm)$, we have that
$$|x-\xaj|\geq\delta\sai.$$
Therefore, we have that
\begin{equation}\label{bnd:j:1}
\int_{\varphi(U_\alpha\cap\partial\rnm\setminus {\mathcal D}_{R, \alpha})}\frac{|x-\xai|^2\maj^{n-2}}{\left(\maj^2+|x-\xaj|^2\right)^{n-1}}\, d\sigma=O\left(\frac{\sai^{n+1}\maj^{n-2}}{\left(\maj^2+|\xaj-\xai|^2\right)^{n-1}}\right)
\end{equation}
for all $\alpha\in\nn$. We distinguish two cases:

\medskip\noindent{\bf Case \ref{th:asymp:bndy}.4.4.1:} assume that $\maj=o(\mai)$ when $\alpha\to +\infty$. Then it follows from \eqref{inf:xai:xaj} and \eqref{bnd:j:1} that
\begin{eqnarray*}
&&\int_{\varphi(U_\alpha\cap\partial\rnm\setminus {\mathcal D}_{R, \alpha})}\frac{|x-\xai|^2\maj^{n-2}}{\left(\maj^2+|x-\xaj|^2\right)^{n-1}}\, d\sigma\\
&&=O\left(\frac{\sai^{n+1}\maj^{n-2}}{\sai^{2n-2}}\right)=o\left(\left(\frac{\mai}{\sai}\right)^{n-3}\mai\right)=o(\mai)
\end{eqnarray*}
when $\alpha\to +\infty$. This proves \eqref{est:bnd:Rc:3} in Case  \ref{th:asymp:bndy}.4.4.1.

\medskip\noindent{\bf Case \ref{th:asymp:bndy}.4.4.2}: assume that $\mai=O(\maj)$ when $\alpha\to +\infty$. Then, we have that $j\in J_i$ and it follows from the definition \eqref{def:sai} of $\sai$ that
$$\sai^2\leq\frac{\mai}{\maj}(\maj^2+|\xai-\xaj|^2)$$
for all $\alpha\in\nn$. Plugging this inequality in \eqref{bnd:j:1}, we get that\begin{eqnarray}
&&\int_{\varphi(U_\alpha\cap\partial\rnm\setminus {\mathcal D}_{R, \alpha})}\frac{|x-\xai|^2\maj^{n-2}}{\left(\maj^2+|x-\xaj|^2\right)^{n-1}}\, d\sigma=O\left(\frac{\mai^{n-1}}{\sai^{n-3}\maj}\right)\\
&&=O\left(\frac{\mai}{\maj}\cdot\left(\frac{\mai}{\sai}\right)^{n-3}\mai\right)=\left\{\begin{array}{ll}
o(\mai)&\hbox{ if }n\geq 4\\
O(\mai)&\hbox{ if }n=3
\end{array}\right.
\end{eqnarray}
when $\alpha\to +\infty$. This proves \eqref{est:bnd:Rc:3} in Case \ref{th:asymp:bndy}.4.4.2.

\smallskip\noindent We have proved \eqref{est:bnd:Rc:3} in all cases. This ends Step \ref{th:asymp:bndy}.4.4.

\medskip\noindent{\bf Step \ref{th:asymp:bndy}.4.5:} Let $j\in\{1,...,N\}$ such that
\begin{equation}\label{hyp:5.5}
\xaj\in\Omega \hbox{ and }\xaj-\xai=o(\sai)\hbox{ when }\alpha\to +\infty.
\end{equation}
Then we claim that
$$\maj=o(\mai)\hbox{ when }\alpha\to +\infty.$$
We prove the claim by contradiction and we assume that $\mai=O(\maj)$ when $\alpha\to +\infty$. Then $j\in J_i$ and it follows from the definition \eqref{def:sai} of $\sai$ that
$$\sai^2\leq\frac{\mai}{\maj}(\maj^2+|\xai-\xaj|^2)$$
for all $\alpha\to +\infty$. It then follows from \eqref{hyp:5.5} that $\sai^2=O(\mai\maj)=O(\maj^2)$ when $\alpha\to +\infty$. It then follows from \eqref{hyp:5.5} that $\xai-\xaj=o(\maj)$ when $\alpha\to +\infty$. Since $\xai\in\partial\Omega$, we then get that $d(\xaj,\partial\Omega)=o(\maj)$ when $\alpha\to +\infty$, and then $\xaj\in\partial\Omega$ (see (i) of Proposition \ref{prop:claimB}): a contradiction with our assumption \eqref{hyp:5.5}. This proves the claim.

\medskip\noindent{\bf Step \ref{th:asymp:bndy}.4.6:} Let $j\in\{1,...,N\}$ such that
\begin{equation}\label{hyp:5.6}
\xaj\not\in\partial\Omega\hbox{ and }\xaj-\xai=o(\sai)\hbox{ when }\alpha\to +\infty.
\end{equation}
We claim that
\begin{equation}\label{est:bnd:Rc:4}
\lim_{\alpha\to +\infty}\mai^{-1}\int_{\varphi(U_\alpha\cap\partial\rnm)\setminus {\mathcal D}_{R, \alpha}}\frac{|x-\xai|^2\maj^{n-2}}{\left(\maj^2+|x-\xaj|^2\right)^{n-1}}\, d\sigma=0\hbox{ when }n\geq 3.
\end{equation}
We prove the claim. Since $\lim_{\alpha\to +\infty}\xaj=x_0$, we write $\xaj=\varphi(x_{j, \alpha,1},\xaj')$ for all $\alpha\in\nn$. Here again, since $d\varphi_0$ is an orthogonal transformation (see Lemma \ref{lem:ext}, we get that
$$d(\xaj,\partial\Omega)=(1+o(1))|x_{j, \alpha,1}|$$
when $\alpha\to +\infty$. For simplicity, we let $d_{j, \alpha}:=d(\xaj,\partial\Omega)$ for all $\alpha\in\nn$. With the change of variables $x=\varphi(z)$, we get that
\begin{eqnarray*}
&&\int_{\varphi(U_\alpha\cap\partial\rnm\setminus {\mathcal D}_{R, \alpha})}\frac{|x-\xai|^2\maj^{n-2}}{\left(\maj^2+|x-\xaj|^2\right)^{n-1}}\, d\sigma\\
&&\leq C\int_{\partial\rnm}\frac{\maj^{n-2}}{\left(\maj^2+|(0,z)-(x_{j, \alpha,1},\xaj')|^2\right)^{n-1}}\, dz\\
&&\leq C\int_{\partial\rnm}\frac{\maj^{n-2}}{\left(d_{j, \alpha}^2+|z|^2\right)^{n-1}}\, dz\leq C\frac{\maj^{n-2}}{d_{j, \alpha}^{n-1}}
\end{eqnarray*}
for all $\alpha\in\nn$. Since $\xaj\not\in\partial\Omega$, we apply \eqref{rate:sai:int} and \eqref{d:s:bord} and we get that
\begin{eqnarray*}
&&\int_{\varphi(U_\alpha\cap\partial\rnm\setminus {\mathcal D}_{R, \alpha})}\frac{|x-\xai|^2\maj^{n-2}}{\left(\maj^2+|x-\xaj|^2\right)^{n-1}}\, d\sigma=O\left(\frac{\maj^{n-2}}{d_{j, \alpha}^{n-1}}\right)=o\left(\frac{\maj^{n-2}}{\saj^{n-1}}\right)\\
&&=o\left(\frac{\maj^{n-2}}{\maj^{\frac{(n-1)(n-4)}{n-2}}}\right)=o(\maj^{\frac{n}{n-2}})=o(\maj)=o(\mai)
\end{eqnarray*}
when $\alpha\to +\infty$, where we have used Step \ref{th:asymp:bndy}.4.5. This proves \eqref{est:bnd:Rc:1} and ends Step \ref{th:asymp:bndy}.4.6.

\medskip\noindent{\bf Step \ref{th:asymp:bndy}.4.7:} Plugging  \eqref{est:bnd:Rc:1}, \eqref{est:bnd:Rc:2}, \eqref{est:bnd:Rc:3} and \eqref{est:bnd:Rc:4} into \eqref{inter:Rc}, we get \eqref{lim:bnd:Rc}. This ends Step \ref{th:asymp:bndy}.4.

\medskip\noindent{\bf Step \ref{th:asymp:bndy}.5:} We claim that there exists $c>0$ such that
\begin{eqnarray}\label{lim:bnd:R}
&&\lim_{R\to +\infty}\lim_{\alpha\to +\infty}\mai^{-1}\int_{\varphi(U_\alpha\cap\partial\rnm)\cap{\mathcal D}_{R,\alpha}}(x-\xai,\nu)\left(\frac{|\nabla \ua|^2}{2}-c_n\frac{\ua^{\crit}}{\crit}+\frac{\ea\ua^2}{2}\right)\, d\sigma\nonumber\\
&&=\frac{H(x_0)}{2n}\int_{\partial\rnm}|x|^2\left(\frac{|\nabla U_0|_{\xi}^2}{2}-c_n\frac{U_0^{\crit}}{\crit}\right)\, d\xi\nonumber\\
&&-\frac{(\partial_{ij}\varphi(0),\nu(x_0))}{2}\theta^i\theta^j\int_{\partial\rnm}\left(\frac{|\nabla U_0|_{\xi}^2}{2}-n(n-2)\frac{U_0^{\crit}}{\crit}\right)\, d\xi\hbox{ when }n\geq 4,
\end{eqnarray}
where $\theta$ is as in \eqref{cv:ua:ma:2:bis}. We prove the claim. We assume that $n\geq 4$. As a preliminary remark, using the definition \eqref{def:DRa} of ${\mathcal D}_{R, \alpha}$ and \eqref{d:s:bord}, note that
$$\varphi(U_\alpha\cap\partial\rnm)\cap{\mathcal D}_{R,\alpha}={\mathcal D}_{R,\alpha}\cap\partial\Omega$$
for all $\alpha\in\nn$. We define
$$\tuai(x):=\mai^{\frac{n-2}{2}}\ua\circ\varphi((0,\xai')+\mai x)$$
for all $x\in\mai^{-1}(\varphi^{-1}(U_{x_0})-\varphi^{-1}(\xai))\cap\overline{\rnm}$. It follows from Theorem \ref{cv:ua:ma:2} modified by \eqref{cv:ua:ma:2:bis} that
\begin{equation}\label{cv:tuai:U}
\lim_{\alpha\to +\infty}
\tuai=U_0(\cdot+\theta)\hbox{ in }C^1_{loc}\left(\overline{\rnm}\setminus\{\theta_k/\, k\in L\}\right).
\end{equation}
where $U_0(x):=(1+|x|^{2})^{1-\frac{n}{2}}$ for all $x\in\rn$ and $\theta\in\rr^{n-1}$. With the change of variable $x=\varphi((0,\xai')+\mai z)$, we get that
\begin{eqnarray*}
&&\int_{{\mathcal D}_{R,\alpha}\cap\partial\Omega}(x-\xai,\nu)\left(\frac{|\nabla \ua|^2}{2}-c_n\frac{\ua^{\crit}}{\crit}+\frac{\ea\ua^2}{2}\right)\, d\sigma\\
&&=\mai\int_{B_{R}(0)\cap\partial\rnm\setminus \cup_{k\in L}B_{R^{-1}}(\theta_{k, \alpha}))}G_\alpha(z)\left(\frac{|\nabla \tuai|_{g_\alpha}^2}{2}-c_n\frac{\tuai^{\crit}}{\crit}+\frac{\ea\mai^2\tuai^2}{2}\right)\, d\sigma_\alpha
\end{eqnarray*}
where $g_\alpha(x):=(\varphi^\star\xi)((0,\xai')+\mai x)$ is the pull-back of $\xi$ by the chart $\varphi$ and $\sigma_\alpha$ is the surface area associated to the metric $g_\alpha$ and
$$G_\alpha(z):=\frac{(\varphi((0,\xai')+\mai z)-\varphi(0,\xai'),\nu\circ\varphi((0,\xai')+\mai z))}{\mai^2}.$$
With \eqref{dl:ps}, \eqref{cv:tuai:U}, using that $\varphi^\star\xi(0)=\xi$ and that $n\geq 4$, we get that
\begin{eqnarray}
&&\lim_{R\to +\infty}\lim_{\alpha\to +\infty}\mai^{-1}\int_{{\mathcal D}_{R, \alpha}\cap\Omega}(x-\xai,\nu)\left(\frac{|\nabla \ua|^2}{2}-n(n-2)\frac{\ua^{\crit}}{\crit}+\frac{\ea\ua^2}{2}\right)\, d\sigma\label{lim:bnd:1}\\
&&=\lim_{R\to +\infty}\int_{{\mathcal D}_R}\frac{-(\partial_{kl}\varphi(0), \nu(x_0))}{2}x^kx^l\left(\frac{|\nabla U_0|_{\xi}^2(x+\theta)}{2}-c_n\frac{U_0^{\crit}(x+\theta)}{\crit}\right)\, d\xi\nonumber\\
&&=\int_{\partial\rnm}\frac{-(\partial_{kl}\varphi(0),\nu(x_0))}{2}(x-\theta)^k(x-\theta)^l\left(\frac{|\nabla U_0|_{\xi}^2(x)}{2}-c_n\frac{U_0^{\crit}(x)}{\crit}\right)\, d\xi\nonumber\\
&&=\int_{\partial\rnm}\frac{-(\partial_{kl}\varphi(0),\nu(x_0))}{2}x^kx^l\left(\frac{|\nabla U_0|_{\xi}^2}{2}-c_n\frac{U_0^{\crit}}{\crit}\right)\, d\xi\nonumber\\
&&-\frac{(\partial_{kl}\varphi(0),\nu(x_0))}{2}\theta^k\theta^l\int_{\partial\rnm}\left(\frac{|\nabla U_0|_{\xi}^2}{2}-c_n\frac{U_0^{\crit}}{\crit}\right)\, d\xi
\end{eqnarray}
where in these computations, we have defined
$${\mathcal D}_R:=B_{R}(0)\cap\partial\rnm\setminus \cup_{k\in L}B_{R^{-1}}(\theta_{k})).$$
We let ${\mathcal A}(\varphi,\theta,x_0)$ be the right-hand-side of this expression. Since $U_0$ is radially symmetrical, we get that
\begin{eqnarray*}
&&{\mathcal A}(\varphi,\theta,x_0)=\frac{-\sum_k (\partial_{kk}\varphi(0),\nu(x_0))}{2n}\int_{\partial\rnm}|x|^2\left(\frac{|\nabla U_0|_{\xi}^2}{2}-c_n\frac{U_0^{\crit}}{\crit}\right)\, d\xi\\
&&-\frac{(\partial_{kl}\varphi(0),\nu(x_0))}{2}\theta^k\theta^l\int_{\partial\rnm}\left(\frac{|\nabla U_0|_{\xi}^2}{2}-c_n\frac{U_0^{\crit}}{\crit}\right)\, d\xi
\end{eqnarray*}
Since $d\varphi_0$ is an orthogonal transformation, the first and second fundamental formes of $\partial\Omega$ at $x_0$ in the chart $\varphi$ are respectiveley $I_{kl}=\delta_{kl}$ and $II_{kl}=-(\partial_{kl}\varphi(0),\nu(x_0))$. Therefore the mean curvature of $\partial\Omega$ at $x_0$ is $H(x_0)=\sum_kII_{kk}$ and then
\begin{eqnarray*}\label{id:H}
&&\int_{\partial\rnm}\frac{-\partial_{kl}\varphi(0)}{2}x^kx^l\left(\frac{|\nabla U_0|_{\xi}^2}{2}-c_n\frac{U_0^{\crit}}{\crit}\right)\, d\xi\\
&&=\frac{H(x_0)}{2n}\int_{\partial\rnm}|x|^2\left(\frac{|\nabla U_0|_{\xi}^2}{2}-c_n\frac{U_0^{\crit}}{\crit}\right)\, d\xi.
\end{eqnarray*}
Combining \eqref{lim:bnd:1} and \eqref{id:H} yields \eqref{lim:bnd:R}. This ends Step \ref{th:asymp:bndy}.6.

\medskip\noindent{\bf Step \ref{th:asymp:bndy}.6:} we claim that
\begin{equation}\label{integ:positive}
\left\{\begin{array}{ll}
\int_{\partial\rnm}|x|^2\left(\frac{|\nabla U_0|_{\xi}^2}{2}-c_n\frac{U_0^{\crit}}{\crit}\right)\, d\xi>0\\
\hbox{ and }\int_{\partial\rnm}\left(\frac{|\nabla U_0|_{\xi}^2}{2}-c_n\frac{U_0^{\crit}}{\crit}\right)\, d\xi=0
\end{array}\right\}\hbox{ when }n\geq 4.
\end{equation}
We prove the claim and assume that $n\geq 4$. Using the explicit expression of $U_0$, we get that
\begin{eqnarray*}
\int_{\partial\rnm}|x|^2\left(\frac{|\nabla U_0|_{\xi}^2}{2}-c_n\frac{U_0^{\crit}}{\crit}\right)\, d\xi &=& \frac{\omega_{n-2}(n-2)^2}{2}\int_0^\infty\frac{(r^2-1)r^n}{(1+r^2)^{n}}\, dr \\
&=& \frac{\omega_{n-2}(n-2)^2}{2} \Biggl[ \int_0^1 \frac{(r^2-1)r^n}{(1+r^2)^{n}}\, dr + \int_1^\infty\frac{(r^2-1)r^n}{(1+r^2)^{n}}\, dr \Biggr] \\
&=& \frac{\omega_{n-2}(n-2)^2}{2} \Biggl[ \int_0^1 \frac{(r^2-1)r^n}{(1+r^2)^{n}}\, dr + \int_0^1\frac{(1-r^2)r^{n-4}}{(1+r^2)^{n}}\, dr \Biggr] \\
&=& \frac{\omega_{n-2}(n-2)^2}{2}  \int_0^1 \frac{(1-r^2)(r^{n-4}-r^{n}) }{(1+r^2)^{n}}\, dr >0.
\end{eqnarray*}


Similarly, we prove that the second integral in \eqref{integ:positive} vanishes. The claim is proved. This ends Step \ref{th:asymp:bndy}.6.

\medskip\noindent{\bf Step \ref{th:asymp:bndy}.7:} Assume that $n\geq 4$. Plugging together \eqref{lim:bnd:Rc}, \eqref{lim:bnd:R} and \eqref{integ:positive}, we get that there exists $d_i>0$ such that
\begin{equation}\label{lim:H}
\lim_{\alpha\to +\infty}\mai^{-1}\int_{\varphi(U_\alpha\cap\partial\rnm)}(x-\xai,\nu)\left(\frac{|\nabla \ua|^2}{2}-c_n\frac{\ua^{\crit}}{\crit}+\frac{\ea\ua^2}{2}\right)\, d\sigma=d_iH(x_0)
\end{equation}
Plugging \eqref{poho:bord:1}, \eqref{poho:bord:2} and \eqref{lim:H} into \eqref{id:poho:bord}, we get that
$$\left(\frac{\mai}{\sai}\right)^{n-2}(c_i+o(1)) +(d_iH(x_0)+o(1))\mai=0$$
when $\alpha\to +\infty$, and then there exists $c_i'>0$ such that
\begin{equation}\label{lim:m:s:H}
\lim_{\alpha\to +\infty}\frac{\mai^{n-3}}{\sai^{n-2}}=-c_i'H(x_0).
\end{equation}
In particular, $H(x_0)\leq 0$. This ends the proof of Theorem \ref{th:asymp:bndy} when $n\geq 4$. We are now left with the case $n=3$.

\medskip\noindent{\bf Step \ref{th:asymp:bndy}.8:} We assume that $n=3$. We define $\tuai$ as above. We let $(z_\alpha)_{\alpha}\in \rn$ be such that
$$\lim_{\alpha\to +\infty}|z_\alpha|=+\infty.$$
Then, we have that
\begin{equation}\label{lim:3.1}
|z_\alpha|\leq\delta\frac{\sai}{\mai}\,\Rightarrow\, |z_\alpha|^{n-2}\tuai(z_\alpha)=O(1)\hbox{ when }\alpha\to +\infty
\end{equation}
and
\begin{equation}\label{lim:3.2}
|z_\alpha|=o\left(\frac{\sai}{\mai}\right)\,\Rightarrow\, \lim_{\alpha\to +\infty}|z_\alpha|^{n-2}\tuai(z_\alpha)=1\hbox{ when }\alpha\to +\infty.
\end{equation}
\medskip\noindent We prove the claim. As in Case \ref{th:cv:ra}.2 of the proof of Theorem \ref{th:cv:ra}, we have that
$$\tua(x)=(1+o(1))\left(\bua+\sum_{i=1}^NV_{j,\alpha}(x)\right)$$
for all $x\in B_{\delta_0}(x_0)$ and all $\alpha\in\nn$ for $\delta_0>0$ small enough. Therefore, we have that
\begin{eqnarray*}
|z_\alpha|^{n-2}\tuai(z_\alpha)&=&(1+o(1))\left(|z_\alpha|^{n-2}\mai^{\frac{n-2}{2}}\bua\right.\\
&&\left.+\sum_{i=1}^N|z_\alpha|^{n-2}\mai^{\frac{n-2}{2}}V_{j,\alpha}(\varphi(\varphi^{-1}(\xai)+\mai z_\alpha))\right)
\end{eqnarray*}
for all $\alpha\in\nn$. It follows from Theorem \ref{th:asymp:int} that there is no blowup point in the interior when $n=3$: therefore, \eqref{ppty:cv:bord} rewrites
\begin{equation}\label{mino:xi:xj}
|\xai-\xaj|\geq 2\delta\sai\hbox{ for all }j\neq i\hbox{ and all }\alpha\in\nn.
\end{equation}
We fix $j\neq i$. Similar to what was done in Step \ref{th:cv:ra}.1.2 of the proof of Theorem \ref{th:cv:ra}, we have that
$$\left||z_\alpha|^{n-2}\mai^{\frac{n-2}{2}}V_{j,\alpha}(\varphi(\varphi^{-1}(\xai)+\mai z_\alpha))\right|\leq C\left(\frac{\sai^2\maj}{\mai(\maj^2+|\xai-\xaj|^2)}\right)^{\frac{n-2}{2}}$$
for all $\alpha\in\nn$ (we have used that $|z_\alpha|\leq \delta\mai^{-1}\sai$). Therefore, if $\mai=O(\maj)$ when $\alpha\to +\infty$, it follows from the definition of $\sai$ that the right-hand-side is bounded. If $\maj=O(\mai)$, using \eqref{mino:xi:xj}, we get that the right-hand-side is also bounded.

\smallskip\noindent In the case $j=i$, it follows from Case (i) of Proposition \ref{lemmaprojection} that
\begin{eqnarray*}
|z_\alpha|^{n-2}\mai^{\frac{n-2}{2}}V_{j,\alpha}(\varphi(\varphi^{-1}(\xai)+\mai z_\alpha))&=&(1+o(1))\left(\frac{|z_\alpha|^2\mai}{(\mai^2+|\mai z_\alpha|^2)}\right)^{\frac{n-2}{2}}\\
&=&1+o(1)
\end{eqnarray*}
when $\alpha\to +\infty$ since $\lim_{\alpha\to +\infty}|z_\alpha|=+\infty$.

\smallskip\noindent Finally, noting in addition that $|z_\alpha|^{n-2}\mai^{\frac{n-2}{2}}\bua=O(\sai^{n-2}\mai^{-\frac{n-2}{2}}\bua)=O(1)$ by definition of $\sai$, we get that \eqref{lim:3.1} holds. With a little more careful analysis, we get \eqref{lim:3.2}. This ends Step \ref{th:asymp:bndy}.8.

\medskip\noindent{\bf Step \ref{th:asymp:bndy}.9:} We still assume that $n=3$ and we let $(z_\alpha)_{\alpha}\in \rn$ be such that $\lim_{\alpha\to +\infty}|z_\alpha|=+\infty$. Then, we claim that
\begin{equation}\label{lim:3.3}
|z_\alpha|=o\left(\frac{\sai}{\mai}\right)\,\Rightarrow\, \lim_{\alpha\to +\infty}|z_\alpha|^{n-1}|\nabla\tuai(z_\alpha)|_{g_\alpha}=n-2,
\end{equation}
where $g_\alpha(x):=(\varphi^\star \tg)(\varphi(\varphi^{-1}(\xai)+\mai x))$.

\smallskip\noindent We prove the claim by contradiction and assume that there exists $(z_\alpha)_\alpha$ as above and $\epsilon_0>0$ such that
\begin{equation}\label{lim:abs:3}
||z_\alpha|^{n-1}|\nabla\tuai(z_\alpha)|_{g_\alpha}-(n-2)|\geq \epsilon_0
\end{equation}
for all $\alpha\in\nn$. We define $r_\alpha:=|z_\alpha|$ and $w_\alpha(x):=\ra^{n-2}\tuai(\ra x)$ for $x\neq 0$: this is well defined and it follows from Step \ref{th:asymp:bndy}.8 that $\lim_{\alpha\to +\infty}w_\alpha(x)=|x|^{2-n}$ in $C^0_{loc}(\rn\setminus\{0\})$. Moreover, $\Delta_{g_\alpha'}w_\alpha+\ea (\mai|z_\alpha|)^2 w_\alpha=c_n|z_\alpha|^{-2}w_\alpha^{\crit-1}$ where $g_\alpha'(x)=\ra^2g_\alpha(\ra x)$, and therefore, it follows from standard elliptic theory that $w_\alpha$ converges in $C^1_{loc}(\rn\setminus\{0\})$. Computing $\nabla w_\alpha(r_\alpha^{-1}z_\alpha)$ and passing to the limit when $\alpha\to +\infty$ contradicts \eqref{lim:abs:3}. This ends Step \ref{th:asymp:bndy}.9.

\medskip\noindent The rough estimate \eqref{est:pointwise} and computations similar to the case $n\geq 4$ yield
\begin{equation}\nonumber
\int_{\varphi(B_{\delta\sai}(\varphi^{-1}(\xai))\cap\partial\Omega}\left|(x-\xai,\nu)\left(c_n\frac{\ua^{\crit}}{\crit}-\frac{\ea\ua^2}{2}\right)\right|\, dx=O(\mai)
\end{equation}
when $\alpha\to +\infty$. Similarly, we have that
\begin{equation}\nonumber
\int_{\varphi(B_{\delta\sai}(\varphi^{-1}(\xai))\cap\partial\Omega}|x-\xai|^2|\nabla\ua|^2\, dx=O\left(\mai\ln\frac{\sai}{\mai}\right)
\end{equation}
when $\alpha\to +\infty$. Therefore, we have that
\begin{eqnarray}\label{eq:F:3}
&&\int_{\varphi(B_{\delta\sai}(\varphi^{-1}(\xai))\cap\partial\Omega}(x-\xai,\nu)F_\alpha(x)\, dx\\
&&=-\frac{(\partial_{kl}\varphi(0),\nu(x_0))}{4}\int_{\varphi(B_{\delta\sai}(\varphi^{-1}(\xai))\cap\partial\Omega}x^kx^l|\nabla\ua|^2\, dx+o\left(\mai\ln\frac{\sai}{\mai}\right)\nonumber
\end{eqnarray}
when $\alpha\to +\infty$ and $n=3$.

\medskip\noindent{\bf Step \ref{th:asymp:bndy}.10:} Assume that $n=3$. We claim that
\begin{equation}\label{lim:3.5}
\lim_{\alpha\to +\infty}\frac{\int_{\varphi(B_{\delta\sai}(\varphi^{-1}(\xai))\cap\partial\Omega}x^kx^l|\nabla\ua|^2\, dx}{\mai\ln\frac{\sai}{\mai}}=\frac{2\pi}{3}\delta^{kl}
\end{equation}
We prove the claim. We let $(\delta_\alpha)_\alpha\in (0,+\infty)$ be such that $\lim_{\alpha\to +\infty}\delta_\alpha=0$ and $\lim_{\alpha\to +\infty}\mai^{-1}\delta_\alpha\sai=+\infty$. The sequence $(\delta_\alpha)$ will be chosen later. With a change of variables and noting that the element of volume satisfies $dv_{g_\alpha}=(1+o(1))d\sigma$, we have that
\begin{eqnarray*}
&&\int_{\varphi(B_{\delta_\alpha\sai}(\varphi^{-1}(\xai))\cap\partial\Omega}x^kx^l|\nabla\ua|^2\, dx=\mai\int_{B_{\frac{\delta_\alpha\sai}{\mai}}(0)\cap\partial\rnm}x^kx^l|\nabla\tuai|_{g_\alpha}^2dv_{g_\alpha}\\
&&=\mai\int_{B_{\frac{\delta_\alpha\sai}{\mai}}(0)\cap\partial\rnm\setminus B_R(0)}(1+o(1))x^kx^l|\nabla\tuai|_{g_\alpha}^2d\sigma+O(\mai)\\
&&=\mai\int_{B_{\frac{\delta_\alpha\sai}{\mai}}(0)\cap\partial\rnm\setminus B_R(0)}(1+o(1))\frac{x^kx^l}{|x|^{2(n-1)}}(|x|^{n-1}|\nabla\tuai|_{g_\alpha})^2dv_{g_\alpha}+O(\mai)
\end{eqnarray*}
when $\alpha\to +\infty$ for $R>0$ arbitrary large. With \eqref{lim:3.3}, we then get that there exists $\epsilon_R$ such that $\lim_{R\to +\infty}\epsilon_R=0$ and
\begin{eqnarray}
&&\int_{\varphi(B_{\delta_\alpha\sai}(\varphi^{-1}(\xai))\cap\partial\Omega}x^kx^l|\nabla\ua|^2\, dx\nonumber\\
&&=\mai(n-2)^2\int_{B_{\frac{\delta_\alpha\sai}{\mai}}(0)\cap\partial\rnm\setminus B_R(0)}\frac{x^kx^l}{|x|^{2(n-1)}}dv_\sigma+(o(1)+\epsilon_R)\left(\mai\ln\frac{\delta_\alpha\sai}{\mai}\right)\nonumber\\
&&=\frac{(n-2)^2\omega_{n-2}\delta^{kl}}{n}\mai\ln\frac{\delta_\alpha\sai}{\mai}+o\left(\mai\ln\frac{\delta_\alpha\sai}{\mai}\right)\label{eq:bnv:1}
\end{eqnarray}
We now estimate the complementing term. It follows from \eqref{lim:3.3} and the local convergence of $\tuai$ that there exists $C>0$ such that $|\nabla\tuai|_{g_\alpha}(x)\leq C|x|^{1-n}$ for all $x\in B_{\mai^{-1}\delta\sai}(0)$. Therefore, we have that
\begin{eqnarray}
&&\int_{\varphi(B_{\delta\sai}(\varphi^{-1}(\xai))\cap\partial\Omega\setminus B_{\delta_\alpha\sai}(\varphi^{-1}(\xai))}x^kx^l|\nabla\ua|^2\, dx\nonumber\\
&\leq & C\mai\int_{B_{\mai^{-1}\delta\sai}(0)\setminus B_{\mai^{-1}\delta_\alpha\sai}(0)}|z|^2(1+|z|^2)^{1-n}\, dz\nonumber\\
&\leq &C\mai\ln \frac{\delta}{\delta_\alpha}\label{eq:bnv:2}
\end{eqnarray}
We now choose $(\delta_\alpha)_\alpha$ such that
$$\lim_{\alpha\to +\infty}\delta_\alpha=0\; ,\; \lim_{\alpha\to +\infty}\frac{\delta_\alpha\sai}{\mai}=+\infty\hbox{ and }\ln\frac{1}{\delta_\alpha}=o\left(\ln\frac{\sai}{\mai}\right)$$
when $\alpha\to +\infty$. Clearly, this choice is possible: combining \eqref{eq:bnv:1} and \eqref{eq:bnv:2} yields \eqref{lim:3.5}. This ends Step \ref{th:asymp:bndy}.10.

\medskip\noindent{\bf Step \ref{th:asymp:bndy}.11:} Assume that $n=3$. We claim that there exists $c_i>0$ such that
$$\lim_{\alpha\to +\infty}\frac{1}{\sai\ln\frac{1}{\mai}}=-c_i H(x_0).$$
\smallskip\noindent We prove the claim. Putting \eqref{lim:3.5} into \eqref{eq:F:3} and arguing as in Step \ref{th:asymp:bndy}.5 yields
\begin{eqnarray*}
&&\int_{\varphi(B_{\delta\sai}(\varphi^{-1}(\xai))\cap\partial\Omega}(x-\xai,\nu)F_\alpha(x)\, dx\\
&&=-\frac{\pi\sum_k(\partial_{kk}\varphi(0), \nu(x_0))}{6}\mai\ln\frac{\sai}{\mai}+o\left(\mai\ln\frac{\sai}{\mai}\right)\\
&&=\frac{\pi H(x_0)}{6}\mai\ln\frac{\sai}{\mai}+o\left(\mai\ln\frac{\sai}{\mai}\right)
\end{eqnarray*}
When $\alpha\to +\infty$. Plugging this asymptotic behavior into \eqref{id:poho:bord} and using \eqref{poho:bord:1} and \eqref{poho:bord:2} yields the existence of $c_i>0$ such that
$$c_i(H(x_0)+o(1))\sai\ln\frac{\sai}{\mai}+1=0$$
when $\alpha\to +\infty$. This yields the desired result and this ends Step \ref{th:asymp:bndy}.11. Theorem \ref{th:asymp:bndy} is proved for $n=3$.

\section{Proof of Theorems \ref{th:sing:bndy} and \ref{th:conj}}
\label{sec:proof:th}
Let $(\ua)_{\alpha\in\nn}\in C^2(\omegabar)$ be as in the statement of Theorem \ref{th:sing:bndy}. We let
$$\hat{\mathcal{S}}:=\left\{\lim_{\alpha\to +\infty}\xai/\, i\in\{1,...,N\}\right\}.$$

\medskip\noindent{\bf Step 1:} We claim that
$$\hat{\mathcal{S}}=\mathcal{S}$$
where $\mathcal{S}$ is as in Definition \ref{def:sing}. We prove the claim. Let $x_0\in \hat{\mathcal{S}}$ and let $i\in\{1,..., N\}$ such that $\lim_{\alpha\to +\infty}\xai=x_0$. In particular, we have that $\lim_{\alpha\to +\infty}\ua(\xai)=+\infty$: then $x_0$ is a singular point, and then $x_0\in \mathcal{S}$. This proves that $\hat{\mathcal{S}}\subset \mathcal{S}$.\par

\medskip\noindent Let $x_0\in \hat{\mathcal{S}}^c$: then there exists $\delta>0$ such that $|x_0-\xai|\geq 2\delta$ for all $i\in\{1,...,N\}$. In particular, it follows from \eqref{est:pointwise} that there exists $C>0$ such that $\ua(x)\leq C$ for all $x\in B_\delta(x_0)\cap\Omega$, and then $x_0$ is not a singular point, that is $x\not\in \mathcal{S}$. This proves that $\hat{\mathcal{S}}^c\subset \mathcal{S}^c$.\par
\medskip\noindent These two assertions prove that $\hat{\mathcal{S}}=\mathcal{S}$, and the claim is proved. This ends Step 1.\hfill$\Box$

\medskip\noindent{\bf Step 2:} Let $x_0\in \mathcal{S}$. Assume that $n\geq 7$. We claim that
$$\hbox{there exists }(\xai)_{\alpha\in\nn}\in\partial\Omega\hbox{ such that }\lim_{\alpha\to +\infty}\xai=x_0.$$

\medskip\noindent We prove the claim by contradiction and assume that for all $i\in\{1,...,N\}$ such that $\lim_{\alpha\to +\infty}\xai=x_0$, then $\xai\in\Omega$. We let $i\in\{1,...,N\}$ such that
$$\mu_{i, \alpha}:=\max\{\maj/\, j\in\{1,...,N\} \hbox{ such that }\lim_{\alpha\to +\infty}\xaj=x_0\}.$$
It then follows from Theorem \ref{th:asymp:int} that
\begin{equation}\label{as:sai:ea}
\ea\sai^{n-2}\asymp\mai^{n-4}
\end{equation}
when $\alpha\to +\infty$ and there exists $j\in\{1,...,N\}$ such that $\mai=o(\maj)$ when $\alpha\to +\infty$ and
$$\sai^2=\frac{\mai}{\maj}(\maj^2+|\xai-\xaj|^2)$$
for all $\alpha\in\nn$.\par
\smallskip\noindent Assume that $\lim_{\alpha\to +\infty}\xaj=x_0$. Then it follows from the definition of $\mai$ that $\mai\geq \maj$: a contradiction with $\mai=o(\maj)$ when $\alpha\to +\infty$.\par
\smallskip\noindent Assume that  $\lim_{\alpha\to +\infty}\xaj\neq x_0$. Then $\xai-\xaj\not\to 0$ and we have that
$$\sai^2\asymp\frac{\mai}{\maj}$$
when $\alpha\to +\infty$. Plugging this estimate in \eqref{as:sai:ea}, we get with \eqref{comp:ea:ma} that
$$\ea\asymp \mai^{\frac{n-6}{2}}\maj^{\frac{n-2}{2}}=o(\maj^{\frac{n-2}{2}})=o(\ea)$$
when $\alpha\to +\infty$. A contradiction since $n\geq 7$.\par
\smallskip\noindent This proves the claim, and this ends Step 2.\hfill$\Box$

\medskip\noindent{\bf Step 3:} Let $x_0\in \mathcal{S}$. Assume that $n=3$ or $n\geq 7$. We claim that
$$x_0\in \partial\Omega\hbox{ and that }H(x_0)\leq 0.$$

\smallskip\noindent We prove the claim. We let $i\in\{1,...,N\}$ be such that
$$\sai=\min\{\saj/\, \xaj\in\partial\Omega\hbox{ and }\lim_{\alpha\to +\infty}\xaj=x_0\}.$$
This minimum is well-defined: this follows from Theorem \ref{th:asymp:int} for $n=3$ and from Step 2 when $n\geq 7$. In particular, $\xai\in\partial\Omega$ and $x_0\in\partial\Omega$. We claim that for all $j\in\{1,...,N\}\setminus\{i\}$
\begin{equation}\label{hyp:concl:th}
\xaj\in\partial\Omega\;\Rightarrow\;\xaj-\xai\neq o(\sai)\hbox{ when }\alpha\to +\infty
\end{equation}
We prove the claim by contradiction and we assume that there exists $j\in\{1,...,N\}\setminus\{i\}$ such that $\lim_{\alpha\to +\infty}\xaj=x_0$, $\xai-\xaj=o(\sai)$ and $\xaj\in\partial\Omega$ for all $\alpha\in\nn$.\par
\smallskip\noindent We claim that $\mai=o(\maj)$ when $\alpha\to +\infty$. We argue by contradiction and assume that $\maj=O(\mai)$ when $\alpha\to +\infty$: then $i\in J_j$ and it follows from the definition \eqref{def:sai} of $\saj$ that
$$\saj^2\leq\frac{\maj}{\mai}\left(\mai^2+|\xai-\xaj|^2\right)$$
for all $\alpha\in\nn$. Since $|\xai-\xaj|=o(\sai)$ and $\mai=o(\sai)$ when $\alpha\to +\infty$, we get that $\saj=o(\sai)$ when $\alpha\to +\infty$: a contradiction since $\sai\leq \saj$ for all $\alpha\in\nn$. This proves that $\mai=o(\maj)$ when $\alpha\to +\infty$.\par

\smallskip\noindent In particular, we have that $j\in J_i$, and then
$$\sai^2\leq\frac{\mai}{\maj}(\maj^2+|\xai-\xaj|^2)$$
for all $\alpha\in\nn$. Since $\xai-\xaj=o(\sai)$ and $\mai=o(\maj)$ when $\alpha\to +\infty$, we then get that $\sai=o(\maj)$ and then $\xai-\xaj=o(\maj)$ when $\alpha\to +\infty$. A contradiction with \eqref{zero:S}. This proves that \eqref{hyp:concl:th} holds.

\medskip\noindent Therefore, we can apply Theorem \ref{th:asymp:bndy} to $i$, and we get that $H(x_0)\leq 0$ when $n=3$ or $n\geq 7$. This proves the claim, and therefore this ends Step 3.

\medskip\noindent Theorem \ref{th:sing:bndy} is a consequence of Step 3.

\medskip\noindent Theorem \ref{th:conj} is a consequence of Theorem \ref{th:sing:bndy} and Proposition \ref{prop:claimA}.

\section*{Appendix A: Construction and estimates on the Green's function}\label{sec:Green}
This appendix is devoted to a construction and to pointwise properties of the Green's functions of the Laplacian with Neumann boundary condition on a smooth bounded domain of $\rn$. These proof are essentially self-contained and require only standard elliptic theory.

\medskip\noindent  Let $\Omega$ be a smooth bounded domain of $\rn$ (see Definition \ref{def:smoothdomain} in Section \ref{sec:reflection}). We consider the following problem:
\begin{equation}\label{app:NB}
\left\{\begin{array}{ll}
\Delta u=f &\hbox{ in }\Omega\\
\partial_\nu u=0&\hbox{ in }\partial\Omega
\end{array}\right.
\end{equation}
where $u\in C^2(\omegabar)$ and $f\in C^0(\omegabar)$. Note that the solution $u$ is defined up to the addition of a constant and that it is necessary that $\int_\Omega f\, dx=0$ (this is a simple integration by parts). Our objective here is to study the existence and the properties of the Green kernel associated to \eqref{app:NB}.

\begin{defi} We say that a function $G: \Omega\times \Omega\setminus\{(x,x)/\, x\in\Omega\}\to\rr$ is a Green's function for \eqref{app:NB} if for any $x\in\Omega$, noting $G_x:=G(x,\cdot)$, we have that

\smallskip(i) $G_x\in L^1(\Omega)$,

\smallskip(ii) $\int_\Omega G_x\, dy=0$,

\smallskip(iii) for all $\varphi\in C^2(\omegabar)$ such that $\partial_\nu\varphi=0$ on $\partial\Omega$, we have that
$$\varphi(x)-\bar{\varphi}=\int_{\Omega}G_x\Delta\varphi\, dy.$$
\end{defi}

\smallskip\noindent Condition (ii) here is required for convenience in order to get uniqueness, symmetry and regularity for the Green's function. Note that if $G$ is a Green's function and if $c:\Omega\to\rr$ is any function, the function $(x,y)\mapsto G(x,y)+c(x)$ satisfies (i) and (iii). The first result concerns the existence of the Green's function:

\begin{thm}\label{app:th:green:exist} Let $\Omega$ be a smooth bounded domain of $\rn$. Then there exists a unique Green's function $G$ for \eqref{app:NB}. Moreover, $G$ is symmetric and extends continuously to $\omegabar\times\omegabar\setminus\{(x,x)/\, x\in\omegabar\}$ and for any $x\in\omegabar$, we have that $G_x\in C^{2,\alpha}(\omegabar\setminus\{x\})$ and satisfies
\begin{equation*}
\left\{\begin{array}{ll}
\Delta G_x=-\frac{1}{|\Omega|} &\hbox{ in }\Omega\setminus\{x\}\\
\partial_\nu G_x=0&\hbox{ in }\partial\Omega.
\end{array}\right.
\end{equation*}
In addition, for all $x\in\omegabar$ and for all $\varphi\in C^2(\omegabar)$  we have that
$$\varphi(x)-\bar{\varphi}=\int_{\Omega}G_x\Delta\varphi\, dy+\int_{\partial\Omega}G_x\partial_\nu\varphi\, dy.$$
\end{thm}

\medskip\noindent A standard and useful estimate for Green's function is the following uniform pointwise upper bound:
\begin{prop}\label{app:prop:upp:bnd} Let $G$ be the Green's function for \eqref{app:NB}. Then there exist $C(\Omega)>0$ and $m(\Omega)>0$ depending only on $\Omega$ such that
\begin{equation}\label{app:est:green}
\frac{1}{C(\Omega)}|x-y|^{2-n}-m(\Omega)\leq G(x,y)\leq C(\Omega)|x-y|^{2-n}\hbox{ for all }x,y\in\Omega, \, x\neq y.
\end{equation}
Concerning the derivatives, we get that
\begin{equation}\label{app:est:der:green}
|\nabla_y G_x(y)|\leq C|x-y|^{1-n}\hbox{ for all }x,y\in\Omega, \, x\neq y.
\end{equation}
\end{prop}
Estimate \eqref{app:est:green} was proved by Rey-Wei \cite{ReyWei} with a different method. We also refer to Faddeev \cite{f} for very nice estimates in the two-dimensional case.

\medskip\noindent{\bf Notations:} in the sequel, $C(a,b,...)$ denotes a constant that depends only on $\Omega$, $a$, $b$... We will often keep the same notation for different constants in a formula, and even in the same line.

\medskip\noindent We will intensively use the following existence and regularity for solutions to the Neumann problem (this is in Agmon-Douglis-Nirenberg \cite{adn}):

\begin{thm}\label{app:th:existence} Let $\Omega$ be a smooth bounded domain of $\rn$ and let $f\in L^p(\Omega)$, $p>1$ be such that $\int_\Omega f\, dx=0$. Then there exists $u\in H_2^p(\Omega)$ which is a weak solution to
$$\left\{\begin{array}{ll}
\Delta u=f &\hbox{ in }\Omega\\
\partial_\nu u=0&\hbox{ in }\partial\Omega
\end{array}\right.$$
The function $u$ is unique up to the addition of a constant. Moreover, there exists $C(p)>0$ such that
$$\Vert u-\bar{u}\Vert_{H_2^p(\Omega)}\leq C(p)\Vert f\Vert_p.$$
If $f\in C^{0,\alpha}(\omegabar)$, $\alpha\in (0,1)$, then $u\in C^{2,\alpha}(\omegabar)$ is a strong solution and there exists $C(\alpha)>0$ such that
$$\Vert u-\bar{u}\Vert_{C^{2,\alpha}(\Omega)}\leq C(\alpha)\Vert f\Vert_{C^{0,\alpha}(\omegabar)}.$$
\end{thm}

\medskip\noindent{\bf A.1. Construction of the Green's function and proof of the upper bound.}\par
\noindent This section is devoted to the proof of Theorem \ref{app:th:green:exist}.

\medskip\noindent{\it A.1.1. Construction of $G_x$.}\par
\noindent We define $k_n:=\frac{1}{(n-2)\omega_{n-1}}$. We fix $x\in\Omega$ and we take $u_x\in C^2(\omegabar)$ that will be chosen later, and we define
$$H_x:=k_n |\cdot-x|^{2-n}+u_x.$$
In particular, $H_x\in L^p(\Omega)$ for all $p\in (1,\frac{n}{n-2})$. We let $u\in C^2(\omegabar)$ be a function. Standard computations (see \cite{gt} or \cite{robert:green}) yield
\begin{equation}\label{app:eq:H:1}
\int_\Omega H_x\Delta u\, dy=u(x)+\int_\Omega u\Delta u_x\, dy+\int_{\partial\Omega}(-\partial_\nu uH_x+u\partial_\nu H_x)\, d\sigma.
\end{equation}
We let $\eta\in C^\infty(\rr)$ be such that $\eta(t)=0$ if $t\leq 1/3$ and $\eta(t)=1$ if $t\geq 2/3$. We define
$$v_x(y):=\eta\left(\frac{|x-y|}{d(x,\partial\Omega)}\right)k_n|x-y|^{2-n}$$
for all $y\in \omegabar$. Clearly, $v_x\in C^\infty(\omegabar)$ and $v_x(y)=k_n|x-y|^{2-n}$ for all $y\in\Omega$ close to $\partial\Omega$. It follows from Theorem \ref{app:th:existence} that there exists $u'_x\in C^{2,\alpha}(\omegabar)$ for all $\alpha\in (0,1)$ unique such that
$$\left\{\begin{array}{ll}
\Delta u'_x=\Delta v_x -\overline{\Delta v_x}&\hbox{ in }\Omega\\
\partial_\nu u_x'=0&\hbox{ in }\partial\Omega\\
\overline{u'_x}=0&
\end{array}\right.$$
We now define $u_x:=u'_x-v_x\in C^{2,\alpha}(\omegabar)$ and $c_x:=\overline{\Delta v_x}\in\rr$ so that
$$\left\{\begin{array}{ll}
\Delta u_x= -c_x&\hbox{ in }\Omega\\
\partial_\nu u_x=-\partial_\nu(k_n|\cdot-x|^{2-n})&\hbox{ in }\partial\Omega
\end{array}\right.$$
Therefore, $\partial_\nu H_x=0$ on $\partial\Omega$ and \eqref{app:eq:H:1} rewrites
$$\int_\Omega H_x\Delta u\, dy=u(x)-c_x\int_\Omega u\, dy-\int_{\partial\Omega}\partial_\nu uH_x\, d\sigma$$
for all $u\in C^2(\omegabar)$. Taking $u\equiv 1$ yields $c_x=\frac{1}{|\Omega|}$, and then, we have that
$$\int_\Omega H_x\Delta u\, dy=u(x)-\bar{u}-\int_{\partial\Omega}\partial_\nu uH_x\, d\sigma$$
for all $u\in C^2(\omegabar)$. Finally, we define $G_x:=H_x-\overline{H_x}$ and we have that:
$$\int_\Omega G_x\Delta u\, dy=u(x)-\bar{u}-\int_{\partial\Omega}\partial_\nu uG_x\, d\sigma$$
for all $u\in C^2(\omegabar)$. Therefore $G$ is a Green's function for \eqref{app:NB}. In addition,
$$G_x\in C^{2,\alpha}(\omegabar\setminus\{x\})\cap L^p(\Omega)\hbox{ for all }\alpha\in (0,1)\hbox{ and } p\in \left(1,\frac{n}{n-2}\right).$$
Taking $u\in C^\infty_c(\omegabar\setminus\{x\})$ above, and the definition of $G_x$, we get that
\begin{equation}\label{app:eq:Gx}
\left\{\begin{array}{ll}
\Delta G_x=-\frac{1}{|\Omega|} &\hbox{ in }\Omega\setminus\{x\}\\
\partial_\nu G_x=0&\hbox{ in }\partial\Omega.
\end{array}\right.
\end{equation}
\noindent{\it A.1.2. Uniform $L^p-$bound.}\par
\begin{lemma}\label{app:lem:lp} Fix $x\in\omegabar$ and assume that there exist $H\in L^1(\Omega)$ such that
$$\int_\Omega H\Delta u\, dy=u(x)-\bar{u}$$
for all $u\in C^2(\omegabar)$ such that $\partial_\nu u=0$ on $\partial\Omega$. Then $H\in L^p(\Omega)$ for all $p\in \left(1,\frac{n}{n-2}\right)$ and there exists $C(p)>0$ independent of $x$ such that
\begin{equation}\label{app:lp:bnd}
\Vert H-\bar{H}\Vert_p\leq C(\Omega,p)
\end{equation}
for all $x\in \Omega$.
\end{lemma}
\begin{proof} For $p$ as above, we define $q:=\frac{p}{p-1}>\frac{n}{2}$. We fix $\psi\in C^\infty(\omegabar)$. It follows from Theorem \ref{app:th:existence} that there exists $u\in C^2(\omegabar)$ such that
$$\left\{\begin{array}{ll}
\Delta u=\psi-\bar{\psi}&\hbox{ in }\Omega\\
\partial_\nu u=0&\hbox{ in }\partial\Omega\\
\bar{u}=0&
\end{array}\right.$$
It follows from the properties of $H$ that
$$\int_\Omega (H-\bar{H})\psi\, dy=\int_\Omega H(\psi-\bar{\psi})\, dy=u(x).$$
It follows from Sobolev's embedding that $H_2^q(\Omega)$ is continuously embedded in $L^\infty(\Omega)$: therefore, using the control of the $H_2^q-$norm of Theorem \ref{app:th:existence} yields
$$\left|\int_\Omega (H-\bar{H})\psi\, dy\right|\leq \Vert u\Vert_\infty\leq C(q)\Vert u\Vert_{H_2^q}\leq C'(q)\Vert \psi-\bar{\psi}\Vert_q\leq C''(q)\Vert \psi\Vert_q$$
for all $\psi\in C^\infty_c(\omegabar)$. It then follows from duality that $H-\bar{H}\in L^p(\Omega)$ and that \eqref{app:lp:bnd} holds.\end{proof}

\medskip\noindent{\it A.1.3. Uniqueness.}\par
\noindent We prove the following uniqueness result:
\begin{lemma}\label{app:lem:unic} Fix $x\in\omegabar$ and assume that there exist $G_1,G_2\in L^1(\Omega)$ such that
$$\int_\Omega G_i\Delta u\, dy=u(x)-\bar{u}$$
for all $i\in\{1,2\}$ and for all $u\in C^2(\omegabar)$ such that $\partial_\nu u=0$ on $\partial\Omega$. Then there exists $c\in\rr$ such that $G_1-G_2=c$ a.e on $\Omega$.
\end{lemma}

\begin{proof} We define $g:=G_1-G_2$. We have that
$$\int_\Omega g\Delta u\, dy=0$$
for all $u\in C^2(\omegabar)$ such that $\partial_\nu u=0$ on $\partial\Omega$. We fix $\psi\in C^\infty_c(\Omega)$. It follows from  Theorem \ref{app:th:existence} that there exists $u\in C^2(\omegabar)$ such that $\Delta u=\psi-\bar{\psi}$ in $\Omega$, $\partial_\nu u=0$ on $\partial\Omega$ and $\bar{u}=0$ . Therefore, we get that
$$\int_\Omega (g-\bar{g})\psi\, dy=\int_\Omega g(\psi-\bar{\psi})\, dy=\int_\Omega g\Delta u\, dy=0.$$
for all $\psi\in C^\infty_c(\Omega)$. Moreover, it follows from Lemma \ref{app:lem:lp} that $g\in L^p(\Omega)$ for some $p>1$,  and then we get that $g-\bar{g}=0$ a.e, and then $G_1=G_2+\bar{g}$.\end{proof}

\medskip\noindent As an immediate corollary, we get that the function $G$ constructed above is the unique Green's function for \eqref{app:NB}.

\medskip\noindent{\it A.1.4. Pointwise control.}\par
\noindent We let $G$ be the Green's function for \eqref{app:NB}. The objective here is to prove that there exists $C(\Omega)>0$ such that
\begin{equation}\label{app:est:green:pointwise}
|G_x(y)|\leq C(\Omega)|x-y|^{2-n}
\end{equation}
for all $x,y\in\Omega$, $x\neq y$.
\begin{proof} The proof of \eqref{app:est:green:pointwise} goes through six steps.

\medskip\noindent {\bf Step 1:} We fix $K\subset \Omega$ a compact set. We claim that there exists $C(K)>0$ such that
$$|G_x(y)|\leq C(K)|x-y|^{2-n}$$
for all $x\in K$ and all $y\in\Omega$, $y\neq x$.

\smallskip\noindent We prove the claim. We use the notations $u_x,u'_x, v_x$ above. As easily checked, $v_x\in C^{2}(\omegabar)$ and $\Vert v_x\Vert_{C^2}\leq Cd(x,\partial\Omega)^{-n}\leq C d(K,\partial\Omega)^{-n}\leq C(K)$. Therefore, it follows from Theorem \ref{app:th:existence} that $\Vert u_x'\Vert_{\infty}\leq C(K)$, and then $|H_x(y)|\leq C(K)|x-y|^{2-n}$ for all $y\in\Omega$, $y\neq x$. Since $G_x=H_x-\overline{H_x}$ and \eqref{app:lp:bnd} holds, the claim follows.

\medskip\noindent {\bf Step 2:} We fix $\delta>0$. We claim that there exists $C(\delta)>0$ such that
\begin{equation}\label{app:bnd:c2}
\Vert G_x\Vert_{C^2(\Omega\setminus\bar{B}_x(\delta))}\leq C(\delta)
\end{equation}
for all $x,y\in \Omega$ such that $|x-y|\geq \delta$.

\smallskip\noindent We prove the claim. It follows from \eqref{app:eq:Gx} and standard elliptic theory (see for instance \cite{adn}) that for any $p>1$, there exists $C(\delta,p)>0$ such that $\Vert G_x\Vert_{C^{2}(\Omega\setminus\bar{B}_x(\delta))}\leq C(\delta)+C(\delta)\Vert G_x\Vert_{L^p(\Omega)}$. Step 2 is then a consequence of \eqref{app:lp:bnd}.

\medskip\noindent We are now interested in the neighborhood of $\partial\Omega$. We fix $x_0\in\partial\Omega$ and we choose a chart $\varphi$ as in Lemma \ref{app:lem:ext}. For simplicity, we assume that $\varphi: B_\delta(0)\to\rn$ and that $\varphi(0)=x_0$ and we define $V:=\varphi(B_\delta(0))$. We fix $x\in V\cap\Omega$ and we let $\tilde{G}_x$ be the extension $\tilde{G}_x:=G_x\circ\tpip=G_x\circ \varphi\circ \tilde{\pi}\circ\varphi^{-1}$: we have that
$$\tilde{G}_x: V\setminus\{x,x^\star\}\to\rr\hbox{ with }x^\star:=\pip^{-1}(x)=\varphi\circ\pi^{-1}\circ\varphi^{-1}(x)\in\overline{\Omega}^c.$$
Moreover, since $G_x$ is $C^{2,\alpha}$ outside $x$ and $\tpi$ is Lipschitz continuous, we have that $\tilde{G}_x\in H_{1,loc}^q(V\setminus\{x,x^\star\})$ for all $q>1$; in addition, it follows from \eqref{app:lp:bnd} that $\tilde{G}_x\in L^p(V)$ for all $p\in \left(1,\frac{n}{n-2}\right)$ and that there exists $C(p)>0$ independent of $x$ such that
\begin{equation*}
\Vert \tilde{G}_x\Vert_p\leq C(p).
\end{equation*}

\medskip\noindent{\bf Step 3:} We claim that
\begin{equation}\label{app:eq:tG:dist}
\Delta_{\tg}\tilde{G}_x=\delta_x+\delta_{x^\star}-\frac{1}{|\Omega|}\hbox{ in }{\mathcal D}'(V).
\end{equation}
We prove the claim. We let $\psi\in C^\infty_c(V)$ be a smooth function. Separating $V\cap\Omega$ and $V\cap\Omega^c$ and using a change of variable, we get that
$$\int_V\tilde{G}_x\Delta_{\tg}\psi\, dv_{\tg}=\int_{V\cap \Omega}G_x\Delta\left(\psi+\psi\circ\pip^{-1}\right)\, dy.$$
Noting that $\partial_\nu \left(\psi+\psi\circ\pip^{-1}\right)=0$ on $\partial\Omega$ (we have used that $\nu(\varphi(0,x'))=d\varphi_{(0,x')}(\vec{e}_1)$) and using the definition of the Green's function $G_x$, we get that
\begin{eqnarray*}
\int_V\tilde{G}_x\Delta_{\tg}\psi\, dv_{\tg}&=&\psi(x)+\psi\circ\pip^{-1}(x)-\frac{1}{|\Omega|}\int_{V\cap \Omega}\left(\psi+\psi\circ\pip^{-1}\right)\, dy\\
&=&\psi(x)+\psi(x^\star)-\frac{1}{|\Omega|}\int_{V}\psi\, dv_{\tg}.
\end{eqnarray*}
This proves \eqref{app:eq:tG:dist} and ends the claim.

\medskip\noindent{\bf Step 4:} We fix $z\in V$. We claim that there exists $\Gamma_z: V\setminus\{z\}\to \rr$ such that the following properties hold:
\begin{equation}\label{app:ppty:gamma}
\left\{\begin{array}{ll}
\Delta_{\tg}\Gamma_z=\delta_z &\hbox{ in }{\mathcal D}'(V),\\
&\\
|\Gamma_z(y)|\leq C|z-y|^{2-n} & \hbox{ for all }y\in V\setminus\{z\},\\
&\\
\Gamma_z\in C^{1}(V\setminus\{z\})&
\end{array}\right\}\end{equation}

\medskip\noindent We prove the claim. We define $r(y):=\sqrt{\tg_{ij}(z)(y-z)^i(y-z)^j}$ for all $y\in V$. As easily checked, $r^{2-n}\in C^\infty(V\setminus\{z\})$: we define $f:=\Delta_{\tg}r^{2-n}$ on $V\setminus\{z\}$. It follows from the properties of $\tg$ that $f\in L^\infty_{loc}(V\setminus\{z\})$. Moreover, straightforward computations yield the existence of $C>0$ such that
\begin{equation}\label{app:bnd:f}
|f(y)|\leq C|z-y|^{1-n}\hbox{ for all }y\in V\setminus\{z\}.
\end{equation}
Computing $\Delta_{\tg}r^{2-n}$ in the distribution sense yields
$$\Delta_{\tg}r^{2-n}=f+K_z\delta_z\hbox{ in }{\mathcal D}'(V),$$
where $K_z:=(n-2)\int_{\partial B_{1}(0)}(\nu(y),y)_{\tg(z)}r(y)^{2-n}\, dv_{\tg(z)}>0$. Moreover, $\lim_{z\to x_0}K_z=K_{x_0}>0$.

\medskip\noindent We define $h$ such that
$$\left\{\begin{array}{ll}
\Delta_{\tg}h=f &\hbox{ in }V\\
h=0 & \hbox{ on }\partial V
\end{array}\right\}$$
It follows from \eqref{app:bnd:f} and elliptic theory that $h$ is well defined and that $h\in H_{2,0}^p(V)$ for all $p\in \left(1,\frac{n}{n-1}\right)$ and $h\in C^{1,\theta}_{loc}(V\setminus\{z\})$. Moreover, there exists $C>0$ such that
\begin{equation}\label{app:bnd:h:h2}
\Vert h\Vert_{H_2^p}\leq C(p)\hbox{ for all }p\in \left(1,\frac{n}{n-1}\right).
\end{equation}

\medskip\noindent We claim that for any $\alpha\in (n-3,n-2)$, there exists $C(\alpha)>0$ such that
$$|h(y)|\leq C(\alpha)|y-z|^{-\alpha}$$
for all $y\in V\setminus\{z\}$.

\smallskip\noindent We prove the claim. We let $\epsilon>0$ be a small parameter and we define
$$h_\epsilon(y):=\epsilon^\alpha h(z+\epsilon y)\hbox{ and }f_\epsilon(y):=\epsilon^{2+\alpha} f(z+\epsilon y)$$
for all $y\in B_2(0)\setminus \bar{B}_{1/2}(0)$. We then have that
\begin{equation}\label{app:eq:he}
\Delta_{\tg_\epsilon}h_\epsilon=f_\epsilon\hbox{ in }B_2(0)\setminus \bar{B}_{1/2}(0),
\end{equation}
where $\tg_\epsilon=\tg(\epsilon\cdot)$. Since $\alpha>n-3$, we have with \eqref{app:bnd:f} that
\begin{equation}\label{app:bnd:fe}
|f_\epsilon(y)|\leq C \epsilon^{\alpha-(n-3)}|y|^{1-n}\leq 2^{n-1}C
\end{equation}
for all $y\in B_2(0)\setminus \bar{B}_{1/2}(0)$. We fix $p:=\frac{n}{\alpha+2}\in \left(1,\frac{n}{n-1}\right)$ and $q:=\frac{n}{\alpha}$. A change of variable, Sobolev's embedding theorem and \eqref{app:bnd:h:h2} yield
\begin{equation}\label{app:bnd:he}
\Vert h_\epsilon\Vert_{L^q(B_2(0)\setminus \bar{B}_{1/2}(0))}\leq C\Vert h\Vert_q\leq C\Vert h\Vert_{H_2^p}\leq C
\end{equation}
for all $\epsilon>0$ small. It then follows from \eqref{app:eq:he}, \eqref{app:bnd:fe} and $\eqref{app:bnd:he}$ that there exists $C>0$ such that
$$|h_\epsilon(y)|\leq C\hbox{ for all }y\in\rn\hbox{ such that }|y|=1.$$
Therefore, coming back to $h$, we get that $|h(y)|\leq C|y-z|^{-\alpha}$ for all $|y-z|=\epsilon$. Since $\epsilon$ can be chosen arbitrary small and $h$ is bounded outside $y$, the claim is proved.

\medskip\noindent We now set $\Gamma_z:=\frac{1}{K_z}\left(r^{2-n}-h\right)$. It follows from the above estimates that $\Gamma$ satisfies \eqref{app:ppty:gamma}. This ends Step 4.

\medskip\noindent We define $\mu_x:=\tilde{G}_x-\Gamma_x-\Gamma_{x^\star}$. It follows from Steps 2 and 3 above that
\begin{equation}\label{app:eq:mu}
\Delta_{\tg}\mu_x=-\frac{1}{|\Omega|}\hbox{ in }{\mathcal D}'(V).
\end{equation}
Moreover, we have that $\mu_x\in H_1^q(V\setminus \{x,x^\star\})$ for all $q>1$ and that
\begin{equation}\label{app:bnd:mu:lp}
\Vert \mu_x\Vert_p\leq C(p)\hbox{ for all }p\in \left(1,\frac{n}{n-2}\right).
\end{equation}

\medskip\noindent{\bf Step 5:}
We claim that for all $V'\subset V$, there exists $C(V')>0$ such that
\begin{equation}\label{app:bnd:mu:infinity}
\Vert \mu_x\Vert_{L^\infty(V')}\leq C(V'),
\end{equation}
where $C(V')$ is independent of $x$.

\medskip\noindent We prove the claim. Since $x\in \Omega\cap V$, we have that $\tg=\xi$ in a neighborhood of $x$, and then $\tg$ is hypoelliptic around $x$: therefore, it follows from  \eqref{app:eq:mu} that $\mu_x$ is $C^\infty$ around $x$. Similarly, around $x^\star\in V\cap \omegabar^c$, $\tg=(\varphi\circ \tpi\circ \varphi^{-1})^\star \xi$ is also hypoelliptic, and therefore, $\mu_x$ is $C^\infty$ around $x^\star$. It then follows that $\mu_x\in H_1^q(V)$ for $q>1$ and \eqref{app:eq:mu} rewrites
$$\int_V (\nabla\mu_x,\nabla\psi)_{\tg}\, dv_{\tg}=-\frac{1}{|\Omega|}\int_V \psi\, dv_g\hbox{ for all }\psi\in C^\infty_c(V).$$
Therefore, it follows from Theorem 8.17 of \cite{gt} that $\mu_x\in L^\infty_{loc}(V)$ and that there exists $C(V,V',p)>0$ such that
$$\Vert\mu_x\Vert_{L^\infty(V')}\leq C(V,V',p)\left(1+\Vert \mu_x\Vert_{L^p(V)}\right)$$
for all $p>1$. Taking $p\in \left(1,\frac{n}{n-2}\right)$ and using \eqref{app:bnd:mu:lp}, we get \eqref{app:bnd:mu:infinity} and the claim is proved.

\medskip\noindent{\bf Step 6:} We are now in position to conclude. It follows from the definition of $\mu_x$ from \eqref{app:bnd:mu:infinity} and from \eqref{app:ppty:gamma} that there exists $C(V')>0$ such that
$$|\tilde{G}_x(y)|\leq C+C|x-y|^{2-n}+|x^\star-y|^{2-n}$$
for all $x,y\in V'$ such that $x\neq y$. As easily checked, one has that $|x^\star-y|\geq c|x-y|$ for all $x,y\in V'\cap \Omega$, and therefore
\begin{equation}\label{app:bnd:G:bord}
|G_x(y)|\leq C|x-y|^{2-n}
\end{equation}
for all $x,y\in V'\cap\Omega$ such that $x\neq y$. Recall that $V'$ is a small neighborhood of $x_0\in \partial\Omega$. Combining \eqref{app:bnd:G:bord} with Step 1, we get that there exists $\delta(\Omega)>0$ such that \eqref{app:bnd:G:bord} holds for all $x,y\in \Omega$ distinct such that $|x-y|\leq \delta(\Omega)$. For points $x,y$ such that $|x-y|\geq \delta(\Omega)$, this is Step 2. This ends the proof of the pointwise estimate \eqref{app:est:green:pointwise}.\end{proof}

\medskip\noindent{\it A.1.4. Extension to the boundary and regularity with respect to the two variables.}\par
\noindent We are now in position to extend the Green's function to the boundary.
\begin{prop}\label{app:prop:co} The Green's function extends continuously to $\omegabar\times\omegabar\setminus\{(x,x)/\, x\in\omegabar\}\to\rr$.
\end{prop}
\begin{proof} As above, we denote $G$ the Green's function for \eqref{app:NB}. We fix $x\in \partial\Omega$ and $y\in \omegabar\setminus\{x\}$ and we define
$$G_x(y):=\lim_{i\to +\infty}G(x_i,y) \hbox{ for all }y\in \omegabar\setminus\{x\},$$
where $(x_i)_i\in\Omega$ is any sequence such that $\lim_{i\to +\infty}x_i=x$.

\medskip\noindent We claim that this definition makes sense. It follows from \eqref{app:bnd:c2} that for all $\delta>0$, we have that
$$\Vert G_{x_i}\Vert_{C^{2}(\Omega\setminus \bar{B}_{\delta}(x))}\leq C(\delta)$$
for all $i$. Let $(i')$ be a subsequence of $i$: it then follows from Ascoli's theorem that there exists $G'\in C^{1}(\omegabar\setminus\{x\})$ and a subsequence $i"$ of $i'$ such that
$$\lim_{i\to +\infty}G_{x_i}=G'\hbox{ in }C^{1}_{loc}(\omegabar\setminus\{x\}).$$
Moreover, It follows from \eqref{app:est:green:pointwise} that
$|G'(y)|\leq C|x-y|^{2-n}$
for all $y\neq x$. We choose $u\in C^2(\omegabar)$ such that $\partial_\nu u=0$ on $\partial\Omega$. We then have that $\int_\Omega G_{x_i}\Delta u\, dy=u(x_i)-\bar{u}$ for all $i$. Letting $i\to +\infty$ yields
$$\int_\Omega G'\Delta u\, dy=u(x)-\bar{u},$$
and then it follows from Lemma \ref{app:lem:unic} that $G'$ does not depend of the choice of the sequence $(x_i)$ converging to $x$. We then let $G_x:=G'$ and the definition above makes sense.

\medskip\noindent We claim that $G\in C^0(\omegabar\times\omegabar\setminus\{(x,x)/\, x\in\omegabar\})$. We only sketch the proof since it is similar to the proof of the extension to the boundary. We fix $x\in \omegabar$ and we let $(x_i)_i$ be such that $\lim_{i\to +\infty}x_i=x$. Arguing as above, we get that any subsequence of $(G_{x_i})$ admits another subsequence that converges to some function $G"$ in $C^{1}_{loc}(\omegabar\setminus\{x\})$. We choose $u\in C^2(\omegabar)$ such that $\partial_\nu u$ vanishes on $\partial\Omega$ and we get that $\int_\Omega G_{x_i}\Delta u\, dy=u(x_i)-\bar{u}$ for all $i$. With the pointwise bound \eqref{app:est:green:pointwise}, we pass to the limit and get that $\int_\Omega G"\Delta u\, dy=u(x)-\bar{u}$: it then follows from Lemma \ref{app:lem:unic} that $G"=G_x$, and then $(G_{x_i})$ converges uniformly to $G_x$ outside $x$. The continuity of $G$ outside the diagonal follows immediately.
\end{proof}

\medskip\noindent It is essential to assume that $G$ satisfies point (ii) of the definition: indeed, for any $c:\Omega\to\rr$, the function $(x,y)\mapsto G(x,y)+c(x)$ satisfies (i) and (iii), but it is not continous outside the diagonal if $c$ is not continuous.

\medskip\noindent{\it A.1.5. Symmetry.}
\begin{prop}\label{app:sym}
Let $G$ be the Green's function for \eqref{app:NB}. Then $G(x,y)=G(y,x)$ for all $x,y\in \Omega\times\Omega$, $x\neq y$.
\end{prop}
\begin{proof}
Let $f\in C^\infty_c(\Omega)$ be a smooth compactly supported function. We define
$$F(x):=\int_\Omega G(y,x)(f-\bar{f})(y)\, dy\hbox{ for all }x\in\omegabar.$$
It follows from \eqref{app:est:green:pointwise} and Proposition \ref{app:prop:co} above that $F\in C^0(\omegabar)$. We fix $g\in C^\infty_c(\Omega)$ and we let $\varphi,\psi\in C^2(\omegabar)$ be such that
$$\left\{\begin{array}{ll}
\Delta \varphi=f-\bar{f}&\hbox{ in }\Omega\\
\partial_\nu \varphi=0&\hbox{ in }\partial\Omega\\
\bar{\varphi}=0&
\end{array}\right.\hbox{ and }\left\{\begin{array}{ll}
\Delta \psi=g-\bar{g}&\hbox{ in }\Omega\\
\partial_\nu \psi=0&\hbox{ in }\partial\Omega\\
\bar{\psi}=0&
\end{array}\right.$$
It follows from Fubini's theorem (which is valid here since $G\in L^1(\Omega\times\Omega)$ due to \eqref{app:est:green:pointwise} and Proposition \ref{app:prop:co}) that
\begin{eqnarray*}
\int_{\Omega}(F-\bar{F})g\, dx&=& \int_\Omega F(g-\bar{g})\, dx=\int_\Omega F\Delta\psi\, dx\\
&=&\int_\Omega (f-\bar{f})(y)\left(\int_\Omega G(y,x)\Delta\psi(x)\, dx\right)\, dy=\int_\Omega (\Delta\varphi)\psi\, dy\\
&=&\int_\Omega\varphi\Delta\psi\, dy=\int_\Omega \varphi(g-\bar{g})\, dy=\int_\Omega g\varphi\, dy,
\end{eqnarray*}
and therefore $\int_\Omega (F-\bar{F}-\varphi)g\, dx=0$ for all $g\in C^\infty_c(\Omega)$. Since $F,\varphi\in C^0(\omegabar)$, we then get that $F(x)=\varphi(x)+\bar{F}$ for all $x\in\Omega$. We now fix $x\in\Omega$. Using the definition of the Green's function and the definition of $F$, we then get that
$$\int_\Omega G(y,x)(f-\bar{f})(y)\, dy=\int_\Omega G(x,y)(f-\bar{f})(y)\, dy+\frac{1}{|\Omega|}\int_\Omega\left(\int_\Omega G(y,x)\, dx\right) (f-\bar{f})(y)\, dy,$$
and then, setting
$$H_x(y):=G(y,x)-G(x,y)-\frac{1}{|\Omega|}\int_\Omega G(y,z)\, dz$$
for all $y\in\Omega\setminus \{x\}$, we get that
$$0=\int_\Omega H_x (f-\bar{f})\, dy=\int_\Omega (H_x-\bar{H_x})f\, dy$$
for all $f\in C_c^\infty(\Omega)$. Therefore, $H_x\equiv \bar{H_x}$, which rewrites
$$G(y,x)-G(x,y)=\frac{1}{|\Omega|}\int_\Omega(G(y,z)-G(x,z))\, dz+h(x),$$
for all $x\neq y$, where $h(x):=\frac{1}{|\Omega|}\int_\Omega G(z,x)\, dz-\frac{1}{|\Omega^2|}\int_{\Omega\times\Omega}G(s,t)\, ds\, dt$ for all $x\in \omegabar$. Exchanging $x,y$ yields $h(x)+h(y)=0$ for all $x\neq y$, and then $h\equiv 0$ since $h$ is continuous. Therefore, we get that
\begin{equation}\label{app:eq:sym}
G(y,x)-G(x,y)=\frac{1}{|\Omega|}\int_\Omega(G(y,z)-G(x,z))\, dz=\bar{G_y}-\bar{G_y}
\end{equation}
for all $x\neq y$.  The normalization (i) in the definition of the Green's function then yields Proposition \ref{app:sym}.\end{proof}

\medskip\noindent If one does not impose the normalization (ii), we have already remarked that we just get $G': (x,y)\mapsto G(x,y)+c(x)$ where $G$ is the Green's function as defined in the definition and $c$ is any function. We then get that $G'(x,y)-G'(y, x)=c(x)-c(y)$ for all $x\neq y$, which is not vanishing when $c$ is nonconstant.

\medskip\noindent These different lemmae and estimates prove Theorem \ref{app:th:green:exist}.

\medskip\noindent{\bf A.2. Asymptotic analysis}\par
\noindent This section is devoted to the proof of general asymptotic estimates for the Green's function. As a byproduct, we will get the control \eqref{app:est:der:green} of the derivatives of Proposition \ref{app:prop:upp:bnd}. The following proposition is the main result  of this section:
\begin{prop}\label{app:prop:asymp:green}
Let $G$ be the Green's function for \eqref{app:NB}. Let $(\xa)_\alpha\in\Omega$ and let $(\ra)_{\alpha}\in (0,+\infty)$ be such that $\lim_{\alpha\to +\infty}\ra=0$.

\smallskip\noindent Assume that $$\lim_{\alpha\to +\infty}\frac{d(\xa,\partial\Omega)}{\ra}=+\infty.$$
Then for all $x,y\in\rn$, $x\neq y$, we have that
$$\lim_{\alpha\to +\infty}\ra^{n-2}G(\xa+\ra x,\xa+\ra y)=k_n|x-y|^{2-n}.$$
Moreover, for fixed $x\in\rn$, this convergence holds uniformly in $C^2_{loc}(\rn\setminus\{x\})$.

\smallskip\noindent Assume that $$\lim_{\alpha\to +\infty}\frac{d(\xa,\partial\Omega)}{\ra}=\rho\geq 0.$$
Then $\lim_{\alpha\to +\infty}\xa=x_0\in\partial\Omega$. We choose a chart $\varphi$ at $x_0$ as in Lemma \ref{app:lem:ext} and we let $(x_{\alpha, 1},x_\alpha')=\varphi^{-1}(\xa)$. Then for all $x,y\in\rn\cap\{x_1\leq 0\}$, $x\neq y$, we have that
$$\lim_{\alpha\to +\infty}\ra^{n-2}G(\varphi((0,\xa')+\ra x),\varphi((0,\xa')+\ra y)=k_n\left(|x-y|^{2-n}+|\pi^{-1}(x)-y|^{2-n}\right),$$
where $\pi^{-1}(x_1,x')=(-x_1,x')$.
Moreover, for fixed $x\in\overline{\rnm}$, this convergence holds uniformly in $C^2_{loc}(\overline{\rnm}\setminus\{x\})$.
\end{prop}

\smallskip\noindent{\it Proof of Proposition \ref{app:prop:asymp:green}:}

\smallskip\noindent{\bf Step 1:} We first assume that
\begin{equation}\label{app:lim:d:infty}
\lim_{\alpha\to +\infty}\frac{d(\xa,\partial\Omega)}{\ra}=+\infty.
\end{equation}
We define
$$\tilde{G}_\alpha(x,y):=\ra^{n-2}G(\xa+\ra x,\xa+\ra y)$$
for all $\alpha\in\nn$ and all $x,y\in \Omega_\alpha:=\ra^{-1}(\Omega-\xa)$, $x\neq y$. We fix $x\in\rn$. It follows from Theorem  \ref{app:th:green:exist} that $\tilde{G}_\alpha\in C^2(\Omega_\alpha\times\Omega_\alpha\setminus \{(x,x)/\,x\in\Omega_\alpha\})$ and that
\begin{equation}\label{app:eq:tGa}
\Delta (\tilde{G}_\alpha)_x=-\frac{\ra^n}{|\Omega|} \hbox{ in }\Omega_\alpha\setminus\{x\}
\end{equation}
for $\alpha\in\nn$ large enough. Moreover, it follows from \eqref{app:est:green} that there exists $C>0$ such that
\begin{equation}\label{app:est:tGa}
|(\tilde{G}_\alpha)_x(y)|\leq C|y-x|^{2-n}
\end{equation}
for all $\alpha\in\nn$ and all $y\in \Omega_\alpha\setminus\{x\}$. It then follows from \eqref{app:lim:d:infty}, \eqref{app:eq:tGa}, \eqref{app:est:tGa} and standard elliptic theory that, up to a subsequence, there exists $\tilde{G}_x\in C^2(\rn\setminus\{x\})$ such that
\begin{equation}\label{app:cv:tGa}
\lim_{\alpha\to +\infty}(\tilde{G}_\alpha)_x=\tilde{G}_x\hbox{ in }C^2_{loc}(\rn\setminus\{x\}).
\end{equation}
with
\begin{equation}\label{app:est:tG}
|\tilde{G}_x(y)|\leq C|y-x|^{2-n}
\end{equation}
for all $y\in\rn\setminus\{x\}$. We consider $f\in C^\infty_c(\rn)$ and we define $f_\alpha(y):=f(\ra^{-1}(y-\xa-\ra x))$: it follows from \eqref{app:lim:d:infty} that $f_\alpha\in C^\infty_c(\Omega)$ for $\alpha\in\nn$ large enough. Applying Green's representation formula yields
$$f_\alpha(\xa+\ra x)-\overline{f_\alpha}=\int_\Omega G(\xa+\ra x, z)\Delta f_\alpha(z)\, dz.$$
With a change of variable, this equality rewrites
$$f(x)=\int_{\rn}\tilde{G}_\alpha(x,y)\Delta f(y)\, dy+\overline{f_\alpha}$$
for $\alpha\in\nn$ large enough. With \eqref{app:est:tGa}, \eqref{app:cv:tGa} and the definition of $f_\alpha$, we get that
$$f(x)=\int_{\rn}\tilde{G}_x\Delta f\, dy,$$
and then
$$\Delta(\tilde{G}_x-k_n|\cdot-x|^{2-n})=0\hbox{ in }{\mathcal D}'(\rn).$$
The ellipticity of the Laplacian, \eqref{app:est:tG} and Liouville's Theorem yield
$$\tilde{G}_x(y)=k_n|y-x|^{2-n}\hbox{ for all }y\neq x.$$
This ends Step 1.

\medskip\noindent{\bf Step 2:}
\begin{equation*}
\lim_{\alpha\to +\infty}\frac{d(\xa,\partial\Omega)}{\ra}=\rho\geq 0.
\end{equation*}
We take $\varphi$ as in the statement of the Proposition and we define
$$\tilde{G}_\alpha(x,y):=\ra^{n-2}G(\varphi((0,\xa')+\ra x),\varphi((0,\xa')+\ra y)$$
for all $x,y\in\rnm$, $x\neq y$ with $\alpha\in\nn$ large enough. We fix $x\in\rnm$ and we symmetrize $\tilde{G}$ as usual:
$$\hat{G}_\alpha(x,y):=\tilde{G}_\alpha(x,\tpi(y))$$
for all $y\in\rn$ close enough to $0$ and where, as above, $\tpi: \rn\to\overline{\rnm}$. For simplicity, we assume that $x\in\overline{\rnm}$ (only the notation has to be change in case $x\in\rn_+$). As in the first case, we get that there exists $C>0$ such that
\begin{equation*}
|\hat{G}_\alpha(x,y)|\leq C\left(|y-x|^{2-n}+|y-\pi^{-1}(x)|^{2-n}\right)
\end{equation*}
for all $y\neq x,\tilde{\pi}(x)$ and there exists $\hat{G}_x\in C^2(\rn\setminus\{x,\pi^{-1}(x)\})$ such that
$$\lim_{\alpha\to +\infty}(\hat{G}_{\alpha})_x=\hat{G}_x\hbox{ in }C^2_{loc}(\rn\setminus\{x,\pi^{-1}(x)\}).$$
Moreover, letting $L=d\varphi_0$ be the differential of $\varphi$ at $0$, arguing again as in the first case, we have that
$$\Delta_{L^\star\xi}\hat{G}_x=\delta_{x}+\delta_{\pi^{-1}(x)}\hbox{ in }{\mathcal D}'(\rnm).$$
Therefore, with a change of variable, we get that
$$\Delta_{\xi}(\hat{G}_x\circ L^{-1})=\delta_{L(x)}+\delta_{L\circ\pi^{-1}(x)}\hbox{ in }{\mathcal D}'(\rnm),$$
and then
$$\Delta_{\xi}\left(\hat{G}_x\circ L^{-1}-k_n\left(|L(x)-\cdot|^{2-n}+|L\circ\pi^{-1}(x)-\cdot|^{2-n}\right)\right)=0\hbox{ in }{\mathcal D}'(\rnm),$$
Arguing as above, we get that $\hat{G}_x\circ L^{-1}=k_n\left(|L(x)-y|^{2-n}+|L\circ\pi^{-1}(x)-y|^{2-n}\right)$, and then
$$\hat{G}_x=k_n\left(|\cdot-x|^{2-n}+|\cdot-\pi^{-1}(x)|^{2-n}\right)$$
since $L$ is an orthogonal transformation. This ends Step 2.

\medskip\noindent Proposition \ref{app:prop:asymp:green} is a direct consequence of Steps 1 and 2.\hfill$\Box$

\medskip\noindent We now prove Proposition \ref{app:prop:upp:bnd}:

\begin{cor} Let $G$ be the Green's function for \eqref{app:NB}. Then there $C,M>0$ such that
$$\frac{1}{C|x-y|^{n-2}}-M\leq G(x,y)\leq \frac{C}{|x-y|^{n-2}}$$
and
$$|\nabla_yG(x,y)|\leq \frac{C}{|x-y|^{n-1}}$$
for all $x,y\in\Omega$, $x\neq y$.
\end{cor}

\smallskip\noindent{\it Proof of the corollary:} We claim that there exists $m\in\rr$ such that
\begin{equation}\label{app:lower:bnd:G}
G(x,y)\geq -m\hbox{ for all }x,y\in\Omega, \; x\neq y.
\end{equation}
We argue by contradiction and we assume that there exists $(\xa)_\alpha,(\ya)_{\alpha}\in\Omega$ such that
\begin{equation}\label{app:hyp:moins:infty}
\lim_{\alpha\to +\infty}G(\xa,\ya)=-\infty.
\end{equation}
Assume that $\lim_{\alpha\to +\infty}|\ya-\xa|=0$. We then define $\ra:=|\ya-\xa|$ and we apply Proposition \ref{app:prop:asymp:green}:

\smallskip\noindent If $\lim_{\alpha\to +\infty}\frac{d(\xa,\partial\Omega)}{\ra}=+\infty$, we have that
$$|\ya-\xa|^{n-2}G(\xa,\ya)=\ra^{n-2}G\left(\xa, \xa+\ra\frac{\ya-\xa}{|\ya-\xa|}\right)=k_n+o(1)$$
when $\alpha\to +\infty$. This contradicts \eqref{app:hyp:moins:infty}.

\smallskip\noindent If $d(\xa,\partial\Omega)=O(\ra)$ when $\alpha\to +\infty$, we get also a contradiction.

\medskip\noindent This proves that $\lim_{\alpha\to +\infty}|\xa-\ya|\neq 0$. Therefore, with \eqref{app:est:green}, we get that $G(\xa,\ya)=O(1)$ when $\alpha\to +\infty$: this contradicts \eqref{app:hyp:moins:infty}. Therefore, there exists $m$ such that \eqref{app:lower:bnd:G} holds.

\medskip\noindent We define $M:=m+1$. With \eqref{app:est:green}, there exists also $C>0$ such that $|G(x,y)|\leq C|x-y|^{2-n}$ for all $x\neq y$. We claim that there exists $c>0$ such that
\begin{equation}\label{app:bnd:low:G}
G(x,y)+M\geq c|x-y|^{2-n}
\end{equation}
for all $x\neq y$. Here again, we argue by contradiction and we assume that there exists $(\xa)_{\alpha}, (\ya)_{\alpha}\in\Omega$ such that
\begin{equation}\label{app:hyp:G:0}
\lim_{\alpha\to +\infty}|\xa-\ya|^{n-2}(G(\xa,\ya)+M)=0.
\end{equation}
Since $G+M\geq 1$, it follows from \eqref{app:hyp:G:0} that $\lim_{\alpha\to+\infty}|\xa-\ya|=0$. Therefore, as above, we get that the limit of the left-hand-side in \eqref{app:hyp:G:0} is positive: a contradiction. This proves that \eqref{app:bnd:low:G} holds. In particular, this proves the first part of the corollary.

\medskip\noindent Concerning the estimate of the gradient, we argue by contradiction and we use again Proposition \ref{app:prop:asymp:green}. We just sketch the proof. Assume by contradiction that there exists $(\xa)_{\alpha}, (\ya)_{\alpha}\in\Omega$ such that
$$\lim_{\alpha\to +\infty}|\ya-\xa|^{n-1}|\nabla_y G(\xa,\ya)|=+\infty.$$
It follows from \eqref{app:bnd:c2} that $\lim_{\alpha\to +\infty}|\ya-\xa|=0$. We set $\ra:=|\ya-\xa|$. Assume that $\ra=o(d(\xa,\partial\Omega))$ when $\alpha\to +\infty$. It then follows from Proposition \ref{app:prop:asymp:green} that
$$\lim_{\alpha\to +\infty}|\ya-\xa|^{n-1}|\nabla_y G(\xa,\ya)|=\frac{1}{\omega_{n-1}},$$
which contradicts the hypothesis. The proof goes the same way when $d(\xa,\partial\Omega)=O(\ra)$ when $\alpha\to +\infty$. This ends the proof of the gradient estimate.\hfill$\Box$

\section*{Appendix B: Projection of the test functions}\label{sec:proj}

\begin{prop}\label{lemmaprojection}
Let $\left(\xa\right)$ be a sequence of points in $\bar{\Om}$ and let $\left(\ma\right)$ be a sequence of positive real numbers such that $\ma\to 0$ as $\epstozero$. We set
$$\Ua(x)=\ma^{\frac{n-2}{2}} \left(\left\vert x-\xa\right\vert^2 +\ma^2\right)^{1-\frac{n}{2}}\hskip.1cm.$$
There exists $\Va \in C^1(\bar{\Om})$ which satisfies
\begin{equation}\label{def:Va}
\left\{\begin{array}{ll}
\Delta \Va = c_n\Ua^{2^\star-1} - c_n\left\vert \Om\right\vert^{-1} \int_\Om \Ua^{2^\star-1}\, dx&\hbox{ in }\Om\\
\partial_\nu \Va=0&\hbox{ on }\partial \Om
\end{array}\right.
\end{equation}
such that the following asymptotics hold for any sequence of points $\left(\ya\right)$ in $\bar{\Om}$~:

\smallskip (i) If $\xa\in \partial \Om$, then
$$\Va \left(\ya\right) = \bigl(1+o(1)\bigr) \Ua\left(\ya\right) + O\left(\ma^{\frac{n-2}{2}}\right)\hskip.1cm.$$

\smallskip (ii) If $d\left(\xa,\partial \Om\right)\not\to 0$, then
$$\Ua \left(\ya\right) = (1+o(1))\Ua\left(\ya\right) + O\left(\ma^{\frac{n-2}{2}}\right)\hskip.1cm.$$

\smallskip (iii) If $d\left(\xa,\partial\Om\right)\to 0$ but $\frac{d\left(\xa,\partial\Om\right)}{\ma}\to +\infty$ as $\epstozero$, then
$$\Va\left(\ya\right) =  \bigl(1+o(1)\bigr) \Ua\left(\ya\right) + \tUe\left(\ya\right) + O\left(\ma^{\frac{n-2}{2}}\right)$$
where
$$\tUe(x) = \ma^{\frac{n-2}{2}}\left(\ma^2 + \left\vert x-\pip^{-1}\left(\xa\right)\right\vert^2\right)^{1-\frac{n}{2}}$$
with $\pip:=\varphi\circ\pi\circ\varphi^{-1}$ where $\varphi$ is a chart at $x_0:=\lim_{\alpha\to +\infty}\xa$ as in Lemma \ref{lem:ext}.

\smallskip\noindent In addition, we have that
$$\left\{\begin{array}{ll}
\Ua-\Ua=o(\Ua) + O(\ma^{\frac{n-2}{2}}) &\hbox{ in cases (i) and (ii)}\\
|\Va-\Ua|\leq C \left(\frac{\ma}{\ma^2+d(\xa,\partial\Omega)^2}\right)^{\frac{n-2}{2}}+ o(\Ua) + O(\ma^{\frac{n-2}{2}})&\hbox{ in case (iii).}
\end{array}\right.$$

\smallskip In any case, there exists $C>0$ such that
\begin{equation}\label{ineq:Ua:Va}
\frac{1}{C}\Ua \le \Va\le C \Ua\hskip.1cm.
\end{equation}

\end{prop}

\medskip\noindent {\it Proof of Proposition \ref{lemmaprojection}:} We let $\Va\in C^2(\omegabar)$ be as in \eqref{def:Va}. Indeed, $\Va$ is defined up to the addition of a constant: therefore, $\overline{\Va}$ will be determined later on. Let $(\ya)_\alpha\in\Omega$: Green's representation formula yields

$$\Va(\ya)-\overline{\Va}=\int_\Omega G(\ya, y)\left(c_n\Ua^{\crit-1}-c_n|\Omega|^{-1}\int_\Omega \Ua^{\crit-1}\right)\, dy$$
for all $\alpha\in\nn$ where $G$ is the Green's function for \eqref{app:NB} with vanishing average. With the explicit expression of $\Ua$, we get that
\begin{equation}\label{est:Va:1}
\Va(\ya)-\overline{\Va}=c_n\int_\Omega G(\ya, y)\Ua^{\crit-1}\, dy+O(\ma^{\frac{n-2}{2}})
\end{equation}
for all $\alpha\in\nn$. The estimate of $\Va(\ya)$ goes through five steps.

\medskip\noindent{\bf Step \ref{lemmaprojection}.1.} We first assume that $\lim_{\alpha\to +\infty}|\ya-\xa|\neq 0$. It then follows from \eqref{est:Va:1}, the pointwise estimates \eqref{app:est:green} on the Green's function and the explicit expression of $\Ua$ that
\begin{eqnarray*}
\Va(\ya)&=&\overline{\Va}+(G(\ya,\xa)+o(1))\int_\Omega c_n\Ua^{\crit-1}\, dx\\
&&-\left(\int_\Omega G(\ya,x)\, dx\right)\int_\Omega c_n\Ua^{\crit-1}\, dx
\end{eqnarray*}
when $\alpha\to +\infty$. It then follows from this estimate that
\begin{equation*}
\Va(\ya)=\overline{\Va}+O(\ma^{\frac{n-2}{2}})
\end{equation*}
when $\alpha\to +\infty$ and that there exists $K>0$ independent of $(\ya)_\alpha$ such that
\begin{equation*}
\Va(\ya)\geq \overline{\Va}-K\ma^{\frac{n-2}{2}}
\end{equation*}
for all $\alpha\in\nn$.

\medskip\noindent{\bf Step \ref{lemmaprojection}.2.} We claim that
\begin{equation}\label{lim:R:a:ratio}
\lim_{R\to +\infty}\lim_{\alpha\to +\infty}\frac{\int_{\Omega\setminus B_{R\ma}(\xa)} G(\ya, y)\Ua^{\crit-1}\, dy}{\Ua(\ya)}=0\hbox{ if }\lim_{\alpha\to +\infty}|\ya-\xa|=0.
\end{equation}
We prove the claim. We let $R_0>0$ such that $\Omega\subset B_{R_0}(\xa)$ for all $\alpha\in\nn$. It follows from the explicit expression of $\Ua$ and of \eqref{app:est:green} that
\begin{equation}\label{ineq:proj:1}
\left| \int_{\Omega\setminus B_{R\ma}(\xa)} G(\ya, y)\Ua^{\crit-1}\, dy\right|\leq C\int_{B_{R_0}(\xa)\setminus B_{R\ma}(\xa)}|\ya-x|^{2-n}\frac{\ma^{\frac{n+2}{2}}}{\left(\ma^2+|x-\xa|^2\right)^{\frac{n+2}{2}}}\, dx
\end{equation}
for all $\alpha\in\nn$. We define
$$D_\alpha:=\left\{x\in\rn/\, |x-\ya|\geq \frac{1}{2}\sqrt{\ma^2+|\xa-\ya|^2}\right\}\hbox{ for all }\alpha\in\nn.$$
We split the RHS of \eqref{ineq:proj:1} in two terms. On the one hand, we have that
\begin{eqnarray}
&&\int_{D_\alpha\cap(B_{R_0}(\xa)\setminus B_{R\ma}(\xa))}|\ya-x|^{2-n}\frac{\ma^{\frac{n+2}{2}}}{\left(\ma^2+|x-\xa|^2\right)^{\frac{n+2}{2}}}\, dx\nonumber\\
&& \leq\frac{C}{\left(\ma^2+|\xa-\ya|^2\right)^{\frac{n-2}{2}}}\int_{\rn \setminus B_{R\ma}(\xa)}\frac{\ma^{\frac{n+2}{2}}}{\left(\ma^2+|x-\xa|^2\right)^{\frac{n+2}{2}}}\, dx\nonumber\\
&&\leq C \Ua(\ya)\int_{\rn \setminus B_{R}(0)}\frac{1}{\left(1+|x|^2\right)^{\frac{n+2}{2}}}\, dx\label{ineq:proj:2}
\end{eqnarray}
for all $\alpha\in\nn$. On the other hand, as easily checked, there exists $\epsilon_0>0$ such that
$$x\not\in D_\alpha\;\Rightarrow\; |x-\xa|^2+\ma^2\geq \epsilon_0\left(|\ya-\xa|^2+\ma^2\right)$$
for all $\alpha\in\nn$. Consequently, we have that
\begin{eqnarray}
&&\int_{D_\alpha^c\cap(B_{R_0}(\xa)\setminus B_{R\ma}(\xa))}|\ya-x|^{2-n}\frac{\ma^{\frac{n+2}{2}}}{\left(\ma^2+|x-\xa|^2\right)^{\frac{n+2}{2}}}\, dx\nonumber\\
&& \frac{C\ma^\frac{n+2}{2}}{\left(\ma^2+|\xa-\ya|^2\right)^{\frac{n+2}{2}}}\int_{D_\alpha^c}|\ya-x|^{2-n}\, dy\nonumber\\
&&\leq C \Ua(\ya)\frac{\ma^2}{\ma^2+|\xa-\ya|^2}=o(\Ua(\ya))\label{ineq:proj:3}
\end{eqnarray}
if $\ma=o(|\xa-\ya|)$ when $\alpha\to +\infty$. In case $|\ya-\xa|=O(\ma)$ when $\alpha\to +\infty$, it is easily checked that for $R$ large enough, $D_\alpha^c\cap(B_{R_0}(\xa)\setminus B_{R\ma}(\xa))=\emptyset$ for all $\alpha\in\nn$. Therefore \eqref{ineq:proj:3} always holds.

\medskip\noindent Plugging \eqref{ineq:proj:2} and \eqref{ineq:proj:3} into \eqref{ineq:proj:1} yields \eqref{lim:R:a:ratio}. This ends Step \ref{lemmaprojection}.2.

\medskip\noindent It follows from \eqref{est:Va:1} and \eqref{lim:R:a:ratio} that
\begin{equation}\label{est:Va:2}
\Va(\ya)=\overline{\Va}+c_n\int_{\Omega\cap B_{R\ma}(\xa)} G(\ya, y)\Ua^{\crit-1}\, dy+(o(1)+\epsilon_R)\Ua(\ya)
\end{equation}
if $\lim_{\alpha\to +\infty}|\ya-\xa|=0$ when $\alpha\to +\infty$ where $\lim_{R\to +\infty}\epsilon_R=0$.

\medskip\noindent{\bf Step \ref{lemmaprojection}.3.} Assume that
\begin{equation}\label{hyp:step:3}
\lim_{\alpha\to +\infty}|\ya-\xa|=0\hbox{ and }\lim_{\alpha\to +\infty}\frac{d(\xa,\partial\Omega)}{\ma}=+\infty.
\end{equation}
We claim that
\begin{equation}\label{concl:3}
\Va(\ya)=\overline{\Va}+\left\{\begin{array}{ll}
(1+o(1))\Ua(\ya) &\hbox{ if }\lim_{\alpha\to +\infty}d(\xa, \partial\Omega)\neq 0\\
(1+o(1))(\Ua(\ya)+\tUe(\ya)) &\hbox{ if }\lim_{\alpha\to +\infty}d(\xa, \partial\Omega)=0
\end{array}\right.
\end{equation}
when $\alpha\to +\infty$.

\medskip\noindent The proof of \eqref{concl:3} goes through several steps. First note that due to \eqref{hyp:step:3}, we have that $\Omega\cap B_{R\ma}(\xa)=B_{R\ma}(\xa)$ for $\alpha\in\nn$ large enough. Therefore,  with a change of variable, \eqref{est:Va:2} rewrites
\begin{eqnarray}\label{est:Va:3}
&&\Va(\ya)=\overline{\Va}\nonumber\\
&&+\Ua(\ya)\left(\int_{B_{R}(0)} \left(\ma^2+|\ya-\xa|^2\right)^{\frac{n-2}{2}}G(\ya, \xa+\ma x)c_nU_0^{\crit-1}\, dx\right)\nonumber\\
&&+\left(o(1)+\epsilon_R\right)\Ua(\ya)
\end{eqnarray}
for all $R>>1$ and $\alpha\to +\infty$. We distinguish two cases:

\medskip\noindent{\bf Case \ref{lemmaprojection}.3.1:} We assume that
\begin{equation}
|\ya-\xa|=O(\ma)\hbox{ when }\alpha\to +\infty.
\end{equation}
Then we claim that \eqref{concl:3} holds. We prove the claim. We define $\theta_\alpha:=\ma^{-1}(\ya-\xa)$ for all $\alpha\in\nn$, and we let $\theta_\infty:=\lim_{\alpha\to +\infty}\theta_\alpha$. Let $K$ be a compact subset of $\rn\setminus\{\theta_\infty\}$: it follows from Proposition \ref{app:prop:asymp:green} that
$$\ma^{n-2}G(\ya, \xa+\ma x)=(k_n+o(1))|x-\theta_\alpha|^{2-n}$$
when $\alpha\to +\infty$ uniformly for all $x\in K$. Moreover, the LHS is uniformly bounded from above by the RHS on bounded domains of $\rn$ when $\alpha\to +\infty$. It then follows from Lebesgue's theorem that \eqref{est:Va:3} rewrites
\begin{eqnarray*}
\Va(\ya)&=&\overline{\Va}+\Ua(\ya)\left(\int_{B_{R}(0)} \left(1+|\theta_\infty|^2\right)^{\frac{n-2}{2}}k_nc_nU_0^{\crit-1}(x)|x-\theta_\infty|^{2-n}\, dx+o(1)+\epsilon_R\right)\\
&=& \overline{\Va}+\Ua(\ya)\left(U_0(\theta_\infty)^{-1}\int_{B_{R}(0)} k_n\Delta U_0 (x)|x-\theta_\infty|^{2-n}\, dx+o(1)+\epsilon_R\right)\\
&=& \overline{\Va}+\Ua(\ya)\left(1+o(1)+\epsilon_R\right)
\end{eqnarray*}
since $\Delta(k_n|\cdot|^{2-n})=\delta_0$ in the distribution sense. Letting $R\to +\infty$ and $\alpha\to +\infty$ yields
$$\Va(\ya)=\overline{\Va}+\Ua(\ya)(1+o(1))
$$
when $\alpha\to +\infty$. As easily checked, this estimate yields \eqref{concl:3} in Case \ref{lemmaprojection}.3.1.

\medskip\noindent{\bf Case \ref{lemmaprojection}.3.2:} We assume that
\begin{equation*}
\lim_{\alpha\to +\infty}\frac{|\ya-\xa|}{\ma}=+\infty.
\end{equation*}
We claim that \eqref{concl:3} holds. We prove the claim.  We define $\ra:=|\ya-\xa|=o(1)$ when $\alpha\to +\infty$. Given $x\in B_R(0)$, we define
$$A_\alpha(x):=\left(\ma^2+|\ya-\xa|^2\right)^{\frac{n-2}{2}}G(\ya, \xa+\ma x)$$
for all $\alpha\in\nn$.

\medskip\noindent{\bf Case \ref{lemmaprojection}.3.2.1:} We assume in addition that
$$\lim_{\alpha\to +\infty}\frac{d(\xa,\partial\Omega)}{\ra}=+\infty.$$
We claim that in this case, we have that
\begin{equation}\label{lim:Aalpha:1}
\lim_{\alpha\to +\infty}A_\alpha(x)=k_n
\end{equation}
uniformly when $\alpha\to +\infty$. We prove the claim. Indeed, letting $\theta_\alpha:=\ra^{-1}(\ya-\xa)$ and using that $G$ is symmetric, we have that
$$A_\alpha(x)=(1+o(1))\ra^{n-2}G(\xa+\ra\frac{\ma x}{\ra}, \xa+\ra\theta_\alpha)$$
for all $\alpha\in\nn$ uniformly for $x$ in any fixed compact of $\rn$. Then \eqref{lim:Aalpha:1} follows from Proposition \ref{app:prop:asymp:green}.

\medskip\noindent{\bf Case \ref{lemmaprojection}.3.2.2:} We assume here that
$$\lim_{\alpha\to +\infty}\frac{d(\xa,\partial\Omega)}{\ra}=\rho\geq 0.$$
In this case, $\tUe$ is well defined. We claim that in this case, we have that
\begin{equation}\label{lim:Aalpha:2}
A_\alpha(x)=(k_n+o(1))\left(1+\frac{\tUe(\ya)}{\Ua(\ya)}\right)
\end{equation}
uniformly for $x$ in any fixed compact of $\rn$ when $\alpha\to +\infty$. We prove the claim. We denote $\varphi$ a chart as in Lemma \ref{lem:ext} and we define $(\xaun,\xa'):=\varphi^{-1}(\xa)$ and $(\yaun,\ya'):=\varphi^{-1}(\ya)$ for all  $\alpha\in\nn$. Defining
$$X_\alpha:=\left(\frac{\xaun}{\ra},0\right)+o(1)\hbox{ and }Y_\alpha:=\left(\frac{\yaun}{\ra},\frac{\ya'-\xa'}{\ra}\right),$$
using Proposition \ref{app:prop:asymp:green} and the symmetry of $G$, we get that
\begin{eqnarray}
A_\alpha(x)&=& (1+o(1))\ra^{n-2}G(\xa+\ma x, \ya)+o(1)\nonumber\\
&=& (1+o(1))\ra^{n-2}G(\varphi((0,\xa')+\ra X_\alpha), \varphi((0,\xa')+\ra Y_\alpha))+o(1)\nonumber\\
&=& k_n\left(\left|X_\alpha-Y_\alpha\right|^{2-n}+\left|Y_\alpha-\pi^{-1}(X_\alpha)\right|^{2-n}\right)+o(1)\nonumber\\
&=& k_n\left(1+\left|\left(\frac{(\yaun,\ya')}{\ra}-\frac{\pi^{-1}(\xaun,\xa')}{\ra}\right)\right|^{2-n}\right)+o(1)\label{lim:Aalpha:3}
\end{eqnarray}
since $d\varphi_0$ is an orthogonal transformation. independently, using again that $d\varphi_0$ is orthogonal, we have that
\begin{eqnarray*}
\frac{\tUe(\ya)}{\Ua(\ya)}&=&\left(\frac{\ma^2+|\ya-\pip^{-1}(\xa)|^2}{\ma^2+|\ya-\xa|^2}\right)^{-\frac{n-2}{2}}\\
&=& (1+o(1))\left|\left(\frac{\varphi^{-1}(\ya)-\pi^{-1}(\varphi^{-1}(\xa))}{\ra}\right)\right|^{2-n}
\end{eqnarray*}
when $\alpha\to +\infty$. Plugging this estimate into \eqref{lim:Aalpha:3} yields \eqref{lim:Aalpha:2}. This proves the claim.

\medskip\noindent Since
$$\int_{B_{R}(0)} c_nU_0^{\crit-1}\, dx=\int_{B_{R}(0)}\Delta U_0\, dx=-\int_{\partial B_{R}(0)} \partial_\nu U_0\, dx=\frac{(n-2)\omega_{n-1}R^n}{(1+R^2)^{n/2}}$$
for all $R>0$, it follows from \eqref{est:Va:3},  Cases \ref{lemmaprojection}.3.2.1 and \ref{lemmaprojection}.3.2.2 that
\begin{equation*}
\Va(\ya)=\overline{\Va}+\left\{\begin{array}{ll}
(1+o(1))\Ua(\ya) &\hbox{ if }\ra=o(d(\xa, \partial\Omega))\hbox{ when }\alpha\to +\infty\\
(1+o(1))(\Ua(\ya)+\tUe(\ya)) &\hbox{ if }d(\xa, \partial\Omega)=O(\ra)\hbox{ when }\alpha\to +\infty
\end{array}\right.
\end{equation*}
These estimates and a careful evaluation of the quotient $\Ua(\ya)^{-1}\tUe(\ya)$ yields \eqref{concl:3} in Case \ref{lemmaprojection}.3.2. This ends Case \ref{lemmaprojection}.3.2.

\medskip\noindent{\bf Step \ref{lemmaprojection}.4.} We assume that
\begin{equation}\label{hyp:step:4}
\lim_{\alpha\to +\infty}|\ya-\xa|=0\hbox{ and }\xa\in\partial\Omega.
\end{equation}
We claim that
\begin{equation}\label{concl:4}
\Va(\ya)=\overline{\Va}+\Ua(\ya)(1+o(1))
\end{equation}
when $\alpha\to +\infty$. We choose a chart $\varphi$ as in Lemma \ref{lem:ext}. In this case, \eqref{est:Va:2} rewrites
\begin{eqnarray}\label{est:Va:4}
&&\Va(\ya)=\overline{\Va}\nonumber\\
&&+\Ua(\ya)\left(\int_{B_{R}(0)\cap\rnm} c_n(1+o(1)){\mathcal T}_\alpha\, dx\right)\nonumber\\
&&+(o(1)+\epsilon_R)\Ua(\ya)
\end{eqnarray}
for all $R>>1$ and $\alpha\to +\infty$, where
$${\mathcal T}_\alpha(x):=\left(\ma^2+|\ya-\xa|^2\right)^{\frac{n-2}{2}}G(\ya, \varphi((0,\xa')+\ma x))U_0^{\crit-1}(x).$$
Here again, we have to distinguish two cases.

\medskip\noindent{\bf Case \ref{lemmaprojection}.4.1:} Assume that $\ya-\xa=O(\ma)$ when $\alpha\to +\infty$. We define $\theta_{\alpha}:=\ma^{-1}(\ya-\xa)$ for all $\alpha\in\nn$. Using Proposition \ref{app:prop:asymp:green}, we get as in Step \ref{lemmaprojection}.3.2.1 that for all $x\in B_R(0)\cap\overline{\rnm}\setminus\{\theta_\infty\}$,
\begin{equation*}
\lim_{\alpha\to +\infty}\ma^{n-2}G(\ya, \varphi((0,\xa')+\ma x))=k_n\left(|x-\theta_\infty|^{2-n}+|x-\pi^{-1}(\theta_\infty)|^{2-n}\right)
\end{equation*}
and this convergence holds uniformly with respect to $x$. Plugging this limit into \eqref{est:Va:4} yields
\begin{eqnarray*}
&&\Va(\ya)=\overline{\Va}\\
&&+\Ua(\ya)\left((1+|\theta_\infty|^2)^{\frac{n-2}{2}}\int_{B_R(0)\cap\rnm}k_n\left(|x-\theta_\infty|^{2-n}+|x-\pi^{-1}(\theta_\infty)|^{2-n}\right)\Delta U_0(x)\, dx\right)\\
&&+\left(\epsilon_R+o(1)\right)\Ua(\ya)
\end{eqnarray*}
when $\alpha\to +\infty$.  With a change of variable and using that $U_0$ is radially symmetrical, we get that
\begin{eqnarray*}
&&\int_{B_R(0)\cap\rnm}k_n\left(|x-\theta_\infty|^{2-n}+|x-\pi^{-1}(\theta_\infty)|^{2-n}\right)\Delta U_0(x)\, dx\\
&&=\int_{B_R(0)}k_n|x-\theta_\infty|^{2-n}\Delta U_0(x)\, dx
\end{eqnarray*}
for all $R>0$. The, arguing as in Step \ref{lemmaprojection}.3.2.2, we get that \eqref{concl:4} holds in Case \ref{lemmaprojection}.4.1.

\medskip\noindent{\bf Case \ref{lemmaprojection}.4.2:} Assume that $\lim_{\alpha\to +\infty}\ma^{-1}|\ya-\xa|=+\infty$. Using again Proposition \ref{app:prop:asymp:green} and arguing as in Step \ref{lemmaprojection}.3.2.1, we get that (we omit the details)
$$\lim_{\alpha\to +\infty}\left(\ma^2+|\ya-\xa|^2\right)^{\frac{n-2}{2}}G(\ya, \varphi((0,\xa')+\ma x))=2k_n$$
uniformly for all $x\in B_R(0)$. Plugging this limit into \eqref{est:Va:4} yields
\begin{eqnarray*}
\Va(\ya)&=&\overline{\Va}+\Ua(\ya)\left(\int_{B_R(0)\cap\rnm}2k_n\Delta U_0(x)\, dx+\epsilon_R+o(1)\right)\\
&=&\overline{\Va}+\Ua(\ya)\left(\int_{B_R(0)}k_n\Delta U_0(x)\, dx+\epsilon_R+o(1)\right)
\end{eqnarray*}
when $\alpha\to +\infty$.  We then get that \eqref{concl:4} holds in Case \ref{lemmaprojection}.4.2.

\medskip\noindent{\bf Step \ref{lemmaprojection}.5.} We are now in position to prove Proposition \ref{lemmaprojection}. We let $K>0$ be as in Step 1 and we let $\Va$ be the unique solution to \eqref{def:Va} such that
$$\overline{\Va}:=(K+1)\ma^{\frac{n-2}{2}}$$
for all $\alpha\in\nn$. Clearly points (i), (ii) and (iii) of Proposition \ref{lemmaprojection} hold. Moreover, we immediately get with the estimates above that $\lim_{\alpha\to +\infty}\frac{\Va(\ya)}{\Ua(\ya)}$ is a positive real number. This proves Proposition \ref{lemmaprojection}.\hfill$\Box$

\end{document}